\documentclass[12pt]{amsart}

\usepackage{amsmath,amsthm,amssymb}

\usepackage{a4wide}
\usepackage{enumerate}

\newtheorem{theorem}{Theorem}[section]
\newtheorem{lemma}[theorem]{Lemma}
\newtheorem{proposition}[theorem]{Proposition}
\newtheorem{corollary}[theorem]{Corollary}
\newtheorem*{MainTheorem}{Main Theorem}

\theoremstyle{remark}
\newtheorem{remark}[theorem]{Remark}

\newtheorem*{claim*}{Claim}

\newcommand{\C}{\ensuremath{\mathbb{C}}}
\newcommand{\R}{\ensuremath{\mathbb{R}}}

\newcommand{\g}[1]{\ensuremath{\mathfrak{#1}}}
\newcommand{\II}{\ensuremath{I\!I}}
\DeclareMathOperator{\tr}{tr}
\DeclareMathOperator{\sech}{sech}
\DeclareMathOperator{\csch}{csch}

\DeclareMathOperator{\Ad}{Ad}

\DeclareMathOperator{\Exp}{Exp}

\DeclareMathOperator{\spann}{span}

\newcommand{\Ss}{\ensuremath{\mathcal{S}}}

\begin{document}
\title[Isoparametric hypersurfaces in complex hyperbolic spaces]{Isoparametric hypersurfaces\\ in complex hyperbolic spaces}

\author[J.\ C.\ D\'{\i}az-Ramos]{Jos\'{e} Carlos D\'{\i}az-Ramos}
\author[M.\ Dom\'{\i}nguez-V\'{a}zquez]{Miguel Dom\'{\i}nguez-V\'{a}zquez}
\author[V.\ Sanmart\'{\i}n-L\'{o}pez]{V\'{\i}ctor Sanmart\'{\i}n-L\'{o}pez}

\address{Department of Geometry and Topology, Universidade de Santiago de Compostela, Spain}
\email{josecarlos.diaz@usc.es}
\address{ICMAT - Instituto de Ciencias Matem\'aticas (CSIC-UAM-UC3M-UCM), Madrid, Spain.}
\email{miguel.dominguez@icmat.es}
\address{Department of Geometry and Topology, Universidade de Santiago de Compostela, Spain}
\email{victor.sanmartin@usc.es}

\thanks{The authors have been supported by projects MTM2016-75897-P (AEI/FEDER, UE), EM2014/009, GRC2013-045 and MTM2013-41335-P with FEDER funds (Spain). The second author has received funding from IMPA (Brazil), from the ICMAT Severo Ochoa project SEV-2015-0554 (MINECO, Spain), and from the European Union's Horizon 2020 research and innovation programme under the Marie Sk{\l}odowska-Curie grant agreement No.~745722. The third author has been supported by an FPU fellowship and by Fundaci\'{o}n Barri\'{e} de la Maza (Spain)}

\begin{abstract}
We classify isoparametric hypersurfaces in complex hyperbolic spaces.
\end{abstract}




\maketitle


\vspace{-2ex}
\enlargethispage{\baselineskip}

\section{Introduction and main result}\label{sect:Intro}

An isoparametric hypersurface of a Riemannian manifold is a hypersurface such that all its sufficiently close parallel hypersurfaces have constant mean curvature.
In this paper, we prove the following classification result (see below for the explanation of the examples):

\begin{MainTheorem}
Let $M$ be a connected real hypersurface in the complex hyperbolic space $\C H^n$, $n\geq 2$. Then, $M$ is isoparametric if and only if $M$ is congruent to an open part of:
\begin{enumerate}[{\rm (i)}]
\item a tube around a totally geodesic complex hyperbolic space $\C H^k$, $k\in\{0,\dots,n-1\}$, or\label{th:main:chk}

\item a tube around a totally geodesic real hyperbolic space $\R H^n$, or\label{th:main:rhn}

\item a horosphere, or\label{th:main:horosphere}

\item a ruled homogeneous minimal Lohnherr hypersurface $W^{2n-1}$, or some of its equidistant hypersurfaces, or\label{th:main:w}

\item a tube around a ruled homogeneous minimal Berndt-Br\"{u}ck submanifold $W^{2n-k}_\varphi$, for $k\in\{2,\dots,n-1\}$, $\varphi\in(0,\pi/2]$, where $k$ is even if $\varphi\neq\pi/2$, or\label{th:main:wphi}

\item a tube around a ruled homogeneous minimal submanifold $W_{\g{w}}$, for some proper real subspace $\g{w}$ of $\g{g}_\alpha\cong\C^{n-1}$ such that $\g{w}^\perp$, the orthogonal complement of $\g{w}$ in $\g{g}_\alpha$, has nonconstant K\"{a}hler angle.\label{th:main:ww}
\end{enumerate}
\end{MainTheorem}

We say that $M$ is an (extrinsically) homogeneous submanifold of a Riemannian manifold~$\bar{M}$ if for each pair of points $p$, $q\in M$, there exists an isometry $g\colon\bar{M}\to\bar{M}$ such that $g(M)=M$ and $g(p)=q$. Cases (\ref{th:main:chk}), (\ref{th:main:rhn}), and (\ref{th:main:horosphere}) are the standard examples of homogeneous Hopf hypersurfaces in $\C H^n$, also known as the examples of Montiel's list~\cite{Mo85}. We briefly explain the examples in~(\ref{th:main:w}), (\ref{th:main:wphi}), and~(\ref{th:main:ww}) of the Main Theorem. For more details we refer to Subsection~\ref{sec:examples:W}.

Let $\g{g}=\g{su}(1,n)$ be the Lie algebra of the isometry group of $\C H^n$, $n\geq 2$, $K$ the isotropy group at $o\in\C H^n$, and $\g{k}=\g{u}(n)$ its Lie algebra, which is a maximal compact subalgebra of~$\g{g}$. Let $\g{g}=\g{k}\oplus\g{p}$ be the Cartan decomposition of $\g{g}$ with respect to $o\in\C H^n$. We consider $\g{g}=\g{g}_{-2\alpha}\oplus\g{g}_{-\alpha}\oplus\g{g}_0
\oplus\g{g}_\alpha\oplus\g{g}_{2\alpha}$ the root space decomposition of $\g{g}$ with respect to a maximal abelian subspace $\g{a}$ of $\g{p}$. It turns out that $\g{a}$ and $\g{g}_{2\alpha}$ are $1$-dimensional, and that $\g{g}_\alpha$ is an $(n-1)$-dimensional complex vector space with respect to a certain complex structure $J$ induced by $\C H^n$.

Let $\g{w}$ be a real subspace of $\g{g}_\alpha$, that is, a subspace of $\g{g}_\alpha$ with the underlying structure of real vector space. We define the Lie subalgebra $\g{s}_\g{w}$ of $\g{g}$ by $\g{s}_{\g{w}}=\g{a}\oplus\g{w}\oplus\g{g}_{2\alpha}$, and denote by $S_{\g{w}}$ the connected closed subgroup of $SU(1,n)$ whose Lie algebra is $\g{s}_{\g{w}}$. Then, we define $W_{\g{w}}$ as the orbit through $o$ of the subgroup $S_{\g{w}}$. It was shown in~\cite{DD12} that $W_{\g{w}}$ is a homogeneous minimal submanifold of $\C H^n$, and that the tubes around it are isoparametric hypersurfaces of $\C H^n$. We denote by $\g{w}^\perp$ the orthogonal complement of $\g{w}$ in $\g{g}_{\alpha}$.

If $\g{w}$ is a hyperplane of $\g{g}_\alpha$, then $W_{\g{w}}$ is a real hypersurface of $\C H^n$ denoted by $W^{2n-1}$, and it was shown in~\cite{Be98} that the equidistant hypersurfaces to $W^{2n-1}$ are homogeneous. If $\g{w}^\perp$ has constant K\"{a}hler angle, that is, for each nonzero $\xi\in\g{w}^\perp$ the angle $\varphi$ between $J\xi$ and $\g{w}^\perp$ is independent of $\xi$, then the corresponding $W_{\g{w}}$ is denoted by $W^{2n-k}_\varphi$. Here $k$ is the codimension of $\g{w}$ in $\g{g}_\alpha$, and it can be proved~\cite{BB01} that $k$ is even if $\varphi\neq \pi/2$. Moreover, it follows from~\cite{BB01} that the tubes around $W_\varphi^{2n-k}$ are homogeneous. If $\varphi=0$, the submanifold~$W_0^{2n-k}$ is a totally geodesic complex hyperbolic space and we recover the examples in~(\ref{th:main:chk}).

If $\g{w}^\perp$ does not have constant K\"{a}hler angle, then the tubes around $W_{\g{w}}$ are not homogeneous (indeed, they have nonconstant principal curvatures) but are still isoparametric~\cite{DD12}. Taking into account that the examples (\ref{th:main:chk}), (\ref{th:main:rhn}) and (\ref{th:main:horosphere}) are known to be homogeneous, and the fact that homogeneous hypersurfaces are always isoparametric, a consequence of our result is the classification of homogeneous hypersurfaces in complex hyperbolic spaces:

\begin{corollary}\cite{BT07}
A real hypersurface of $\C H^n$, $n\geq 2$, is homogeneous if and only if it is congruent to one of the examples~(\ref{th:main:chk}) through~(\ref{th:main:wphi}) in the Main Theorem.
\end{corollary}

For $n=2$, $\g{g}_\alpha$ is a complex line and thus the examples~(\ref{th:main:wphi}) and~(\ref{th:main:ww}) are not possible. Compare also with the classification of real hypersurfaces in $\C H^2$ with constant principal curvatures~\cite{BD07}.

\begin{corollary}
An isoparametric hypersurface in $\C H^2$ is an open part of a homogeneous hypersurface.
\end{corollary}

Nevertheless, for $n\geq 3$ there are inhomogeneous examples: one family up to congruence for $\C H^3$, and infinitely many for $\C H^n$, $n\geq 4$.

Since the examples in~(\ref{th:main:ww}) of the Main Theorem are the only ones that do not have constant principal curvatures we also get:

\begin{corollary}\label{th:cpc}
An isoparametric hypersurface of $\C H^n$ has constant principal curvatures if and only if it is an open part of a homogeneous hypersurface of $\C H^n$.
\end{corollary}

Another important consequence of our classification is that each isoparametric hypersurface of $\C H^n$ is an open part of a complete, topologically closed, isoparametric hypersurface which, in turn, is a regular leaf of a singular Riemannian foliation on $\C H^n$ whose leaves of maximal dimension are all isoparametric. Corollary~\ref{th:cpc} emphasizes the fact that a hypersurface with constant principal curvatures cannot be a leaf of a singular Riemannian foliation by hypersurfaces with constant principal curvatures unless it is homogeneous. An isoparametric hypersurface in $\C H^n$ determines an isoparametric family of hypersurfaces that fills the whole ambient space and that admits at most one singular leaf. According to our classification, this singular leaf, if it exists, satisfies

\begin{corollary}
The focal submanifold of an isoparametric hypersurface in $\C H^n$ is locally homogeneous.
\end{corollary}

We can determine the congruence classes of isoparametric families of hypersurfaces in~$\C H^n$. Note that, apart from the horosphere foliation  $\mathcal{F}_H$, the family $\mathcal{F}_{\R H^n}$ of tubes around a totally geodesic $\R H^n$, and the family $\mathcal{F}_{o}$ of geodesic spheres around any point $o\in\C H^n$, any other family is given by the collection of tubes around a submanifold $W_\g{w}$, where $\g{w}$ is any real subspace of codimension at least one in $\g{g}_\alpha$. Thus, we have

\begin{theorem}\label{th:congruence}
The moduli space of congruence classes of isoparametric families of hypersurfaces of $\C H^n$ is isomorphic to the disjoint union
\[
\{\mathcal{F}_H,\mathcal{F}_{\R H^n},\mathcal{F}_{o}\}\amalg\biggl(
\coprod_{k=0}^{2n-3}G_k(\R^{2n-2})/U(n-1)\biggr),
\]
where $G_k(\R^{2n-2})/U(n-1)$ stands for the orbit space of the standard action of the unitary group $U(n-1)$ on the Grassmannian of real vector subspaces of dimension $k$ of~$\C^{n-1}$.
\end{theorem}

The study of isoparametric hypersurfaces traces back to the work of Somigliana~\cite{So19}, who studied isoparametric surfaces of the $3$-dimensional Euclidean space in relation to a problem of Geometric Optics. This study was generalized by Segre~\cite{Se38}, who classified isoparametric hypersurfaces in any Euclidean space. It follows from this result that isoparametric hypersurfaces in a Euclidean space $\R^n$ are open parts of affine hyperplanes $\R^{n-1}$, spheres $S^{n-1}$, or generalized cylinders $S^k\times\R^{n-k-1}$, $k\in\{1,\dots,n-2\}$, all of which are homogeneous.

Cartan became interested in this problem and studied it in real space forms. He obtained a fundamental formula relating the principal curvatures and their multiplicities, and derived a classification in real hyperbolic spaces~\cite{Ca38}. In this case, an isoparametric hypersurface of $\R H^n$ is congruent to an open part of a geodesic sphere, a tube around a totally geodesic real hyperbolic space $\R H^k$, $k\in\{1,\dots,n-2\}$, a totally geodesic $\R H^{n-1}$ or one of its equidistant hypersurfaces, or a horosphere. All these examples are homogeneous.

Cartan also made progress in spheres~\cite{Ca39mz}, and succeeded in classifying isoparametric hypersurfaces with one, two or three distinct principal curvatures. However, it turns out that the classification of isoparametric hypersurfaces in spheres is very involved. In fact, its complete classification remains one of the most outstanding problems in Differential Geometry nowadays~\cite{Ya93}. It was a surprise at that moment to find inhomogeneous examples. The most complete list of such examples is due to Ferus, Karcher and M\"{u}nzner~\cite{FKM81}. As of this writing, the classification problem remains still open, although some important progress has been made by Stolz~\cite{St99}, Cecil, Chi and Jensen~\cite{CCJ07}, Immervoll~\cite{Im08} and Chi~\cite{Ch13} for four distinct principal curvatures, and by Dorfmeister and Neher~\cite{DoNe85}, Miyaoka~\cite{Mi13} and Siffert~\cite{Si} for six distinct principal curvatures. See the surveys~\cite{T00} and~\cite{Ce08} for a more detailed story of the problem and related topics.

In real space forms, a hypersurface is isoparametric if and only if it has constant principal curvatures. This is not true in a general Riemannian manifold. Thus, it makes sense to study both isoparametric hypersurfaces or hypersurfaces with constant principal curvatures in complex space forms. The classification of real hypersurfaces with constant principal curvatures in complex projective spaces is known for Hopf hypersurfaces~\cite{Ki86}, and for two or three distinct principal curvatures~\cite{Ta75a},~\cite{Ta75b}; all known examples are open parts of homogeneous hypersurfaces. Using the classification results in spheres, the second author~\cite{Do} derived the classification of isoparametric hypersurfaces in $\C P^n$, $n\neq 15$. A consequence of this classification is that inhomogeneous isoparametric hypersurfaces in $\C P^n$ are relatively common. 

Real hypersurfaces with constant principal curvatures in $\C H^n$ have been classified under the assumption that the hypersurface is Hopf~\cite{Be89}, or if the number of distinct constant principal curvatures is two~\cite{Mo85} or three~\cite{BD06},~\cite{BD07}. All of these examples are again homogeneous.

In this paper we deal with isoparametric hypersurfaces in complex hyperbolic spaces. Apart from the homogenous examples classified by Berndt and Tamaru in~\cite{BT07}, there are also some inhomogeneous examples that were built by the first and second authors in~\cite{DD12}. In the present paper we show that isoparametric hypersurfaces in complex hyperbolic spaces are open parts of the known homogeneous or inhomogeneous examples. To our knowledge, this is the first complete classification in a whole family of Riemannian manifolds since Cartan's classification of isoparametric hypersurfaces in real hyperbolic spaces~\cite{Ca38}.

As we will see in Section~\ref{sec:isoparametric}, the classification of isoparametric hypersurfaces in the complex hyperbolic space $\C H^n$ is intimately related to the study of Lorentzian isoparametric hypersurfaces in the anti-De Sitter space $H_1^{2n+1}$. Following the ideas of Magid in~\cite{M85}, Xiao gave parametrizations of Lorentzian isoparametric hypersurfaces in $H_1^{2n+1}$~\cite{X99}. Burth~\cite{Bu93} pointed out some crucial gaps in Magid's arguments, which Xiao's proof depends on. Furthermore, the classification of isoparametric hypersurfaces in $\C H^n$ does not follow right away from an eventual  classification of Lorentzian isoparametric hypersurfaces in the anti-De Sitter space~$H_1^{2n+1}$, as the projection via the Hopf map $\pi\colon H_1^{2n+1}\to \C H^n$ depends in a very essential way on the complex structure of the semi-Euclidean space $\R^{2n+2}$ where the anti-De Sitter space lies. This is precisely the main difficulty of this approach in the classification of isoparametric submanifolds of complex projective spaces~\cite{Do} using the Hopf map from an odd-dimensional sphere.

Therefore, although the starting point of our arguments is the fact that isoparametric hypersurfaces in $\C H^n$ lift to Lorentzian isoparametric hypersurfaces in $H_1^{2n+1}$, our approach is independent of \cite{M85} and \cite{X99}. The shape operator of a Lorentzian isoparametric hypersurface does not need to be diagonalizable and, indeed, it can adopt four distinct Jordan canonical forms. Using the Lorentzian version of Cartan's fundamental formula, some algebraic arguments, and Gauss and Codazzi equations, we determine the hypersurfaces in $\C H^n$ that lift to Lorentzian hypersurfaces of three of the four types. The remaining case is much more involved. Working in the anti-De Sitter space, we start using Jacobi field theory in order to extract information about the shape operator of the focal submanifold (Proposition~\ref{th:wr}). The key step is to justify the existence of a common eigenvector to all shape operators of the focal submanifold (Proposition~\ref{th:z}). This allows us to define a smooth vector field which is crucial to show that the second fundamental form of the focal set coincides with that of one of the submanifolds $W_{\g{w}}$. After a study of the normal bundle of this focal set, we obtain a reduction of codimension result. Together with a more geometric construction of the submanifolds $W_\g{w}$ (Proposition~\ref{th:construction}), we prove a rigidity result for these submanifolds (Theorem~\ref{th:rigidity}); although the proof of this result is convoluted, it reveals several interesting aspects of the geometry of the ruled minimal submanifolds $W_\g{w}$ in relation to the geometry of the ambient complex hyperbolic space. Altogether, this will allow us to conclude the proof of the Main Theorem.

The paper is organized as follows. In Section~\ref{sec:preliminaries} we introduce the main ingredients to be used in this paper. We start with a quick review of submanifold geometry of semi-Riemannian manifolds (Subsection~\ref{sec:submanifold}), then we describe the complex hyperbolic space and its relation to the anti-De Sitter space (Subsection~\ref{sec:Hopf_map}), the structure of a real subspace of a complex vector space (Subsection~\ref{sec:Kahler_angles}), and the examples of isoparametric hypersurfaces in $\C H^n$ given in the Main Theorem (Subsection~\ref{sec:examples}). Section~\ref{sec:isoparametric} is devoted to presenting Cartan's fundamental formula for Lorentzian space forms and some of its algebraic consequences. It turns out that cases~(\ref{th:main:rhn}) and~(\ref{th:main:horosphere}) can be handled at this point. For the remaining cases, a more thorough study of the focal set is needed, and this is carried out in Section~\ref{sec:type_III}. The ingredients utilized here are the Gauss and Codazzi equations of a hypersurface (Subsection~\ref{sec:codazzi}), Jacobi field theory (Subsection~\ref{sec:jacobi}), and a detailed study of the geometry of the focal submanifold (Subsection~\ref{sec:focal}). In Section~\ref{sec:rigidity} we give a characterization of the submanifolds $W_{\g{w}}$ in terms of their second fundamental form. We need a reduction of codimension argument in Subsection~\ref{sec:normal}, and the proof is concluded in Subsection~\ref{sec:rigidity:proof}. We finish the proofs of the Main Theorem and Theorem~\ref{th:congruence} in Section~\ref{sec:proof}.

\section{Preliminaries}\label{sec:preliminaries}

Let $M$ be a semi-Riemannian manifold. It is assumed that all manifolds in this paper are smooth. We denote by $\langle\,\cdot\,,\,\cdot\rangle$ the semi-Riemannian metric of $M$, and by $R$ its curvature tensor, which is defined by the convention $R(X,Y)=[\nabla_X,\nabla_Y]-\nabla_{[X,Y]}$.

If $p\in M$, $T_p M$ denotes the tangent space at $p$, $TM$ is the tangent bundle of $M$, and $\Gamma(TM)$ is the module of smooth vector fields on $M$. In general, if $\mathcal{D}$ is a distribution along~$M$, we denote by $\Gamma(\mathcal{D})$ the module of sections of $\mathcal{D}$, that is, the vector fields $X\in\Gamma(TM)$ such that $X_p\in\mathcal{D}_p$ for each $p\in M$.

Let $V$ be a vector space with nondegenerate symmetric bilinear form $\langle\cdot,\cdot\rangle$. Recall that $v\in V$ is spacelike, timelike, or null if $\langle v,v\rangle$ is positive, negative, or zero, respectively. We also write $\lVert v\rVert=\sqrt{\lvert \langle v,v\rangle\rvert}$ for $v\in V$. Moreover, if $U$ and $W$ are subspaces of~$V$, we denote $U\ominus W=\{u\in U:\langle u,w\rangle=0,\forall w\in W\}$. We do not require $W\subset U$. This is convenient when dealing with nondefinite scalar products, especially if there are null vectors in $W$. If $\langle\cdot,\cdot\rangle$ is positive definite, this notation stands for the orthogonal complement of $W$ in $U$.

\subsection{Geometry of submanifolds}\label{sec:submanifold}\hfill

Let $(\bar{M},\langle\cdot,\cdot\rangle)$ be a semi-Riemannian manifold and $M$ an embedded submanifold of $\bar{M}$ such that the restriction of $\langle\cdot,\cdot\rangle$ to $M$ is nondegenerate (this is automatically true if $\bar{M}$ is Riemannian). The {normal bundle} of $M$ is denoted by~$\nu M$. Thus, $\Gamma(\nu M)$ denotes the module of all normal vector fields to $M$. A canonical orthogonal decomposition holds at each point $p\in M$, namely, $T_p\bar{M}=T_pM\oplus \nu_p M$. In this work, the symbol $\oplus$ will always denote direct sum (not necessarily orthogonal direct sum).

Let us denote by $\bar{\nabla}$ and $\bar{R}$ the Levi-Civita connection and the curvature tensor of $\bar{M}$, respectively, and by $\nabla$ and $R$ the corresponding objects for $M$. The {second fundamental form}~$\II$
of $M$ is defined by the {Gauss formula}
\[
\bar{\nabla}_XY=\nabla_XY+\II(X,Y)
\]
for any $X$, $Y\in\Gamma(TM)$. Let $\xi\in\Gamma(\nu M)$ be a normal vector field. The {shape operator} $\Ss_\xi$ of $M$ with respect to $\xi$ is the self-adjoint operator on $M$ defined by $\langle \Ss_\xi X,Y\rangle=\langle\II(X,Y),\xi\rangle$, where $X$, $Y\in\Gamma(TM)$. Moreover, denote
by $\nabla^\perp$ the normal connection of~$M$. Then we have the {Weingarten formula}
\[
\bar{\nabla}_X\xi=-\Ss_\xi X+\nabla_X^\perp\xi.
\]
The extrinsic geometry of $M$ is controlled by Gauss, Codazzi and Ricci equations
\begin{align*}
\langle\bar{R}(X,Y)Z,W\rangle
&{}=\langle R(X,Y)Z,W\rangle-\langle\II(Y,Z),\II(X,W)\rangle + \langle\II(X,Z),\II(Y,W)\rangle,\\
\langle\bar{R}(X,Y)Z,\xi\rangle
&{}=\langle(\nabla_X^\perp\II)(Y,Z)-(\nabla_Y^\perp\II)(X,Z),\xi\rangle,\\
\langle R^\perp(X,Y)\xi,\eta\rangle
&{}= \langle\bar{R}(X,Y)\xi,\eta\rangle + \langle[\Ss_\xi,\Ss_\eta]X,Y\rangle,
\end{align*}
where $X$, $Y$, $Z$, $W\in\Gamma(TM)$, $\xi$, $\eta\in\Gamma(\nu M)$, $(\nabla_X^\perp\II)(Y,Z)=\nabla_X^\perp\II(Y,Z)
-\II(\nabla_XY,Z)-\II(Y,\nabla_XZ)$, and $R^\perp$ is the curvature tensor of the normal bundle of $M$, which is defined by $R^\perp(X,Y)\xi=[\nabla_X^\perp,\nabla_Y^\perp]\xi - \nabla_{[X,Y]}^\perp\xi$.

Assume now that $M$ is a hypersurface of $\bar{M}$, that is, a submanifold of codimension one. Locally and up to sign, there is a unique unit normal vector field $\xi\in \Gamma(\nu M)$. We assume here and henceforth that $\xi$ is spacelike, that is, $\langle\xi,\xi\rangle=1$. In this case we write $\Ss=\Ss_\xi$, the shape operator with respect to $\xi$. The Gauss and Weingarten
formulas now read
\begin{align*}
\bar{\nabla}_X Y&{}=\nabla_X Y+\langle \Ss X,Y\rangle\xi,
&\bar{\nabla}_X\xi&{}=-\Ss X.
\end{align*}
Then, the Gauss and Codazzi equations reduce to
\begin{align*}
\langle\bar{R}(X,Y)Z,W\rangle&{}=\langle R(X,Y)Z,W\rangle-\langle \Ss Y,Z\rangle\langle \Ss X,W\rangle
+\langle\Ss X,Z\rangle\langle \Ss Y,W\rangle,\\
\langle\bar{R}(X,Y)Z,\xi\rangle &{}= \langle(\nabla_X\Ss)Y-(\nabla_Y \Ss) X,Z\rangle,
\end{align*}
whereas the Ricci equation does not give further information for hypersurfaces.

The {mean curvature} of a hypersurface $M$ is $h=\tr\Ss$, the trace of its shape operator. For $r\in\R$ we define the map $\Phi^r\colon M\to\bar{M}$ by $\Phi^r(p)=\exp_p(r\xi_p)$, where $\exp$ is the Riemannian exponential map of $\bar{M}$. For a fixed $r$, $\Phi^r(M)$ is not necessarily a submanifold of $\bar{M}$, but at least locally and for $r$ small enough, it is a hypersurface of $\bar{M}$. A parallel hypersurface at a distance $r$ to a given hypersurface $M$ is precisely a hypersurface of the form $\Phi^r(M)$. A hypersurface of $\bar{M}$ is said to be \emph{isoparametric} if it and all its sufficiently close (locally defined) parallel hypersurfaces have constant mean curvature.

We say that $\lambda$ is a principal curvature of a hypersurface $M$ if there exists a nonzero vector field $X\in\Gamma(TM)$ such that $\Ss X=\lambda X$. The vector $X_p$ is then called a principal curvature vector at $p\in M$. By $T_\lambda(p)$ we denote the eigenspace of $\lambda(p)$ at $p$, and we call it the principal curvature space of $\lambda(p)$. Under certain assumptions, $T_\lambda$ defines a smooth distribution along~$M$.
If $M$ is Riemannian, then $\Ss$ is known to be diagonalizable. However, if $M$ is not Riemannian, this is not necessarily true, and the Jordan canonical form of $\Ss$ might have a nondiagonal structure. In such situations it is important to distinguish between the \emph{geometric multiplicity} of a principal curvature $\lambda$, that is, $\dim\ker(\Ss-\lambda)$, and its \emph{algebraic multiplicity} $m_\lambda$, that is, the multiplicity of $\lambda$ as a zero of the characteristic polynomial of $\Ss$. Obviously, the geometric multiplicity is always less or equal than the algebraic multiplicity. In the Riemannian setting both quantities are the same and we simply talk about the multiplicity of $\lambda$. In any case, the number of distinct principal curvatures at $p$ is denoted by $g(p)$. In principle, $g$ does not need to be a constant function.

\subsection{Complex hyperbolic spaces and the Hopf map}\label{sec:Hopf_map}\hfill

We briefly recall the construction of the complex hyperbolic space. In $\C^{n+1}$ we define the flat semi-Riemannian metric given by the formula $\langle z,w\rangle=\textup{Re}\bigl(-z_0\bar{w}_0+\sum_{k=1}^n z_k\bar{w}_k\bigr)$. We consider the anti-De Sitter spacetime of radius $r>0$ as the hypersurface $H_1^{2n+1}(r)=\{z\in\C^{n+1}:\langle z,z\rangle=-r^2\}$. The anti-De Sitter space is a Lorentzian space form of constant negative curvature $c=-4/r^2$. An $S^1$-action can be defined on $H^{2n+1}_1(r)$ by means of $z\mapsto \lambda z$, with $\lambda\in\C$, $\lvert\lambda\rvert=1$. Then, the complex hyperbolic space $\C H^n(c)$ is by definition $H_1^{2n+1}(r)/S^1$, the quotient of the anti-De Sitter spacetime by this $S^1$-action. It is known that $\C H^n(c)$ is a connected, simply connected, K\"{a}hler manifold of complex dimension $n$ and constant holomorphic sectional curvature $c<0$. In what follows we will omit $r$ and $c$ and simply write $H_1^{2n+1}$ and $\C H^n$. If $n=1$, $\C H^1$ is isometric to a real hyperbolic space $\R H^2$ of constant sectional curvature $c$. Thus, throughout this paper we assume $n\geq 2$. If we denote by $J$ the complex structure of $\C H^n$, the curvature tensor $\bar{R}$ of $\C H^n$ reads
\[
\bar{R}(X,Y)Z
=\frac{c}{4}\Bigl(\langle Y,Z\rangle X-\langle X,Z\rangle Y
+\langle JY,Z\rangle JX-\langle JX,Z\rangle JY
-2\langle JX,Y\rangle JZ\Bigr).
\]

One can define a vector field $V$ on $H^{2n+1}_1$ by means of $V_q=i\sqrt{-c}\,q/2$ for each $q\in H^{2n+1}_1$. This vector field is tangent to the $S^1$-flow and $\langle V,V\rangle=-1$. The quotient map $\pi\colon H^{2n+1}_1\to\C H^n$ is called the Hopf map and it is a semi-Riemann\-ian submersion with timelike totally geodesic fibers, whose tangent spaces are generated by the vertical vector field $V$. We have the linear isometry
$T_q H^{2n+1}_1\cong T_{\pi(q)}\C H^n \oplus \R V$,
and the following relations between the Levi-Civita connections $\tilde{\nabla}$ and $\bar{\nabla}$ of $H^{2n+1}_1$ and $\C H^n$, respectively:
\begin{align}
\label{sub1}\tilde{\nabla}_{X^L}Y^L ={}&(\bar{\nabla}_X Y)^L+\frac{\sqrt{-c}}{2}\langle J X^L,Y^L\rangle V,\\
\label{sub2}\tilde{\nabla}_{V}X^L ={}& \tilde{\nabla}_{X^L}V =\frac{\sqrt{-c}}{2}(JX)^L=\frac{\sqrt{-c}}{2}JX^L,
\end{align}
for all $X$, $Y\in\Gamma(T\C H^n)$, and where $X^L$ denotes the horizontal lift of $X$ and $J$ denotes the complex structure on $\C^{n+1}$ as well. These formulas follow from the fundamental equations of semi-Riemannian submersions \cite{ON66}.

Let now $M$ be a real hypersurface in $\C H^n$. Sometimes we say `real' to emphasize that $M$ has real codimension one, as opposed to `complex' codimension one. Then $\tilde{M}=\pi^{-1}(M)$ is a hypersurface in $H^{2n+1}_1$ which is invariant under the $S^1$-action. Thus $\pi|_{\tilde{M}}\colon\tilde{M}\to M$ is a semi-Riemannian submersion with timelike totally geodesic $S^1$-fibers. Conversely, if $\tilde{M}$ is a Lorentzian hypersurface in $H^{2n+1}_1$ which is invariant under the $S^1$-action, then $M=\pi(\tilde{M})$ is a real hypersurface in $\C H^n$, and $\pi|_{\tilde{M}}\colon\tilde{M}\to M$ is a semi-Riemannian submersion with timelike totally geodesic fibers. If $\xi$ is a (local) unit normal vector field to $M$, then $\xi^L$ is a (local) spacelike unit normal vector field to $\tilde{M}$. In order to simplify the notation, we will denote by $\nabla$ the Levi-Civita connections of $M$ and of $\tilde{M}$. Denote by $\Ss$ and $\tilde{\Ss}$ the shape operators of $M$ and $\tilde{M}$, respectively.

The Gauss and Weingarten formulas for the hypersurface $\tilde{M}$ in $H^{2n+1}_1$ are, as we have seen,
$\tilde{\nabla}_X Y=\nabla_X Y +\langle \tilde{\Ss} X, Y\rangle\xi^L$, and $\tilde{\nabla}_X \xi^L=-\tilde{\Ss} X$. Using \eqref{sub1} and \eqref{sub2}, for any $X\in \Gamma(TM)$, we have
\begin{equation}\label{eq:shape}
\tilde{\Ss}X^L =(\Ss X)^L+\frac{\sqrt{-c}}{2}\langle J\xi^L,X^L\rangle V,\qquad
\tilde{\Ss}V=-\frac{\sqrt{-c}}{2}J\xi^L.
\end{equation}
In particular, $\Ss X=\pi_*\tilde{\Ss}X^L$.

Let $X_1,\ldots, X_{2n-1}$ be a local frame on $M$ consisting of principal directions with corresponding principal curvatures $\lambda_1,\ldots, \lambda_{2n-1}$ (obviously, some can be repeated). Then $X_1^L, \ldots, X_{2n-1}^L, V$ is a local frame on $\tilde{M}$ with respect to which $\tilde{\Ss}$ is represented by the matrix
\begin{equation}\label{lift_form}
\begin{pmatrix}
\lambda_1 & & 0& -\frac{b_1\sqrt{-c}}{2}\\
& \ddots& &\vdots\\
0& & \lambda_{2n-1} & -\frac{b_{2n-1}\sqrt{-c}}{2}\\
\frac{b_1\sqrt{-c}}{2} &\cdots &\frac{b_{2n-1}\sqrt{-c}}{2} &0
\end{pmatrix},
\end{equation}
where $b_i=\langle J\xi, X_i\rangle$, $i=1,\ldots, 2n-1$, are $S^1$-invariant functions on (an open set of) $\tilde{M}$.

As a consequence of~\eqref{eq:shape}, $M$ and $\tilde{M}$ have the same mean curvatures. Since horizontal geodesics in $H^{2n+1}_1$ are mapped via $\pi$ to geodesics in $\C H^n$, it follows that $\pi$ maps equidistant hypersurfaces to $\tilde{M}$ to equidistant hypersurfaces to~$M$. Therefore, $M$ is isoparametric if and only if $\tilde{M}$ is isoparametric. This allows us to study isoparametric hypersurfaces in $\C H^n$ by analyzing which Lorentzian isoparametric hypersurfaces in $H^{2n+1}_1$ can result by lifting isoparametric hypersurfaces in $\C H^n$ to the anti-De Sitter space. It is instructive to note that, whereas the isoparametric condition behaves well with respect to the Hopf map, this is not so for the constancy of the principal curvatures of a hypersurface, since the functions $b_i$ might be nonconstant.

The tangent vector field $J\xi$ is called the Reeb or Hopf vector field of $M$. A real hypersurface $M$ in a complex hyperbolic space $\C H^n$ is \emph{Hopf} at a point $p\in M$ if $J\xi_p$ is a principal curvature vector of the shape operator. We say that $M$ is \emph{Hopf} if it is Hopf at all points.

\subsection{Real subspaces of a complex vector space}\label{sec:Kahler_angles}\hfill

In this subsection we compile some information on the structure of a real subspace of a complex vector space $V$. This will be needed to present the examples of isoparametric hypersurfaces introduced in the Main Theorem, and it will also be an important tool in the proof of this classification result. We follow~\cite{DDK}.

Let $W$ be a real subspace of $V$, that is, a subspace of $V$ with the underlying structure of real vector space (as opposed to a complex subspace of $V$). We denote by $J$ the complex structure of $V$, and assume that $V$, as a real vector space, carries an inner product $\langle\,\cdot\,,\,\cdot\,\rangle$ for which $J$ is an isometry.

Let $\xi\in W$ be a nonzero vector. The \emph{K\"{a}hler angle} of $\xi$ with respect to $W$ is the angle $\varphi_\xi\in[0,\pi/2]$ between $J\xi$ and $W$. For each $\xi\in W$, we write $J\xi=F\xi+P\xi$, where $F\xi$ is the orthogonal projection of $J\xi$ onto $W$, and
$P\xi$ is the orthogonal projection of $J\xi$ onto $V\ominus W$, the orthogonal complement of $W$ in $V$. Then, the K\"{a}hler angle of $W$ with respect to $\xi$ is determined by
$\langle F\xi,F\xi\rangle =\cos^2(\varphi_\xi)\langle
\xi,\xi\rangle$. Hence,
if $\xi$ has unit length, $\varphi_\xi$ is determined by the fact
that $\cos(\varphi_\xi)$ is the length of the orthogonal projection of
$J\xi$ onto $W$. Furthermore, it readily follows from $J^2=-I$ that $\langle
P\xi,P\xi\rangle =\sin^2(\varphi_\xi)\langle \xi,\xi\rangle$.

A subspace $W$ of a complex vector space is said to have \emph{constant K\"{a}hler angle} $\varphi\in[0,\pi/2]$ if all nonzero vectors of $W$ have the same K\"{a}hler angle~$\varphi$. In particular, a totally real subspace is a subspace with constant K\"{a}hler angle~$\pi/2$, and a subspace is complex if and only if it has constant K\"{a}hler angle~$0$. It is also known that a subspace $W$ with constant K\"ahler angle has even dimension unless $\varphi=\pi/2$.

Following the ideas in~\cite[Theorem~2.6]{DDK}, we consider the skew-adjoint linear map $F\colon W\to W$, that is $\langle F\xi,\eta\rangle=-\langle\xi,F\eta\rangle$ for any $\xi$, $\eta\in W$, and the symmetric bilinear form $(\xi,\eta)\mapsto \langle F\xi,F\eta\rangle$. Hence, it follows that there is an orthonormal basis $\{\xi_1,\dots,\xi_k\}$ of $W$ and K\"{a}hler angles $\varphi_1,\dots,\varphi_k$ such that $\langle F\xi_i,F\xi_j\rangle=\cos^2(\varphi_i)\delta_{ij}$, for all $i$, $j\in\{1,\dots,k\}$, and where $\delta_{ij}$ is the Kronecker delta. We call $\varphi_1,\dots,\varphi_k$ the \emph{principal K\"{a}hler angles} of $W$, and $\xi_1,\dots,\xi_k$ are called \emph{principal K\"{a}hler vectors}. Moreover, as it is proved in~\cite[Section~2.3]{DDK}, the subspace $W$ can be written as
$W=\oplus_{\varphi\in\Phi}W_\varphi$,
where $\Phi\subset[0,\pi/2]$ is a finite subset, $W_\varphi\neq 0$ for each $\varphi\in\Phi$, and each $W_\varphi$ has constant K\"{a}hler angle $\varphi$. Furthermore, if $\varphi$, $\psi\in\Phi$ and $\varphi\neq\psi$, then $W_\varphi$ and $W_\psi$ are complex-orthogonal, i.e.\ $\C W_\varphi \perp \C W_\psi$. The elements of $\Phi$ are precisely the principal K\"{a}hler angles, the subspaces $W_\varphi$ are called the \emph{principal K\"{a}hler subspaces}, and their dimension is called their multiplicity.

Denote by $W^\perp=V\ominus W$ the orthogonal complement of $W$ in $V$. Then, we can also take the decomposition of $W^\perp$ in subspaces of constant K\"{a}hler angle $W^\perp=\oplus_{\varphi\in\Psi}W_\varphi^\perp$. It is known that $\Phi\setminus\{0\}=\Psi\setminus\{0\}$ and $\dim W_\varphi=\dim W_\varphi^\perp$ for each $\varphi\in\Phi\setminus\{0\}$, that is, except possibly for complex subspaces in $W$ or $W^\perp$, the K\"{a}hler angles of $W$ and $W^\perp$ and their multiplicities are the same. We have $\C W_\varphi=W_\varphi\oplus W_\varphi^\perp$ for $\varphi\in\Phi\setminus\{0\}$, and moreover, $F^2\xi=-\cos^2(\varphi)\xi$ for each $\xi\in W_\varphi$ and each $\varphi\in\Phi$. Conversely, if $\xi\in W$ satisfies $F^2\xi=-\cos^2(\varphi)\xi$, then it follows from the decomposition of $W$ in subspaces of constant K\"{a}hler angle that $\xi\in W_\varphi$.

Finally, two subspaces $W$ and $\hat{W}$ of $V\cong\C^n$ are congruent by an element of $U(n)$ if and only if they have the same principal K\"{a}hler angles with the same multiplicities, that is, if $W=\oplus_{\varphi\in\Phi}W_\varphi$ and $\hat{W}=\oplus_{\varphi\in\Psi}\hat{W}_\varphi$ are as above, then they are congruent by an element of $U(n)$ if and only if $\Phi=\Psi$ and $\dim W_\varphi=\dim \hat{W}_\psi$ whenever $\varphi=\psi$.

\subsection{Examples of isoparametric hypersurfaces in complex hyperbolic spaces}\label{sec:examples}\hfill

\subsubsection{The standard examples}\label{sec:examples:standard}\hfill

The standard set of homogeneous examples of real hypersurfaces in the complex hyperbolic spaces is known as Montiel's list~\cite{Mo85}. Berndt~\cite{Be89} classified these examples:

\begin{theorem}\label{th:Hopf}
Let $M$ be a connected Hopf real hypersurface with constant principal curvatures of the complex hyperbolic space $\C H^n$, $n\geq 2$. Then, $M$ is holomorphically congruent to an open part of:
\begin{enumerate}[{\rm (i)}]
\item a tube around a totally geodesic $\C H^k$, $k\in\{0,\dots,n-1\}$, or

\item a tube around a totally geodesic $\R H^n$, or

\item a horosphere.
\end{enumerate}
\end{theorem}

\begin{remark}\label{rmk:Hopf}
In order to use Theorem~\ref{th:Hopf} efficiently (see for example Corollary~\ref{th:typeiv_classification} and Proposition~\ref{th:different_types}), we need to know the principal curvatures and their multiplicities for a Hopf real hypersurface with constant principal curvatures. These can be found for example in~\cite{Be89} or~\cite{BD09}.

A tube of radius $r>0$ around a totally geodesic $\C H^k$, $k\in\{0,\dots,n-1\}$, has the following principal curvatures:
\begin{align*}
\lambda_1&{}=\frac{\sqrt{-c}}{2}\tanh\Bigl(\frac{r\sqrt{-c}}{2}\Bigr),
&\lambda_2&{}=\frac{\sqrt{-c}}{2}\coth\Bigl(\frac{r\sqrt{-c}}{2}\Bigr),
&\lambda_3&{}=\sqrt{-c}\coth\Bigl(r\sqrt{-c}\Bigr),
\end{align*}
with multiplicities $2k$, $2(n-k-1)$, and $1$. Thus, the number of principal curvatures is $g=2$ if $k=0$ or $k=n-1$, and $g=3$ otherwise. The Hopf vector is associated with $\lambda_3$.

A tube of radius $r>0$ around a totally geodesic $\R H^n$ has three principal curvatures
\begin{align*}
\lambda_1&{}=\frac{\sqrt{-c}}{2}\tanh\Bigl(\frac{r\sqrt{-c}}{2}\Bigr),
&\lambda_2&{}=\frac{\sqrt{-c}}{2}\coth\Bigl(\frac{r\sqrt{-c}}{2}\Bigr),
&\lambda_3&{}=\sqrt{-c}\tanh\Bigl(r\sqrt{-c}\Bigr),
\end{align*}
with multiplicities $n-1$, $n-1$, and $1$, except when $r=\frac{1}{\sqrt{-c}}\log\bigl(2+\sqrt{3}\bigr)$, in which case $\lambda_1=\lambda_3$. The Hopf vector is associated with $\lambda_3$.

Finally, a horosphere has two distinct principal curvatures
\begin{align*}
\lambda_1&{}=\frac{\sqrt{-c}}{2},
&\lambda_2&{}=\sqrt{-c},
\end{align*}
with multiplicities $2(n-1)$ and $1$. The Hopf vector is associated with $\lambda_2$.
\end{remark}

It was believed for some time that, as it is the case for complex projective spaces, the Hopf hypersurfaces with constant principal curvatures (Theorem~\ref{th:Hopf}) should give the list of homogeneous hypersurfaces in complex hyperbolic spaces. However, Lohnherr and Reckziegel found in~\cite{L93} an example of a homogeneous hypersurface that is not Hopf, namely, case~(\ref{th:main:w}) in the Main Theorem. Later, new examples of non-Hopf homogeneous hypersurfaces in complex hyperbolic spaces were found in~\cite{BB01}, and Berndt and Tamaru classified all homogeneous hypersurfaces in~\cite{BT07}. The construction method of these non-Hopf examples was generalized by the first two authors in~\cite{DD12} for the complex hyperbolic space, and in~\cite{DD13} for Damek-Ricci spaces. These examples are in general not homogeneous, but they are isoparametric, and the rest of this section is devoted to present their definition and main properties.

\subsubsection{Tubes around the submanifolds $W_{\g{w}}$}\label{sec:examples:W}\hfill

Before starting with the description of the examples themselves, we need to introduce some concepts related to the algebraic structure of the complex hyperbolic space as a Riemannian symmetric space of rank one and noncompact type. See~\cite{BD09} for further details.

Indeed, $\C H^n$ can be written as $G/K$ where
$G=SU(1,n)$ and $K=S(U(1)U(n))$. We denote by gothic letters the Lie algebras of the corresponding Lie groups. Thus, if
$\g{g}=\g{k}\oplus\g{p}$ is the Cartan decomposition of $\g{g}$ with
respect to a point $o\in\C H^n$, and we choose a maximal abelian subspace
$\g{a}$ of $\g{p}$, it follows that $\g{a}$ is $1$-dimensional. Let
$\g{g}=\g{g}_{-2\alpha}\oplus\g{g}_{-\alpha}\oplus\g{g}_{0}
\oplus\g{g}_{\alpha}\oplus\g{g}_{2\alpha}$ be the root space
decomposition of $\g{g}$ with respect to $o$ and $\g{a}$. We
introduce an ordering in the set of roots so that $\alpha$ is a
positive root. These choices determine a point
at infinity $x$ in the ideal boundary $\C H^n(\infty)$ of $\C H^n$, that is, an equivalence class of geodesics that are asymptotic to the geodesic starting at $o\in\C H^n$, with direction $\g{a}\subset\g{p}\cong T_o\C H^n$ and the orientation determined by the fact that $\alpha$ is positive.
If we define $\g{n}=\g{g}_\alpha\oplus\g{g}_{2\alpha}$, then
$\g{g}=\g{k}\oplus\g{a}\oplus\g{n}$ is the so-called Iwasawa decomposition of
the Lie algebra $\g{g}$ with respect $o\in\C H^n$ and
$x\in\C H^n(\infty)$. If $A$, $N$ and $AN$ are
the connected simply connected subgroups of $G$ whose Lie
algebras are $\g{a}$, $\g{n}$, and $\g{a}\oplus\g{n}$ respectively, then
$G$ turns out to be diffeomorphic to $K\times A \times N$,  $AN$ is
diffeomorphic to $\C H^n$, and $T_o\C H^n\cong\g{a}\oplus\g{n}$. In this case, $G=KAN$ is the so-called Iwasawa decomposition of $G$. The metric and complex structure of $\C H^n$ induce a left-invariant metric $\langle\,\cdot\,,\,\cdot\,\rangle$ and a complex structure $J$ on $AN$ that make $\C H^n$ and
$AN$ isometric as K\"{a}hler manifolds.

Throughout this section $B$ will be the unit left-invariant vector field of
$\g{a}$ determined by the point at infinity~$x$. That is, the geodesic through $o$ whose
initial speed is $B$ converges to~$x$. We also set
$Z=JB\in\g{g}_{2\alpha}$, and thus, $\g{a}=\R B$ and
$\g{g}_{2\alpha}=\R Z$. Moreover, $\g{g}_\alpha$ is $J$-invariant, so
it is isomorphic to $\C^{n-1}$. The Lie algebra structure on
$\g{a}\oplus\g{n}$ is given by the formulas
\begin{equation}\label{eq:brackets}
\begin{aligned}
\left[B,Z\right]&{}=\sqrt{-c}\,Z,
&2\left[B,U\right]&{}=\sqrt{-c}\,U,
&\left[U,V\right]&{}=\sqrt{-c}\,\langle JU, V\rangle Z,
&\left[Z,U\right]&{}=0,
\end{aligned}
\end{equation}
where $U$, $V\in\g{g}_\alpha$.

In Section~\ref{sec:rigidity} we will also need the group structure of the semidirect product $AN$. A standard reference for this is~\cite{BTV95}. The product structure is given by
\begin{equation}\label{eq:product_AN}
\begin{aligned}
\!&\Exp_{\g{a}\oplus\g{n}}(aB+U+xZ)\cdot \Exp_{\g{a}\oplus\g{n}}(bB+V+yZ)\\
\!&\qquad{}=\Exp_{\g{a}\oplus\g{n}}\biggl((a+b)B
+\rho\Bigl(\frac{a+b}{2}\Bigr)^{-1}\Bigl(\rho(a/2)U+e^{a/2}\rho(b/2)V\Bigr)\\
\!&\phantom{\qquad{}=\Exp_{\g{a}\oplus\g{n}}\biggl(}
+\rho(a+b)^{-1}\Bigl(
\rho(a)x+e^a\rho(b)y+\frac{1}{2}e^{a/2}\sqrt{-c}\rho(a/2)\rho(b/2)\langle JU,V\rangle\Bigr)Z\biggr)
\end{aligned}
\end{equation}
for all $a$, $b$, $x$, $y \in \R$ and $U$, $V \in \g{g}_\alpha$. Here, $\Exp_{\g{a}\oplus\g{n}}\colon\g{a}\oplus\g{n}\to AN$ denotes the Lie exponential map of $AN$, and $\rho\colon\R\to\R$ is the analytic function defined by
\[
\rho(s)=
\begin{cases}
\frac{e^s-1}{s}    & \text{if }s\neq 0,\\
1                  & \text{if }s=0.
\end{cases}
\]

The Levi-Civita connection of $AN$ is
given by
\begin{equation}\label{eq:Levi-Civita}
\begin{aligned}
{\nabla}_{aB+U+xZ}(bB+V+yZ)
&{}=\sqrt{-c}\,\Bigl\{\Bigl(\frac{1}{2}\langle U,V\rangle+xy\Bigr)B\\
&\phantom{{}=\sqrt{-c}\,\Bigl(}
-\frac{1}{2}\Bigl(bU+yJU+xJV\Bigr)
+\Bigl(\frac{1}{2}\langle JU,V\rangle-bx\Bigr) Z\Bigr\},
\end{aligned}
\end{equation}
where $a$, $b$, $x$, $y\in\R$, $U$, $V\in\g{g}_\alpha$, and all
vector fields are considered to be left-invariant.\medskip

In order to construct the examples corresponding to cases~(\ref{th:main:w}) to~(\ref{th:main:ww}) of the Main Theorem, let $\g{w}$ be a proper real subspace of $\g{g}_\alpha$, that is, a subspace of $\g{g}_\alpha$, $\g{w}\neq\g{g}_\alpha$, where $\g{g}_\alpha$ is regarded as a real vector space. We define
$\g{w}^\perp=\g{g}_\alpha\ominus\g{w}$, the orthogonal complement of
$\g{w}$ in $\g{g}_\alpha$, and write $k=\dim\g{w}^\perp$. It follows from the bracket relations above that
$\g{a}\oplus\g{w}\oplus\g{g}_{2\alpha}$ is a solvable
Lie subalgebra of $\g{a}\oplus\g{n}$. We define
\[
W_{\g{w}}=S_{\g{w}}\cdot o,\text{ where $\g{s}_{\g{w}}=\g{a}\oplus\g{w}\oplus\g{g}_{2\alpha}$},
\]
the orbit of the group $S_{\g{w}}$ through the point $o$, where $S_{\g{w}}$ is the
connected subgroup of $AN$ whose Lie algebra is $\g{s}_{\g{w}}$. Hence, $W_{\g{w}}$ is a
homogeneous submanifold of $\C H^n$; it was proved in~\cite{DD12} that $W_{\g{w}}$ is minimal and tubes around $W_{\g{w}}$ are isoparametric hypersurfaces of~$\C H^n$.

We give some more information on $W_{\g{w}}$ and its tubes. As we have seen in Subsection~\ref{sec:Kahler_angles}, we can decompose $\g{w}^\perp=\oplus_{\varphi\in\Phi}\g{w}^\perp_\varphi$ as a direct sum of complex-orthogonal subspaces of constant K\"{a}hler angle. The elements of $\Phi$ are the principal K\"{a}hler angles of $\g{w}^\perp$. Recall that $F\colon\g{w}^\perp\to\g{w}^\perp$ and $P\colon\g{w}^\perp\to\g{w}$ map any $\xi \in \g{w}^\perp$ to the orthogonal projections of $J\xi$ onto $\g{w}^\perp$ and $\g{w}$ respectively.
Let $\g{c}$ be the maximal complex subspace of $\g{s}_\g{w}$, that is,
$\g{c}=\g{a}\oplus(\g{g}_\alpha\ominus\C\g{w}^\perp)\oplus\g{g}_{2\alpha}$. Then,  $\g{s}_{\g{w}}=\g{c}\oplus
P\g{w}^\perp$ and $\g{a}\oplus\g{n}=\g{c}\oplus
P\g{w}^\perp\oplus\g{w}^\perp$. Denoting by $\g{C}$, $P\g{W}^\perp$,
and $\g{W}^\perp$ the corresponding left-invariant distributions on
$AN$, then the tangent bundle of $W_{\g{w}}$ is $TW_{\g{w}}=\g{C}\oplus P\g{W}^\perp$
and the normal bundle is $\nu W_{\g{w}}=\g{W}^\perp$.
It follows from~\cite[p.~1039]{DD12} that the second fundamental form of $W_{\g{w}}$ is determined by the trivial symmetric bilinear extension of
\[
2\II(Z,P\xi)=-\sqrt{-c}\,(JP\xi)^\perp,\ \text{ $\xi\in\nu W_{\g{w}}$,}
\]
where $(\cdot)^\perp$ denotes orthogonal projection onto $\nu W_{\g{w}}$. It can be shown that this expression for the second fundamental form implies that the complex distribution $\g{C}$ on $W_\g{w}$ is autoparallel, and hence $W_\g{w}$ is ruled by totally geodesic complex hyperbolic subspaces (see~Lemma~\ref{th:complex}).

If $k=1$, that is, if $\g{w}$ is a real hyperplane in $\g{g}_\alpha$, then the corresponding $W_{\g{w}}$ is denoted by $W^{2n-1}$ and is called the \emph{Lohnherr hypersurface}~\cite{L93}. It follows that $W^{2n-1}$ and its equidistant hypersurfaces are homogeneous hypersurfaces of $\C H^n$. These were also studied by Berndt in~\cite{Be98}, and correspond to case~(\ref{th:main:w}) of the Main Theorem. The corresponding foliation on $\C H^n$ is sometimes called the solvable foliation.

Thus, we assume from now on $k>1$. If $\g{w}^\perp$ has constant K\"{a}hler angle $\varphi=0$, then $W_{\g{w}}$ is congruent to a totally geodesic complex hyperbolic space. If $\g{w}^\perp$ has  constant K\"{a}hler angle $\varphi\in(0,\pi/2]$, then $W_{\g{w}}$ is denoted by $W^{2n-k}_\varphi$. These are the so-called \emph{Berndt-Br\"{u}ck submanifolds}, and it is proved in~\cite{BB01} that the tubes around $W^{2n-k}_\varphi$ are homogeneous. Moreover, it follows from~\cite{BT07} that a real hypersurface in $\C H^n$ is homogeneous if and only if it is congruent to one of the Hopf examples in Theorem~\ref{th:Hopf}, to $W^{2n-1}$ or one of its equidistant hypersurfaces, or to a tube around a $W^{2n-k}_\varphi$.

In general, however, a tube around a submanifold $W_{\g{w}}$ is not necessarily homogeneous. For an arbitrary $\g{w}$, the mean curvature $\mathcal{H}^r$ of the tube $M^r$ of radius $r$ around the submanifold $W_{\g{w}}$ is~\cite{DD12}
\[
\mathcal{H}^r=
\frac{\sqrt{-c}}{2\sinh\frac{r\sqrt{-c}}{2}\cosh\frac{r\sqrt{-c}}{2}} \left(k-1+2n\sinh^2\frac{r\sqrt{-c}}{2}\right).
\]
Therefore, for every $r>0$, the tube $M^r$
of radius $r$ around $W_\g{w}$ is a hypersurface with constant mean curvature, and
hence, tubes around the submanifold $W_\g{w}$ constitute an
isoparametric family of hypersurfaces in $\C H^n$.

\begin{remark}\label{rmk:W:curvatures}
With the notation as above, if $\gamma_\xi$ denotes the geodesic through a point $o\in W_{\g{w}}$ with $\dot\gamma_\xi(0)=\xi\in\nu_o W_{\g{w}}$, then the characteristic polynomial of the shape operator of $M^r$ at $\gamma_\xi(r)$ with respect to~$-\gamma_\xi'(r)$ is
\[
p_{r,\xi}(x)=(\lambda-x)^{2n-k-2}\left(-\frac{c}{4\lambda}-x\right)^{k-2}
f_{\lambda,\varphi_\xi}(x),
\]
where $\lambda=\frac{\sqrt{-c}}{2}\tanh\frac{r\sqrt{-c}}{2}$, $\varphi_\xi$ is the K\"{a}hler angle of $\xi$ respect to $\nu_o W_{\g{w}}$, and
\[
f_{\lambda,\varphi}(x)=
-x^3+\left(-\frac{c}{4\lambda}+3\lambda\right)x^2
+\frac{1}{2}\left(c-6\lambda^2\right)x
+\frac{16\lambda^4-16c\lambda^2-c^2+
(c+4\lambda^2)^2\cos(2\varphi)}{32\lambda}.
\]

As was pointed out in~\cite{DD12}, at $\gamma_\xi(r)$, $M^r$ has the same
principal curvatures, with the same multiplicities, as a tube of radius $r$ around $W^{2n-k}_{\varphi_\xi}$, $\varphi_\xi\in[0,\pi/2]$. However, in general, the
principal curvatures and the number~$g$ of principal curvatures
vary from point to point in $M^r$.
\end{remark}

\section{Lorentzian isoparametric hypersurfaces}\label{sec:isoparametric}

In this section we present the possible eigenvalue structures of the shape operator of a Lorentzian isoparametric hypersurface in the anti-De Sitter space $H^{2n+1}_1$ and use this information to deduce some algebraic properties of an isoparametric hypersurface in the complex hyperbolic space $\C H^n$.

Let $\tilde{M}$ be a Lorentzian isoparametric hypersurface in $H^{2n+1}_1$. Then we know by~\cite[Proposition~2.1]{H84} that it has constant principal curvatures with constant algebraic multiplicities. The shape operator $\tilde{\Ss}_q$ at a point $q$ is a self-adjoint endomorphism of $T_q\tilde{M}$. It is known (see for example \cite[Chapter~9]{ON83}) that there exists a basis of $T_q \tilde{M}$ where $\tilde{\Ss}_q$ assumes one of the following Jordan canonical forms:
\begin{center}
\begin{tabular}{l@{}l@{\hspace{3em}}l@{}l}
I.&
$\begin{pmatrix}
\lambda_1 & & 0\\
&\ddots&\\
0 & & \lambda_{2n}
\end{pmatrix}$
&
II.
&$\begin{pmatrix}
\lambda_1 & 0 & \\
\varepsilon&\lambda_1&\\
 & & \lambda_{2}\\
 &&&\ddots\\
 &&&&\lambda_{2n-1}
\end{pmatrix},\ \varepsilon =\pm 1$
\\
III.&
$\begin{pmatrix}
\lambda_1 & 0 &1 \\
0&\lambda_1&0\\
0 & 1 & \lambda_1\\
&&&\lambda_2\\
 &&&&\ddots\\
 &&&&&\lambda_{2n-2}
\end{pmatrix}$
&IV.&
$\begin{pmatrix}
a & -b & \\
b&a&\\
 & & \lambda_{3}\\
 &&&\ddots\\
 &&&&\lambda_{2n}
\end{pmatrix}$
\end{tabular}
\end{center}
Here, the $\lambda_i\in\R$ can be repeated and, in case IV, $\lambda_1=a+ib,\lambda_2=a-ib$ ($b\neq 0$) are the complex eigenvalues of $\tilde{\Ss}_q$. In cases I and IV the basis with respect to which $\tilde{\Ss}_q$ is represented is orthonormal (with the first vector being timelike), while in cases II and III the basis is semi-null. A semi-null basis is a basis $\{u,v,e_1,\ldots,e_{m-2}\}$ for which all inner products are zero except $\langle u,v\rangle=\langle e_i,e_i\rangle=1$, for all $i=1,\ldots, m-2$. We will say that a point $q\in\tilde{M}$ is of type~I, II, III or~IV if the canonical form of $\tilde{\Ss}_q$ is of type~I, II, III or~IV, respectively.

\begin{remark}\label{rmk:example_types}
It can be seen by direct calculation that all points of the lift of a tube around a totally geodesic $\C H^k$, $k\in\{0,\dots,n-1\}$, are of type~I. Similarly, all points of the lift of a horosphere are of type~II, and all points of the lift of a tube around a totally geodesic $\R H^n$ are of type~IV. For the Lohnherr hypersurface $W^{2n-1}$ and its equidistant hypersurfaces, or for the tubes around the Berndt-Br\"{u}ck submanifolds $W^{2n-k}_\varphi$, all points of their lifts are of type~III. Nevertheless, it is important to point out that, in general, the lift of a tube around a submanifold $W_{\g{w}}$ does not have constant type: there might be points of type~I (if $\varphi_\xi=0$ in the notation of Subsection~\ref{sec:examples:W}) and of type~III (otherwise).
\end{remark}

Cartan's fundamental formula can be generalized to semi-Riemannian space forms. See~\cite{H84}, or~\cite[Satz~2.3.6]{Bu93} for a proof:

\begin{proposition}\label{cartan_formula}
Let $\tilde{M}$ be a Lorentzian isoparametric hypersurface in the anti-De Sitter space $H^{2n+1}_1$ of curvature $c/4$. If its (possibly complex) principal curvatures are $\lambda_1,\ldots, \lambda_{\tilde{g}}$ with algebraic multiplicities $m_1,\ldots, m_{\tilde{g}}$, respectively, and if for some $i\in\{1,\ldots, \tilde{g}\}$ the principal curvature $\lambda_i$ is real and its algebraic and geometric multiplicities coincide, then:
\[
\sum_{j=1,\,j\neq i}^{\tilde{g}} m_j \frac{c+4\lambda_i\lambda_j}{\lambda_i-\lambda_j}=0.
\]
\end{proposition}

Now let $M$ be an isoparametric real hypersurface in $\C H^n$ and~$\tilde{M}=\pi^{-1}(M)$ its lift to~$H^{2n+1}_1$. Then,~$\tilde{M}$ is a Lorentzian isoparametric hypersurface in the anti-De Sitter space. We use Cartan's fundamental formula to analyze the eigenvalue structure of~$M$.
Our approach here will be mostly based on elementary algebraic arguments.

We denote by $\xi$ a (local) unit normal vector field of $M$. For a point $q\in \tilde{M}$, the shape operator $\tilde{\Ss}$ of $\tilde{M}$ at $q$ with respect to $\xi^L_q$ can adopt one of the four possible types described above. We will analyze the possible principal curvatures of $M$ at the point $p=\pi(q)$ going through the four cases.

The following is an elementary result that we state without proof.

\begin{lemma}\label{lemma:inside_Cartan}
Let $c<0$, $p>0$, and define $\phi\colon\R\setminus\{p\}\to\R$ by $\phi(x)=\frac{c+4px}{p-x}$.
Then $\phi(x)>0$ if and only if $x>0$ and $\lvert x+\frac{c}{4x}\rvert<\lvert p+\frac{c}{4p}\rvert$.
\end{lemma}

We begin with a consequence of Cartan's fundamental formula that will be used in subsections~\ref{sec:typei},~\ref{sec:typeii} and~\ref{sec:typeiii}. See~\cite[\S2.4]{Bu93} and~\cite[Lemma~2.3]{X99}.

\begin{lemma}\label{types_I_II_III}
Let $q\in\tilde{M}$ be a point of type I, II or~III. Then the number $\tilde{g}(q)$ of constant principal curvatures at $q$ satisfies $\tilde{g}(q)\in\{1,2\}$. Moreover, if $\tilde{g}(q)=2$ and the principal curvatures are $\lambda$ and $\mu$, then $c+4\lambda\mu=0$.
\end{lemma}

\begin{proof}
Let $\Lambda$ be the set of principal curvatures of $\tilde{M}$ at $q$. The algebraic multiplicity of $\lambda\in\Lambda$ is denoted by $m_\lambda$. If $q$ is of type~II or~III, then the algebraic and geometric multiplicities of only one principal curvature $\mu_0\in\Lambda$ of $\tilde{M}$ at $q$ do not coincide.

By Proposition~\ref{cartan_formula}, we have
\begin{align*}
m_{\mu_0}\sum_{\mu\in\Lambda\setminus\{\mu_0\}} m_\mu\frac{c+4\mu_0\mu}{\mu_0-\mu}
&=\sum_{\lambda\in\Lambda} m_\lambda\Biggl(\sum_{\mu\in\Lambda\setminus\{\lambda\}} m_\mu\frac{c+4\lambda\mu}{\lambda-\mu}\Biggr)\\
&= \sum_{\lambda< \mu} m_\lambda m_\mu (c+4\lambda\mu)\left(\frac{1}{\lambda-\mu}
+\frac{1}{\mu-\lambda}\right)=0.
\end{align*}
Since $m_{\mu_0}\neq 0$, we have that the fundamental formula of Cartan is also satisfied for $\mu_0$.

Now let $q$ be a point of type I, II or~III. Then we have
\begin{equation}\label{eq_cartan}
\sum_{\mu\in\Lambda\setminus\{\lambda\}} m_\mu\frac{c+4\lambda\mu}{\lambda-\mu}=0,
\text{  for each $\lambda\in\Lambda$.}
\end{equation}

By a suitable choice of the normal vector field, we can assume that $\Lambda^+$, the set of positive principal curvatures, is nonempty; otherwise, there would be only one principal curvature $\lambda=0$, and hence $\tilde{g}=1$. Let $\lambda_0\in\Lambda$ be a positive principal curvature that minimizes $\lambda\in\Lambda^+\mapsto\lvert\lambda+c/(4\lambda)\rvert$. By Lemma~\ref{lemma:inside_Cartan} (with $p=\lambda_0$) we have $(c+4\lambda_0\mu)/(\lambda_0-\mu)\leq 0$ for all $\mu\in\Lambda\setminus\{\lambda_0\}$. Therefore,~\eqref{eq_cartan} implies $\tilde{g}\in\{1,2\}$, and if $\tilde{g}=2$, then $\Lambda=\{\lambda_0,\mu\}$ and $c+4\lambda_0\mu=0$.
\end{proof}

We will make extensive use of the relations, see~\eqref{eq:shape},
\[
\tilde{\Ss}V=-\frac{\sqrt{-c}}{2}J\xi^L\qquad\text{and}\qquad \langle \tilde{\Ss}V,V\rangle=0,
\]
where $V$ is a timelike unit vector field on $H_1^{2n+1}$ tangent to the fibers of the Hopf map $\pi$. In order to simplify the notation, we will put $v=V_q$, $\tilde{\Ss}=\tilde{\Ss}_q$, $\Ss=\Ss_p$, and remove the base point of a vector field from the notation whenever it does not lead to confusion.

\subsection{Type I points}\label{sec:typei}\hfill

We start our study with the diagonal setting.

\begin{proposition}\label{th:typei}
If $q\in\tilde{M}$ is of type I and $p=\pi(q)$, then $M$ is Hopf at $p$, and $g(p)\in\{2,3\}$. The principal curvatures of $M$ at $p$ are:
\begin{align*}
\lambda&{}\in\Bigl(-\frac{\sqrt{-c}}{2},\frac{\sqrt{-c}}{2}\Bigr),\lambda\neq 0,
&\mu&{}=-\frac{c}{4\lambda}\in\Bigl(-\infty,-\frac{\sqrt{-c}}{2}\Bigr)\cup
\Bigl(\frac{\sqrt{-c}}{2},\infty\Bigr),
&&\text{ and }\quad\lambda+\mu.
\end{align*}
The first two principal curvatures coincide with those of $\tilde{M}$ (one of them might not exist as a principal curvature of $M$ at $p$) and the last one is of multiplicity one and corresponds to the Hopf vector.
\end{proposition}

\begin{proof}
According to Lemma~\ref{types_I_II_III}, let $\lambda$ and $\mu=-c/(4\lambda)$ be the eigenvalues of $\tilde{\Ss}$ ($\mu$ might not exist).
We assume that the principal curvature space $T_\lambda(q)$ has Lorentzian signature.

First, assume that there exist two distinct principal curvatures $\lambda$ and $\mu$. Since $c+4\lambda\mu=0$, we have $\lambda,\mu\neq 0$. We can write $v=u+w$, where $u\in T_\lambda(q)$, and $w\in T_\mu(q)$.
Since $-1=\langle v,v\rangle=\langle u,u\rangle+\langle w,w\rangle$, we have that $u$ is timelike, and
\[
0=\langle \tilde{\Ss} v,v\rangle=\lambda \langle u,u\rangle+\mu \langle w,w\rangle=(\lambda-\mu)\langle u,u\rangle-\mu,
\]
whence $\langle u,u\rangle=\frac{\mu}{\lambda-\mu}<0$ and $\langle w,w\rangle=\frac{\lambda}{\mu-\lambda}>0$. In addition:
\[
J\xi^L=-\frac{2}{\sqrt{-c}}\tilde{\Ss} v=-\frac{2}{\sqrt{-c}}(\lambda u+\mu w).
\]
Both $T_\lambda(q)\ominus\R u$ and $T_\mu(q)\ominus\R w$ are orthogonal to $v$ and $J\xi^L$, and so, by~\eqref{eq:shape}, they descend via $\pi_{*q}$ to eigenvectors of $\Ss$ (which are orthogonal to $J\xi$) corresponding to the eigenspaces of $\lambda$ and $\mu$, respectively. For dimension reasons, $J\xi$ belongs to one eigenspace of $\Ss$. Since $\pi_* v=0$, we have $\pi_* w=-\pi_* u$, and thus, by~\eqref{eq:shape},
\begin{align*}
\Ss J\xi
&{}=\frac{-2}{\sqrt{-c}}(\lambda^2 \pi_* u +\mu^2 \pi_* w)
=\frac{-2}{\sqrt{-c}}(\lambda^2-\mu^2)\pi_* u\\
&{}=\frac{-2}{\sqrt{-c}}(\lambda+\mu)(\lambda\pi_* u+\mu\pi_* w)
=(\lambda+\mu)J\xi.
\end{align*}
Therefore $M$ has $g(p)\in\{2,3\}$ principal curvatures at $p$: $\lambda$, $\mu$ and $\lambda+\mu$, where one of the first two might not exist (depending on whether $T_\lambda(q)\ominus\R u$ or $T_\mu(q)\ominus\R w$ might be zero) and where the last one is of multiplicity one and corresponds to the Hopf vector. Since $4\lambda\mu+c=0$ and $\frac{\lambda}{\mu-\lambda},\frac{\mu}{\mu-\lambda}>0$ it readily follows that $\lvert\mu\rvert>\lvert\lambda\rvert$, and thus $\lvert\lambda\rvert<\sqrt{-c}/2$.

Now assume that there is just one principal curvature $\lambda$. Then $\tilde{\Ss} v=\lambda v$ and $0=\langle \tilde{\Ss} v,v\rangle=-\lambda$, but then $J\xi^L=-\frac{2}{\sqrt{-c}}\tilde{\Ss} v=0$, which makes no sense. So this case is impossible.
\end{proof}

\begin{remark}\label{rmk:tubes_CHk}
Note that for a certain $r\in\R$, one can write
\[
\lambda=\frac{\sqrt{-c}}{2}\tanh\left(\frac{r\sqrt{-c}}{2}\right),\ \mu=\frac{\sqrt{-c}}{2}\coth\left(\frac{r\sqrt{-c}}{2}\right),\
\text{ and }
\lambda+\mu=\sqrt{-c}\coth(r\sqrt{-c}).
\]
Therefore, if $M$ is an isoparametric hypersurface that lifts to a type I hypersurface, then $M$ is a Hopf real hypersurface with constant principal curvatures and, according to the classification of Hopf real hypersurfaces with constant principal curvatures in the complex hyperbolic space (Theorem~\ref{th:Hopf}) and to the principal curvatures of $M$, it is an open part of a tube around a totally geodesic $\C H^k$, $k\in\{0,\dots,n-1\}$. However, as we have mentioned in Remark~\ref{rmk:example_types}, it is possible for an isoparametric hypersurface of $\C H^n$ to have points of type~I and~III in the same connected component. We will have to address this difficulty later in this paper.
\end{remark}

\subsection{Type II points}\label{sec:typeii}\hfill

Now we tackle the second possibility for the Jordan canonical form of the shape operator.

\begin{proposition}\label{th:typeii}
If $q\in\tilde{M}$ is of type II and $p=\pi(q)$, then $M$ is Hopf at $p$, and $g(p)=2$. Moreover, $\tilde{M}$ has one principal curvature $\lambda=\pm \sqrt{-c}/2$, and the principal curvatures of $M$ at $p$ are $\lambda$ and $2\lambda$. The second one has multiplicity one and corresponds to the Hopf vector.
\end{proposition}

\begin{proof}
Let $\lambda$ and $\mu=-c/(4\lambda)$ be the eigenvalues of $\tilde{\Ss}$ ($\mu$ might not exist). Assume that $\tilde{\Ss}$ has a type II matrix expression with respect to a semi-null basis $\{e_1,e_2,\dots,e_{2n}\}$, where $\tilde{\Ss} e_1=\lambda e_1+\varepsilon e_2$, $\varepsilon\in\{-1,1\}$, $T_\lambda(q)=\spann\{e_2,\dots,e_k\}$ and $T_\mu(q)=\spann\{e_{k+1},\dots, e_{2n}\}$. As a precaution, for the calculations that follow we observe that $e_1\notin T_\lambda(q)$, but it still makes sense to write, for example, $T_\lambda(q)\ominus\R e_1=\spann\{e_3,\dots,e_k\}$.

First, assume that $\tilde{M}$ has two distinct principal curvatures $\lambda, \mu\neq 0$ at $q$ with $c+4\lambda\mu=0$. We can assume that $v=r_1 e_1+r_2 e_2+u+w$, where $u\in T_\lambda(q)$, $\langle e_1,u\rangle=\langle e_2, u\rangle=0$, $w\in T_\mu(q)$ and $r_1,r_2\in \R$. We have $-1=\langle v, v\rangle=2r_1r_2+\langle u,u\rangle+\langle w,w\rangle$, so $r_1,r_2\neq 0$. If $u\neq 0$, we define
\begin{align*}
e_1' &{}=e_1-\frac{\langle u,u\rangle}{2r_1^2}e_2+\frac{1}{r_1}u,
&e_2'   &{}=e_2,
\end{align*}
and then we have $\langle e_i',e_j'\rangle=\langle e_i,e_j\rangle$, $\tilde{\Ss} e_1'=\lambda e_1'+\varepsilon e_2'$, $\tilde{\Ss} e_2'=\lambda e_2'$, and $v=r_1 e_1'+(r_2+{\langle u,u\rangle}/(2r_1))e_2'+w$. This means that we could have assumed from the very beginning $u=0$.

Thus, we have $-1=\langle v, v\rangle=2r_1r_2+\langle w,w\rangle$ and $\tilde{\Ss} v=r_1\lambda e_1 +r_1\varepsilon e_2+r_2\lambda e_2+\mu w$, and hence $J\xi^L=-{2}(r_1\lambda e_1 +(r_1\varepsilon+r_2\lambda) e_2+\mu w)/{\sqrt{-c}}$. Taking into account that $2r_1r_2=-1-\langle w,w\rangle$, we have
\begin{align*}
1&{}=\langle J\xi^L,J\xi^L\rangle=-\frac{4}{c}\left(2r_1^2\lambda\varepsilon +2r_1r_2\lambda^2+\langle w,w\rangle\mu^2\right)
=-\frac{4}{c}\left(2r_1^2\lambda\varepsilon -\lambda^2+\langle w,w\rangle(\mu^2-\lambda^2)\right),\\
0&=\langle \tilde{\Ss} v,v\rangle=2r_1r_2\lambda+r_1^2\varepsilon
+\langle w,w\rangle\mu
=r_1^2\varepsilon-\lambda +\langle w,w\rangle(\mu-\lambda).
\end{align*}
These two equations give a linear system in the unknowns $r_1^2$ and $\langle w,w\rangle$.
As $\lambda\neq \mu$ and $c+4\lambda\mu=0$, it is immediate to prove that this system is compatible and determined, and $r_1^2=-(c+4\lambda\mu)/(4\varepsilon(\lambda-\mu))=0$, which gives a contradiction. Therefore, there cannot be two distinct eigenvalues of $\tilde{\Ss}$.

If $\tilde{\Ss}$ has just one eigenvalue $\lambda$, similar calculations as above (or just setting $w=0$ everywhere) yield $2\lambda\varepsilon r_1^2=-\frac{c}{4}+\lambda^2$, and $\varepsilon r_1^2=\lambda$, which is only possible if $\lambda=\pm{\sqrt{-c}}/{2}$ and $r_1^2={\sqrt{-c}}/{2}$.

Now, $T_\lambda(q)\ominus\R e_1$ is orthogonal to $v$ and $J\xi^L$. Thus, when we apply $\pi_{*q}$, the vectors in $T_\lambda(q)\ominus \R e_1$ descend to eigenvectors of $\Ss$ associated with the eigenvalue $\lambda$, which are also orthogonal to $J\xi$. For dimension reasons, $J\xi$ must also be an eigenvector of~$\Ss$. Furthermore, by~\eqref{eq:shape}, and since $0=\pi_* v=r_1\pi_* e_1+r_2\pi_* e_2$, we get
\begin{align*}
\Ss J\xi={}&-\frac{2}{\sqrt{-c}}(r_1\lambda^2 \pi_*e_1 +(2r_1\varepsilon\lambda+r_2\lambda^2) \pi_*e_2)
=-\frac{4\lambda r_1\varepsilon}{\sqrt{-c}}\pi_*e_2
=2\lambda J\xi\,.
\end{align*}
In conclusion, $M$ has $g(p)=2$ principal curvatures at $p$. One is $\lambda=\pm {\sqrt{-c}}/{2}$, which coincides with the unique principal curvature of $\tilde{M}$, and the other one is $2\lambda=\pm\sqrt{-c}$, which has multiplicity one and corresponds to the Hopf vector.
\end{proof}

\subsection{Type III points}\label{sec:typeiii}\hfill

Now we will assume that the minimal polynomial of the shape operator $\tilde{\Ss}$ has a triple root. This case is much more involved than the others, and indeed, Section~\ref{sec:type_III} will be mainly devoted to dealing with this possibility. For type~III points we will always take vectors $\{e_1,e_2,e_3\}$ such that
\begin{equation}\label{eq:seminull_basis}
\begin{gathered}
\langle e_1,e_1\rangle=\langle e_2,e_2\rangle=\langle e_1,e_3\rangle=\langle e_2,e_3\rangle=0, \qquad\langle e_1,e_2\rangle=\langle e_3,e_3\rangle=1,\\
\tilde{\Ss} e_1=\lambda e_1,
\qquad\tilde{\Ss} e_2=\lambda e_2+e_3,
\qquad\tilde{\Ss} e_3=e_1+\lambda e_3.
\end{gathered}
\end{equation}

\begin{proposition}\label{th:typeiii}
Let $q\in\tilde{M}$ be a point of type~III and let $\lambda$ be the principal curvature of $\tilde{M}$ at $q$ whose algebraic and geometric multiplicities do not coincide. Then, $\tilde{g}(q)\in\{1,2\}$, $\lambda\in\bigl(-{\sqrt{-c}}/{2},{\sqrt{-c}}/{2}\bigr)$; if there are two principal curvatures at $q$ and we denote the other one by $\mu$, then $c+4\lambda\mu=0$.
\end{proposition}

\begin{proof}
Let $\lambda$ and $\mu=-c/(4\lambda)$ be the eigenvalues of $\tilde{\Ss}$ ($\mu$ might not exist). Recall that $c+4\lambda\mu=0$ from Proposition~\ref{types_I_II_III}. Assume that $\tilde{\Ss}$ has a type III matrix expression, and take $\{e_1,e_2,e_3\}$ as in~\eqref{eq:seminull_basis}. The spaces $T_\lambda(q)\ominus\R e_2$ (recall that $e_2\notin T_\lambda(q)$) and $T_\mu(q)$ are spacelike. By changing the sign of the normal vector we can further assume $\lambda\geq 0$.

First, assume that there exist two distinct principal curvatures $\lambda,\mu\neq 0$ with $c+4\lambda\mu=0$. We can write $v=r_1 e_1+r_2 e_2+r_3 e_3+u+w$, where $u\in T_\lambda(q)\ominus\R e_2$, $w\in T_\mu(q)$. Taking an appropriate orientation of $\{e_1,e_2,e_3\}$ we can further assume $r_2\geq 0$. We have $-1=\langle v, v\rangle=2r_1r_2+r_3^2+\langle u,u\rangle+\langle w,w\rangle$. In particular, $r_2>0$ and $r_1<0$. If $u\neq 0$, we define
\begin{equation}\label{eq:modified_seminull}
\begin{aligned}
e_1' &{}=e_1,
&e_2'   &{}=-\frac{\langle u,u\rangle}{2r_2^2}e_1+e_2+\frac{1}{r_2}u,
&e_3'   &{}=e_3.
\end{aligned}
\end{equation}
Then, the $e_i'$'s satisfy~\eqref{eq:seminull_basis}, and also $v=(r_1+\langle u,u\rangle/(2r_2))e_1'+r_2 e_2'+r_3 e_3'+w$. This shows that we could have assumed from the very beginning $u=0$.

Thus we have $-1=\langle v,v\rangle=2r_1r_2+r_3^2+\langle w,w\rangle$, and $\tilde{\Ss} v=(r_1\lambda +r_3)e_1 +r_2\lambda e_2+(r_2+r_3\lambda) e_3+\mu w$, and hence $J\xi^L=-2\left((r_1\lambda +r_3)e_1 +r_2\lambda e_2+(r_2+r_3\lambda) e_3+\mu w\right)/\sqrt{-c}$.
Taking into account that $2r_1r_2=-1-r_3^2-\langle w,w\rangle$ we have
\begin{align*}
1&=\langle J\xi^L,J\xi^L\rangle=-\frac{4}{c}\left(2r_1r_2\lambda^2
+4r_2r_3\lambda+r_2^2+r_3^2\lambda^2+\langle w,w\rangle\mu^2\right)\\
&=-\frac{4}{c}\left(4r_2r_3\lambda+r_2^2-\lambda^2
+(\mu^2-\lambda^2)\langle w,w\rangle\right),\\[1ex]
0&=\langle \tilde{\Ss} v,v\rangle=2r_1r_2\lambda+2r_2r_3+r_3^2\lambda
+\mu\langle w,w\rangle
=2r_2r_3-\lambda+(\mu-\lambda)\langle w,w\rangle.
\end{align*}
Canceling the $r_2r_3$ addend, we get
\[
r_2^2+(\mu-\lambda)^2\langle w,w\rangle = -\frac{c}{4}-\lambda^2.
\]
Since $r_2>0$, we deduce $\lambda\in\bigl(-{\sqrt{-c}}/{2},{\sqrt{-c}}/{2}\bigr)$, $\lambda\neq 0$.

If $\tilde{M}$ has just one principal curvature $\lambda\geq 0$ at $q$, calculations are very similar to what we did above, just putting $w=0$. We also get
$\lambda\in(-{\sqrt{-c}}/{2},{\sqrt{-c}}/{2})$, although in this case $\lambda=0$ is possible.
\end{proof}

\subsection{Type IV points}\label{sec:typeiv}\hfill

The final possibility for the Jordan canonical form of a self-adjoint operator concerns the existence of a complex eigenvalue. Since an isoparametric hypersurface in the anti-De Sitter space has constant principal curvatures, if there is a complex eigenvalue at a point, then there is a complex eigenvalue at all points. Since type~IV matrices are the only ones with a nonreal eigenvalue we conclude

\begin{lemma}\label{th:type_iv_neighborhood}
If $\tilde{M}$ is a connected isoparametric hypersurface of the anti-De Sitter space, and $q\in\tilde{M}$ is a point of type~IV, then all the points of $\tilde{M}$ are of type~IV.
\end{lemma}

As a consequence of Cartan's fundamental formula (Proposition~\ref{cartan_formula}) we have (cf.~\cite[Satz~2.4.3]{Bu93} or~\cite[Lemma~2.4]{X99}):

\begin{lemma}\label{Xiao_IV}
Let $q\in\tilde{M}$ be a point of type IV and let $a\pm i b$ ($b\neq 0$) be the nonreal complex conjugate principal curvatures at $q$. We denote by $\Lambda$ the set of real principal curvatures at~$q$. Then $\tilde{g}(q)\in\{3,4\}$ and
\[
a(4\lambda^2-c)-\lambda(4 a^2+4 b^2-c)=0, \text{ for each $\lambda\in\Lambda$.}
\]
If $\tilde{g}(q)=4$, the real principal curvatures $\lambda$ and $\mu$ satisfy $c+4\lambda\mu=0$.
\end{lemma}

\begin{proof}
Let $a+ib$, $a-ib$ ($b\neq 0$) be the two complex principal curvatures, both with multiplicity one, and as usual we denote by $m_\lambda$ the multiplicity of $\lambda\in\Lambda$. Since $n\geq 2$, we have $\Lambda\neq\emptyset$. By Proposition~\ref{cartan_formula}, for each $\lambda\in\Lambda$ we have
\begin{equation}\label{eq_cartan_IV}
2\frac{a(4\lambda^2-c)-\lambda(4a^2+4b^2-c)}{(\lambda-a)^2+b^2}
+\sum_{\mu\in\Lambda\setminus\{\lambda\}}m_\mu
\frac{c+4\lambda\mu}{\lambda-\mu}=0.
\end{equation}

We denote by $\Lambda^+$ the set of positive principal curvatures at~$q$. We define the map $f\colon\R\to\R$, $x\mapsto f(x)=a(4x^2-c)-x(4a^2+4b^2-c)$.

Assume $a\leq 0$ and $\Lambda^+\neq\emptyset$. We define $\lambda_0$ to be a positive principal curvature that minimizes $\lambda\in\Lambda^+\mapsto\lvert\lambda+c/(4\lambda)\rvert$. Then, by Lemma~\ref{lemma:inside_Cartan} we get $(c+4\lambda_0\mu)/(\lambda_0-\mu)\leq 0$ for all $\mu\in\Lambda\setminus\{\lambda_0\}$. Since $f(\lambda_0)<0$, this gives a contradiction with~\eqref{eq_cartan_IV}. Thus, there cannot be positive principal curvatures if $a\leq 0$.
Similarly, we get that all real principal curvatures are nonnegative if $a\geq 0$. In particular, if $a=0$ then $\Lambda=\{0\}$ and hence $\tilde{g}=3$.

From now on we will assume, without losing generality, that $a>0$. Then, all real principal curvatures are nonnegative. But from \eqref{eq_cartan_IV} one sees that in fact $\lambda>0$ for all $\lambda\in\Lambda$, that is, $\Lambda=\Lambda^+$.

The function $f$ is a quadratic function with discriminant $(c+4a^2-4b^2)^2+64a^2b^2>0$, so $f$ has exactly two zeroes, say $x_1$ and $x_2$. We have $x_1x_2=-{c}/{4}>0$ and $x_1+x_2=(a^2+b^2-c/4)/a>0$, so we can assume $0<x_1<x_2=-c/(4x_1)$.

If $\lambda>0$, note that $\lambda\in(x_1,x_2)$ if and only if $\lvert\lambda+c/(4\lambda)\rvert<\lvert x_1+c/(4x_1)\rvert$. If $\Lambda\cap(x_1,x_2)\neq\emptyset$, we define $\lambda_0$ to be a principal curvature that minimizes $\lambda\in\Lambda\mapsto\lvert\lambda+c/(4\lambda)\rvert$. Then $f(\lambda_0)<0$ and $(c+4\lambda_0\mu)/(\lambda_0-\mu)\leq 0$ for all $\mu\in\Lambda\setminus\{\lambda_0\}$ by Lemma~\ref{lemma:inside_Cartan} (with $p=\lambda_0$), contradiction with~\eqref{eq_cartan_IV}. Thus, let $\lambda_0$ be a principal curvature that maximizes $\lambda\in\Lambda\mapsto\lvert\lambda+c/(4\lambda)\rvert$. In this case, $f(\lambda_0)\geq 0$ and $(c+4\lambda_0\mu)/(\lambda_0-\mu)\geq 0$ for all $\mu\in\Lambda\setminus\{\lambda_0\}$ by Lemma~\ref{lemma:inside_Cartan} (with $p=\mu$). Hence, by~\eqref{eq_cartan_IV} we get $f(\lambda_0)=0$, $\Lambda\subset\{x_1,x_2\}$, and the assertion follows.
\end{proof}

Before starting an algebraic analysis of the shape operator we need to prove the following inequality, which requires obtaining information from the Codazzi and Gauss equations.

\begin{lemma}\label{th:a2b2c}
With the notation as above we have $4a^2+4b^2+c\geq 0$.
\end{lemma}

\begin{proof}
First, recall that by Lemma~\ref{th:type_iv_neighborhood}, $\tilde{M}$ is of type IV everywhere with the same principal curvatures. We denote by $\lambda$ and $\mu$ the real principal curvatures ($\mu$ might not exist), and by $T_\lambda$ and $T_\mu$ the corresponding smooth principal curvature distributions. We also consider smooth vector fields $E_1$ and $E_2$ such that $\tilde{\Ss} E_1=a E_1+b E_2$, $\tilde{\Ss} E_2=-bE_1+a E_2$, $\langle E_1,E_1\rangle=-1$, $\langle E_2,E_2\rangle=1$, $\langle E_1,E_2\rangle=0$.

First of all we claim
\begin{equation}\label{eq:DEiEj}
\nabla_{E_i}E_j\in\Gamma(T_\lambda\oplus T_\mu),\quad\text{ for $i,j\in\{1,2\}$}.
\end{equation}
In order to prove this, first notice that $\langle E_i,E_j\rangle$ is constant, so in particular $\langle\nabla_{E_i}E_j,E_j\rangle=0$. On the other hand, by the Codazzi equation,
\begin{align*}
0
&{}=\langle\tilde{R}(E_1,E_2)E_2,\xi^L\rangle
=\langle(\nabla_{E_1}\tilde{\Ss}) E_2,E_2\rangle-\langle(\nabla_{E_2}\tilde{\Ss})E_1,E_2\rangle\\
&{}=\langle\nabla_{E_1}\tilde{\Ss}E_2,E_2\rangle
-\langle\tilde{\Ss}\nabla_{E_1}E_2,E_2\rangle
-\langle\nabla_{E_2}\tilde{\Ss}E_1,E_2\rangle
+\langle\tilde{\Ss}\nabla_{E_2}E_1,E_2\rangle
=-2b\langle\nabla_{E_1}E_1,E_2\rangle,
\end{align*}
so $\langle\nabla_{E_1}E_1,E_2\rangle=0$. Similarly, writing the Codazzi equation with $(E_1,E_2,E_1)$ gives $\langle\nabla_{E_2}E_2,E_1\rangle=0$. Altogether this proves~\eqref{eq:DEiEj}.

Now let $X\in\Gamma(T_\nu)$, with $\nu\in\{\lambda,\mu\}$. By applying the Codazzi equation to $(E_1,X,E_2)$, $(E_2,X,E_1)$, $(E_1,X,E_1)$, and $(E_2,X,E_2)$, we obtain
\begin{align*}
(\nu-a)\langle \nabla_{E_1}E_2,X\rangle+b \langle \nabla_{E_1}E_1,X\rangle =
(\nu-a)\langle \nabla_{E_2}E_1,X\rangle-b \langle \nabla_{E_2}E_2,X\rangle ={} & 0,\\
(\nu-a)\langle \nabla_{E_1}E_1,X\rangle-b\langle \nabla_{E_1}E_2,X\rangle =
(\nu-a)\langle \nabla_{E_2}E_2,X\rangle+b\langle \nabla_{E_2}E_1,X\rangle ={}&
2b\langle \nabla_{X}E_1,E_2\rangle.
\end{align*}
From this we get the following relations:
\begin{equation}\label{eq:DEiEjX}
\begin{aligned}
\langle\nabla_{E_1}E_1,X\rangle
&{}=\langle\nabla_{E_2}E_2,X\rangle
=\frac{2b(\nu-a)}{(\nu-a)^2+b^2}\langle \nabla_{X}E_1,E_2\rangle,\\
\langle\nabla_{E_1}E_2,X\rangle
&{}=-\langle\nabla_{E_2}E_1,X\rangle
=-\frac{2b^2}{(\nu-a)^2+b^2}\langle \nabla_{X}E_1,E_2\rangle.
\end{aligned}
\end{equation}

Now we use the Gauss equation and~\eqref{eq:DEiEj} to get
\begin{align*}
-\frac{c}{4}={}
&\langle\tilde R(E_1,E_2)E_2,E_1\rangle
=\langle R(E_1,E_2)E_2,E_1\rangle-\langle \tilde{\Ss}E_2,E_2\rangle \langle \tilde{\Ss}E_1,E_1\rangle
+\langle \tilde{\Ss}E_2,E_1\rangle \langle \tilde{\Ss}E_2,E_1\rangle\\
={}& \langle\nabla_{E_1}E_2,\nabla_{E_2}E_1\rangle
-\langle\nabla_{E_1}E_1,\nabla_{E_2}E_2\rangle
-\langle\nabla_{\nabla_{E_1}E_2}E_2,E_1\rangle
+\langle\nabla_{\nabla_{E_2}E_1}E_2,E_1\rangle +a^2+b^2.
\end{align*}

Finally, let $\{X_1,\dots,X_k\}$ be an orthonormal basis of $\Gamma(T_\lambda\oplus T_\mu)$ such that $\tilde{\Ss}X_i=\nu_i X_i$, with $\nu_i\in\{\lambda,\mu\}$. Taking into account~\eqref{eq:DEiEj}, and writing the previous covariant derivatives with respect to the previous basis,~\eqref{eq:DEiEjX} implies
\begin{align*}
-\frac{c}{4}-a^2-b^2&{}=\langle\nabla_{E_1}E_2,\nabla_{E_2}E_1\rangle
-\langle\nabla_{E_1}E_1,\nabla_{E_2}E_2\rangle
-\langle\nabla_{\nabla_{E_1}E_2}E_2,E_1\rangle
+\langle\nabla_{\nabla_{E_2}E_1}E_2,E_1\rangle\\
&{}=\sum_{i=1}^{k}\langle\nabla_{E_1}E_2,X_i\rangle \langle\nabla_{E_2}E_1,X_i\rangle
-\sum_{i=1}^{k}\langle\nabla_{E_1}E_1,X_i\rangle \langle\nabla_{E_2}E_2,X_i\rangle \\
&\phantom{{}={}}-\sum_{i=1}^{k}\langle\nabla_{E_1}E_2,X_i\rangle \langle\nabla_{X_i}E_2,E_1\rangle
+\sum_{i=1}^{k}\langle\nabla_{E_2}E_1,X_i\rangle \langle\nabla_{X_i}E_2,E_1\rangle\\
&{}=-\sum_{i=1}^{k}\frac{8b^2}{(\nu_i-a)^2+b^2} \langle\nabla_{X_i}E_1,E_2\rangle^2
\leq 0,
\end{align*}
from where the result follows.
\end{proof}

\begin{proposition}\label{th:typeiv}
If $q\in\tilde{M}$ is of type~IV and $p=\pi(q)$, then $M$ is Hopf at $p$. Let $\lambda$ and $\mu$ be the real principal curvatures of $\tilde{M}$ at $q$ ($\mu$ might not exist). Then the principal curvatures of $M$ at $p$ are
\[
\lambda,\ \mu,\ \text{ and }\  2a=\frac{4c\lambda}{c-4\lambda^2}\in\bigl(-\sqrt{-c},\sqrt{-c}\bigr),
\]
where $2a$ is the principal curvature associated with the Hopf vector.
\end{proposition}

\begin{proof}
Let $a\pm ib$ be the nonreal complex eigenvalues of $\tilde{\Ss}$ ($b\neq 0$). Let $\lambda$ and $\mu=-c/4\lambda$ be the real eigenvalues of $\tilde{\Ss}$ ($\mu$ might not exist). Assume that $\tilde{\Ss}$ has a type IV matrix expression and let $e_1$, $e_2\in T_q\tilde{M}$ such that $\tilde{\Ss} e_1=a e_1+b e_2$, $\tilde{\Ss} e_2=-be_1+a e_2$, $\langle e_1,e_1\rangle=-1$, $\langle e_2,e_2\rangle=1$, $\langle e_1,e_2\rangle=0$.

We can assume that $v=r_1 e_1+r_2 e_2+u+w$, where $u\in T_\lambda(q)$, $w\in T_\mu(q)$, and $r_1$, $r_2\in\R$. If there is only one principal curvature $\lambda$, then $\mu$ and $T_\mu(q)$ do not exist and it suffices to put $w=0$ throughout. We have $-1=\langle v, v\rangle=-r_1^2+r_2^2+\langle u,u\rangle+\langle w,w\rangle$ and $\tilde{\Ss} v=(r_1 a-r_2 b)e_1 +(r_2 a+r_1 b) e_2+\lambda u+\mu w$, and hence $J\xi^L=-{2}((r_1 a-r_2 b)e_1 +(r_2 a+r_1 b) e_2+\lambda u+\mu w)/{\sqrt{-c}}$.
Taking into account that $\langle u,u\rangle=-1+r_1^2-r_2^2-\langle w,w\rangle$ we have
\begin{align*}
1={}&\langle J\xi^L,J\xi^L\rangle=-\frac{4}{c}\bigl((-a^2+b^2+\lambda^2)(r_1^2-r_2^2) +4abr_1r_2+(\mu^2-\lambda^2)\langle w,w\rangle-\lambda^2\bigr),\nonumber\\
0={}&\langle \tilde{\Ss} v,v\rangle =(\lambda-a)(r_1^2-r_2^2)+2br_1r_2 +(\mu-\lambda)\langle w,w\rangle-\lambda.
\end{align*}
We can view the previous two equations as a linear system in the variables $r_1^2-r_2^2$ and~$r_1r_2$. The matrix of this system has determinant $-8b((a-\lambda)^2+b^2)/c\neq 0$, and thus has a unique solution. In fact,
\[
r_1^2-r_2^2=\frac{-c-8a\lambda+4\lambda^2+4(\lambda+\mu-2a)(\lambda-\mu)
\langle w,w\rangle}{4((a-\lambda)^2+b^2)}.
\]
Then we have
\[
0\leq\langle u,u\rangle=-1+r_1^2-r_2^2-\langle w,w\rangle=
-\frac{4a^2+4b^2+c+4((a-\mu)^2+b^2)\langle w,w\rangle}{4((a-\lambda)^2+b^2)}.
\]
Hence, as we knew that $4a^2+4b^2+c\geq 0$ by Lemma~\ref{th:a2b2c}, we must have $4a^2+4b^2+c=0$, and thus $u=w=0$.

This implies that $T_\lambda(q)$ and $T_\mu(q)$ are orthogonal to $v$ and $J\xi^L$, and therefore, they descend to the $\lambda$ and $\mu$ eigenspaces of $\Ss$ respectively, and they are orthogonal to $J\xi$. Again, for dimension reasons, $J\xi$ must be an eigenvector of $\Ss$ and thus $M$ is Hopf at $p$. We also have, taking into account $0=\pi_{*q} v=r_1\pi_* e_1+r_2\pi_* e_2$ and $b^2=-a^2-c/4$,
\begin{align*}
\Ss J\xi={}&-\frac{2}{\sqrt{-c}}((r_1a^2-2r_2ab-r_1b^2)\pi_* e_1
+(2r_1ab-r_2b^2+r_2a^2) \pi_* e_2)\\
={}&-\frac{2}{\sqrt{-c}}(2a(a r_1-b r_2)\pi_* e_1+2a(b r_1+a r_2)\pi_* e_2+\frac{c}{4}\pi_* v)
=2aJ\xi.
\end{align*}
Lemma \ref{Xiao_IV} and $4a^2+4b^2+c=0$ yield $a={2c\lambda}/{(c-4\lambda^2)}$. If $\lvert 2a\rvert\geq \sqrt{-c}$, then $0=4a^2+4b^2+c\geq 4b^2$, which is impossible because $b\neq 0$. Therefore, $\lvert 2a\rvert<\sqrt{-c}$, that is, the principal curvature associated with the Hopf vector in $M$ is in $(-\sqrt{-c},\sqrt{-c})$.
\end{proof}

\begin{corollary}\label{th:typeiv_classification}
Let $M$ be a connected isoparametric hypersurface in $\C H^n$ which lifts to a type IV hypersurface in $H^{2n+1}_1$ at some point. Then $M$ is an open part of a tube around a totally geodesic $\R H^n$.
\end{corollary}

\begin{proof}
By Lemma~\ref{th:type_iv_neighborhood}, every point of $\tilde{M}$ is of type~IV. From Proposition~\ref{th:typeiv} and the fact that $\tilde{M}$ has constant principal curvatures, we deduce that $M$ is Hopf and has constant principal curvatures. From the classification of Hopf hypersurfaces with constant principal curvatures in $\C H^n$ (Theorem~\ref{th:Hopf}), it follows that the unique such hypersurface whose Hopf principal curvature is less than $\sqrt{-c}$ in absolute value (see Remark~\ref{rmk:Hopf}) is a tube around a totally geodesic $\R H^n$.
\end{proof}

\subsection{Variation of the Jordan canonical form}\label{sec:different_types}\hfill

As was pointed out in Remark~\ref{rmk:example_types}, there are examples of isoparametric hypersurfaces in $\C H^n$ whose lift to the anti-De Sitter space might have varying Jordan canonical form. We clarify this a little more in the following

\begin{proposition}\label{th:different_types}
Let $M$ be a connected isoparametric hypersurface in $\C H^n$, $n\geq 2$, and denote by $\tilde{M}=\pi^{-1}(M)$ its lift to the anti-De Sitter space. Then,
\begin{enumerate}[{\rm (i)}]
\item  If a point $q\in\tilde{M}$ is of type~IV, then all the points of $\tilde{M}$ are of type~IV, and $M$ is an open part of a tube around a totally geodesic $\R H^n$ in $\C H^n$.

\item If a point $q\in\tilde{M}$ is of type~II, then all the points of $\tilde{M}$ are of type~II, and $M$ is an open part of a horosphere in $\C H^n$.

\item If there is a point $q\in\tilde{M}$ of type~III, then there is a neighborhood of $q$ where all points are of type~III.
\end{enumerate}
\end{proposition}

\begin{proof}
The first statement is simply a consequence of Lemma~\ref{th:type_iv_neighborhood} and Corollary~\ref{th:typeiv_classification}.

Assume now that $q\in\tilde{M}$ is of type~II, and recall that $\tilde{M}$ has constant principal curvatures. Then, according to Proposition~\ref{th:typeii}, $\tilde{M}$ has exactly one principal curvature at $q$ that is $\pm\sqrt{-c}/2$. If $q_0\in\tilde{M}$ is another point of type~I or~III, then propositions~\ref{th:typei} and~\ref{th:typeiii} say that $\pm\sqrt{-c}/2$ cannot be a principal curvature of~$\tilde{M}$ at~$q_0$. Since $\tilde{M}$ is connected we conclude that all the points of $\tilde{M}$ are of type~II. But now the classification of Hopf real hypersurfaces with constant principal curvatures in complex hyperbolic spaces (Theorem~\ref{th:Hopf} together with Remark~\ref{rmk:Hopf}) implies that $M$ is an open part of a horosphere.

Finally, assume that $q\in\tilde{M}$ is of type~III. By definition, the difference between the algebraic and geometric multiplicities of $\lambda$ is a lower semi-continuous function on $\tilde{M}$. In our case, this function can only take the values $0$ (at points of type~I) and $2$ (at points of type~III). Hence we conclude.
\end{proof}

\section{Type~III hypersurfaces}\label{sec:type_III}

The aim of this section is to study isoparametric hypersurfaces of the anti-De Sitter space all of whose points are of type~III, and determine the extrinsic geometry of their focal submanifolds.

Let $M$ be a connected isoparametric real hypersurface in the complex hyperbolic space $\C H^n$, $n\geq 2$. We denote by $\tilde{M}=\pi^{-1}(M)$ its lift to the anti-De Sitter space. Assume that there are points in $\tilde{M}$ of type~III. According to Proposition~\ref{th:different_types}, if $q\in \tilde{M}$ is a point of type~III, then there is a neighborhood of $q$ where all points are also of type~III.

Thus, we assume that we are working on a connected open subset $\tilde{\mathcal{W}}$ of $\tilde{M}=\pi^{-1}(M)$ where all points are of type~III. We denote by $\xi$ a unit (spacelike) normal vector field along~$\tilde{\mathcal{W}}$. We know that $\tilde{M}$ has at most two distinct constant principal curvatures (see Proposition~\ref{th:typeiii}). We call $\lambda$ the principal curvature whose algebraic and geometric multiplicities do not coincide, and $\mu$ the other one, if it exists. Note that if there are two distinct principal curvatures, then $c+4\lambda\mu=0$. We denote by $T_\lambda$ and $T_\mu$ the corresponding principal curvature distributions, and choose smooth vector fields $E_1$, $E_2$, $E_3\in\Gamma(T\tilde{\mathcal{W}})$ satisfying~\eqref{eq:seminull_basis} at each point. Recall that $T_\lambda=\R E_1\oplus(T_\lambda\ominus\R E_2)$.

We also denote $m_\lambda=\dim T_\lambda+2$ and $m_\mu=\dim T_\mu$, the algebraic multiplicities of $\lambda$ and~$\mu$. Since $\tilde{\mathcal{W}}$ is isoparametric and all points are of type~III, $m_\lambda$ and $m_\mu$ are constant functions, and in principle $m_\lambda\geq 3$, $m_\mu\geq 0$. In fact, $\mu$ might not exist, and in this case, $m_\mu=0$.

\subsection{Covariant derivatives of an isoparametric hypersurface}\label{sec:codazzi}\hfill

Recall that $\xi^L$ denotes a unit normal vector field along $\tilde{\mathcal{W}}$. By $\nabla$ and $R$ we denote the Levi-Civita connection and curvature tensor of $\tilde{\mathcal{W}}$, and by $\tilde{\nabla}$ and $\tilde{R}$ the Levi-Civita connection and the curvature tensor of the anti-De Sitter spacetime, respectively. The aim of this subsection is to prove the following result:

\begin{proposition}\label{th:DE1W}
For any $W\in\Gamma(T_\mu)$ we have $\tilde{\nabla}_{E_1}W\in\Gamma(T_\mu)$.
\end{proposition}

We may assume $m_\mu>0$; otherwise, if $m_\mu=0$, this is trivial. We will carry out the proof in several steps. The first step almost finishes the argument except for an $E_1$-component.

\begin{lemma}\label{th:DE1Wa}
For any $W\in\Gamma(T_\mu)$ we have $\tilde{\nabla}_{E_1}W\in\Gamma(\R E_1\oplus T_\mu)$.
\end{lemma}

\begin{proof}
First, recall that $\tilde{\nabla}_{E_1}W=\nabla_{E_1}W+\langle\Ss E_1,W\rangle\xi^L=\nabla_{E_1}W$, so it suffices to work with~$\nabla$. Let $X\in\Gamma(\R E_3\oplus T_\lambda)$. The result follows if we show $\langle\nabla_{E_1}W,X\rangle=0$. First of all, the Codazzi equation and the fact that $\Ss$ is self-adjoint imply:
\begin{align*}
0
&{}=\langle\tilde{R}(E_1,W)X,\xi^L\rangle
=\langle(\nabla_{E_1}\Ss)W,X\rangle-\langle(\nabla_{W}\Ss)E_1,X\rangle\\
&{}=\mu\langle\nabla_{E_1}W,X\rangle-\langle\nabla_{E_1}W,\Ss X\rangle
-\lambda\langle\nabla_W E_1,X\rangle+\langle\nabla_W E_1,\Ss X\rangle.
\end{align*}
Taking $X\in\Gamma(T_\lambda)$ in this formula gives $0=(\mu-\lambda)\langle\nabla_{E_1}W,X\rangle$. In particular, $\langle\nabla_{E_1}W,E_1\rangle=0$. Using this, $\langle\nabla_W E_1,E_1\rangle=0$ (because $\langle E_1,E_1\rangle =0$), and putting $X=E_3$ in the previous equation yields
\begin{align*}
0
&{}=\mu\langle\nabla_{E_1}W,E_3\rangle-\langle\nabla_{E_1}W,E_1+\lambda E_3\rangle
-\lambda\langle\nabla_W E_1,E_3\rangle+\langle\nabla_W E_1,E_1+\lambda E_3\rangle\\
&{}=(\mu-\lambda)\langle\nabla_{E_1}W,E_3\rangle,
\end{align*}
from where the assertion follows.
\end{proof}

Thus, in order to conclude the proof of Proposition~\ref{th:DE1W} it just remains to show that $\langle\nabla_{E_1}W,E_2\rangle=0$. This will take most of the effort of this subsection. The next lemma is known (see for example~\cite[Propostion~2.6]{H84}), but we include its proof here for the sake of completeness.

\begin{lemma}\label{th:DTmuTmu}
$T_\mu$ is an autoparallel distribution: if $W_1$, $W_2\in\Gamma(T_\mu)$, then $\nabla_{W_1}W_2\in\Gamma(T_\mu)$.
\end{lemma}

\begin{proof}
Let $X\in\Gamma(\R E_2\oplus\R E_3\oplus T_\lambda)$. It suffices to prove that $\langle\nabla_{W_1}W_2,X\rangle=0$. Since $\Ss$ is self-adjoint and $\Ss X$ is orthogonal to $T_\mu$, the Codazzi equation implies
\begin{align*}
0
&{}=\langle\tilde{R}(X,W_1)W_2,\xi^L\rangle
=\langle(\nabla_X\Ss)W_1,W_2\rangle-\langle(\nabla_{W_1}\Ss)X,W_2\rangle\\
&{}=-\langle\nabla_{W_1}\Ss X,W_2\rangle+\langle\Ss\nabla_{W_1}X,W_2\rangle
=\langle\nabla_{W_1}W_2,\Ss X-\mu X\rangle.
\end{align*}
Taking $X\in\Gamma(T_\lambda)$ in this formula yields $0=(\lambda-\mu)\langle\nabla_{W_1}W_2,X\rangle=0$. In particular, $\langle\nabla_{W_1}W_2,E_1\rangle=0$. This, and setting $X=E_3$ above yields
\[
0=\langle\nabla_{W_1}W_2,E_1+\lambda E_3-\mu E_3\rangle=(\lambda-\mu)\langle\nabla_{W_1}W_2,E_3\rangle.
\]
This equation, and setting $X=E_2$ in the previous equation yields
\[
0=\langle\nabla_{W_1}W_2,\lambda E_2+E_3-\mu E_2\rangle=(\lambda-\mu)\langle\nabla_{W_1}W_2,E_2\rangle,
\]
as we wanted to show.
\end{proof}

In order to finish the proof of Proposition~\ref{th:DE1W} we use the Gauss equation to get
\begin{align}
0
&{}=\langle\tilde{R}(W,E_1)W,E_3\rangle
=\langle{R}(W,E_1)W,E_3\rangle\nonumber\\
&\phantom{{}={}}+\langle\Ss W,W\rangle\langle\Ss E_1,E_3\rangle
-\langle\Ss E_1,W\rangle\langle\Ss W,E_3\rangle\nonumber\\
&{}=\langle\nabla_{W}\nabla_{E_1} W,E_3\rangle
-\langle\nabla_{E_1}\nabla_{W}W,E_3\rangle
-\langle\nabla_{\nabla_{W}E_1}W,E_3\rangle
+\langle\nabla_{\nabla_{E_1}W}W,E_3\rangle.\label{eq:gauss1w3w}
\end{align}

Lemma~\ref{th:DE1Wa} yields $\nabla_{E_1}W\in\Gamma(\R E_1\oplus T_\mu)$. Write $\nabla_{E_1}W=\langle\nabla_{E_1}W,E_2\rangle E_1+(\nabla_{E_1}W)_{T_\mu}$ accordingly. By Lemma~\ref{th:DTmuTmu}, $\nabla_W(\nabla_{E_1}W)_{T_\mu}\in\Gamma(T_\mu)$, and thus $\langle\nabla_W(\nabla_{E_1}W)_{T_\mu},E_3\rangle=0$. Since $\langle E_1,E_3\rangle=0$, this implies $\langle\nabla_{W}\nabla_{E_1} W,E_3\rangle=\langle\nabla_{E_1}W,E_2\rangle\langle\nabla_W E_1,E_3\rangle$.

From Lemma~\ref{th:DTmuTmu} we have $\nabla_WW\in\Gamma(T_\mu)$, and thus Lemma~\ref{th:DE1Wa} yields $\nabla_{E_1}\nabla_W W\in\Gamma(\R E_1\oplus T_\mu)$. Hence, $\langle\nabla_{E_1}\nabla_{W}W,E_3\rangle=0$.

Lemma~\ref{th:DE1Wa} yields $\nabla_{E_1}W\in\Gamma(\R E_1\oplus T_\mu)$, which together with lemmas~\ref{th:DE1Wa} and~\ref{th:DTmuTmu} gives $\langle\nabla_{\nabla_{E_1}W}W,E_3\rangle=0$.

Hence, \eqref{eq:gauss1w3w} now reads
\begin{equation}\label{eq:gauss2}
0=\langle\nabla_{E_1}W,E_2\rangle\langle\nabla_W E_1,E_3\rangle
-\langle\nabla_{\nabla_{W}E_1}W,E_3\rangle.
\end{equation}

\begin{lemma}
Let $U\in\Gamma(T_\lambda\ominus\R E_2)$ and $W\in\Gamma(T_\mu)$. Then,
\begin{align}
\langle\nabla_W E_1,E_3\rangle={}
&(\lambda-\mu)\langle\nabla_{E_1}W,E_2\rangle,\label{eq:DWE1E3}\\
\langle\nabla_{E_3}W,E_3\rangle={}
&-2\langle\nabla_{E_1}W,E_2\rangle,\label{eq:DE3WE3}\\
\langle\nabla_W E_1,U\rangle={}
&-(\lambda-\mu)\langle\nabla_U W,E_3\rangle.\label{eq:DWE1U}
\end{align}
\end{lemma}

\begin{proof}
The Codazzi equation and Lemma~\ref{th:DE1Wa} imply
\begin{align*}
0
&{}=\langle\tilde{R}(E_1,W)E_2,\xi^L\rangle
=\langle(\nabla_{E_1}\Ss)W,E_2\rangle-\langle(\nabla_W \Ss)E_1,E_2\rangle\\
&{}=\mu\langle\nabla_{E_1}W,E_2\rangle
-\langle\nabla_{E_1}W,\Ss E_2\rangle
-\lambda\langle\nabla_W E_1,E_2\rangle
+\langle\nabla_W E_1,\Ss E_2\rangle\\
&{}=(\mu-\lambda)\langle\nabla_{E_1}W,E_2\rangle+\langle\nabla_W E_1,E_3\rangle,
\end{align*}
from where we get~\eqref{eq:DWE1E3}.

We also have
\begin{align*}
0
&{}=\langle\tilde{R}(E_3,W)E_1,\xi^L\rangle
=\langle(\nabla_{E_3}\Ss)W,E_1\rangle-\langle(\nabla_W \Ss)E_3,E_1\rangle
=(\mu-\lambda)\langle\nabla_{E_3}W,E_1\rangle.
\end{align*}
Thus, $\langle\nabla_{E_3}W,E_1\rangle=0$. This, the Codazzi equation, and~\eqref{eq:DWE1E3} yield
\begin{align*}
0
&{}=\langle\tilde{R}(E_3,W)E_3,\xi^L\rangle
=\langle(\nabla_{E_3}\Ss)W,E_3\rangle-\langle(\nabla_W \Ss)E_3,E_3\rangle\\
&{}=(\mu-\lambda)\langle\nabla_{E_3}W,E_3\rangle
-\langle\nabla_{E_3}W,E_1\rangle
-2\langle\nabla_W E_1,E_3\rangle\\
&{}=-(\lambda-\mu)\langle\nabla_{E_3}W,E_3\rangle
-2(\lambda-\mu)\langle\nabla_{E_1}W,E_2\rangle,
\end{align*}
which gives~\eqref{eq:DE3WE3}.

Now, the Codazzi equation and Lemma~\ref{th:DE1Wa} imply
\begin{align*}
0
&{}=\langle\tilde{R}(E_1,U)W,\xi^L\rangle
=\langle(\nabla_{E_1}\Ss)U,W\rangle-\langle(\nabla_U \Ss)E_1,W\rangle\\
&{}=(\lambda-\mu)\langle\nabla_{E_1}U,W\rangle
-(\lambda-\mu)\langle\nabla_U E_1,W\rangle
=-(\lambda-\mu)\langle\nabla_U E_1,W\rangle,
\end{align*}
and thus we get $\langle\nabla_U E_1,W\rangle=0$. This implies
\begin{align*}
0
&{}=\langle\tilde{R}(E_3,U)W,\xi^L\rangle
=\langle(\nabla_{E_3}\Ss)U,W\rangle-\langle(\nabla_U \Ss)E_3,W\rangle\\
&{}=(\lambda-\mu)\langle\nabla_{E_3}U,W\rangle
-\langle\nabla_{U}E_1,W\rangle
-(\lambda-\mu)\langle\nabla_U E_3,W\rangle\\
&{}=(\lambda-\mu)(\langle\nabla_{E_3}U,W\rangle-\langle\nabla_U E_3,W\rangle),
\end{align*}
from where we obtain $\langle\nabla_{E_3}W,U\rangle=\langle\nabla_U W,E_3\rangle$.
Finally, this equation gives
\begin{align*}
0
&{}=\langle\tilde{R}(E_3,W)U,\xi^L\rangle
=\langle(\nabla_{E_3}\Ss)W,U\rangle-\langle(\nabla_W \Ss)E_3,U\rangle\\
&{}=(\mu-\lambda)\langle\nabla_{E_3}W,U\rangle
-\langle\nabla_{W}E_1,U\rangle
=-(\lambda-\mu)\langle\nabla_{U}W,E_3\rangle
-\langle\nabla_{W}E_1,U\rangle,
\end{align*}
which concludes the proof of the lemma.
\end{proof}

Now we come back to~\eqref{eq:gauss2} and finish the proof of Proposition~\ref{th:DE1W}.

Using Lemma~\ref{th:DTmuTmu} we see that $\nabla_W E_1\in\Gamma(\R E_1\oplus\R E_3\oplus T_\lambda)$. Take $\{U_1,\dots,U_k\}$ an orthonormal basis of vector fields of the distribution $T_\lambda\ominus \R E_2$. Thus, we can write, taking into account~\eqref{eq:DWE1E3} and~\eqref{eq:DWE1U},
\begin{equation}\label{eq:DWE1}
\begin{aligned}
\nabla_W E_1
&{}=\langle\nabla_W E_1,E_2\rangle E_1+
\langle\nabla_W E_1,E_3\rangle E_3
+\sum_{i=1}^k\langle\nabla_W E_1,U_i\rangle U_i\\[-2ex]
&{}=\langle\nabla_W E_1,E_2\rangle E_1+
(\lambda-\mu)\langle\nabla_{E_1}W,E_2\rangle E_3
-(\lambda-\mu)\sum_{i=1}^k\langle\nabla_{U_i}W,E_3\rangle U_i.
\end{aligned}
\end{equation}
Hence, using~\eqref{eq:DWE1}, Lemma~\ref{th:DE1Wa},~\eqref{eq:DWE1E3} and~\eqref{eq:DE3WE3}, Equation~\eqref{eq:gauss2} becomes
\begin{align*}
0
&{}=(\lambda-\mu)\langle\nabla_{E_1}W,E_2\rangle^2
-\langle\nabla_W E_1,E_2\rangle\langle\nabla_{E_1}W,E_3\rangle
-(\lambda-\mu)\langle\nabla_{E_1}W,E_2\rangle
\langle\nabla_{E_3}W,E_3\rangle\\[-1ex]
&\phantom{{}={}}+(\lambda-\mu)\sum_{i=1}^k
\langle\nabla_{U_i}W,E_3\rangle\langle\nabla_{U_i}W,E_3\rangle\\[-2ex]
&{}=(\lambda-\mu)\Bigl(3\langle\nabla_{E_1}W,E_2\rangle^2
+\sum_{i=1}^k \langle\nabla_{U_i}W,E_3\rangle^2\Bigr).
\end{align*}
Since the addends are all nonnegative, we must have
\[
\langle\nabla_{E_1}W,E_2\rangle=0, \text{ and } \langle\nabla_U W,E_3\rangle=0
\text{ for any $U\in\Gamma(T_\lambda\ominus\R E_2)$ and $W\in\Gamma(T_\mu)$},
\]
which is what was left to finish the proof of Proposition~\ref{th:DE1W}.

%

\subsection{Parallel hypersurfaces and the focal manifold}\label{sec:jacobi}\hfill

We continue to denote by $\tilde{\mathcal{W}}$ a connected open subset of the Lorentzian isoparametric hypersurface $\tilde{M}=\pi^{-1}(M)$ of the anti-De Sitter space $H_1^{2n+1}$ where all points are of type~III, and let $\mathcal{W}=\pi(\tilde{\mathcal{W}})\subset M$. If $\xi$ denotes a unit normal vector field along $\mathcal{W}$, then $\xi^{L}$ is a local unit vector field along $\tilde{\mathcal{W}}$. As a matter of notation, $\tilde{\gamma}_{q}$ will be the geodesic in $H_1^{2n+1}$ such that $\tilde{\gamma}_{q}(0) =q\in\tilde{\mathcal{W}}$ and ${\tilde{\gamma}}'_{q} (0) =  \xi^L_q$. Accordingly, we write $\gamma_p=\pi\circ\tilde{\gamma}_q$ for the geodesic in $\mathbb{C}H^n$ with initial conditions $\gamma_p(0) = p =\pi(q)$ and ${\gamma}'_p(0)= \xi_p$.

Recall from Section~\ref{sec:submanifold} the definition of the map $\tilde{\Phi}^t\colon\tilde{\mathcal{W}}\to H_1^{2n+1}$, given by $\tilde{\Phi}^t(q)=\exp_q(t\xi^L)=\tilde{\gamma}_q(t)$, where $\exp$ is the Riemannian exponential map. We also consider the vector field $\eta^t$ along $\tilde{\Phi}^t$ defined by $\eta^t(q)=\tilde{\gamma}'_q(t)$.

The differential of $\tilde{\Phi}^t$ is given by $\tilde{\Phi}^t_{*q}(X)=\zeta_X(t)$, where $\zeta_X$ is a Jacobi vector field along $\tilde{\gamma}_q$ with initial conditions $\zeta_X(0)=X\in T_q\tilde{\mathcal{W}}$, and $\zeta_X'(0)=-\tilde{\Ss} X$, where $(\cdot)'$ stands for covariant differentiation along $\tilde{\gamma}_q$ (see~\cite[\S8.2]{BCO03}). Since $H_1^{2n+1}$ is a space of constant sectional curvature $c/4$ and $\tilde{\gamma}'$ is spacelike, it follows that the Jacobi equation is written as $4\zeta_X''+c\zeta_X=0$.

Let $\mathcal{P}_X(t)$ denote the parallel translation of $X\in T_q\tilde{\mathcal{W}}$ along $\tilde{\gamma}_q$. For $\nu\in\R$, we also define
\[
g_\nu(t)=\cosh \Bigl(\frac{t \sqrt{-c}}{2}\Bigr) - \frac{2\nu}{\sqrt{-c}} \sinh\Bigl(\frac{t \sqrt{-c}}{2}\Bigr)\quad \text{ and }\quad
h(t)=-\frac{2}{\sqrt{-c}}\sinh\Bigl(\frac{t \sqrt{-c}}{2}\Bigr).
\]
Solving the Jacobi equation we get
\begin{equation}\label{eq:Jacobi}
\begin{aligned}
\zeta_X(t)&{}=g_\lambda(t)\mathcal{P}_X(t),\ \text{ if $X\in T_\lambda(q)$,}&
\zeta_X(t)&{}=g_\mu(t)\mathcal{P}_X(t),\ \text{ if $X\in T_\mu(q)$,}\\
\zeta_{E_2(q)}(t)&{}=g_\lambda(t)\mathcal{P}_{E_2(q)}(t)
+h(t)\mathcal{P}_{E_3(q)}(t),&
\zeta_{E_3(q)}(t)&{}=h(t)\mathcal{P}_{E_1(q)}(t)
+g_\lambda(t)\mathcal{P}_{E_3(q)}(t).
\end{aligned}
\end{equation}

Since we are denoting by $\lambda$ the principal curvature whose geometric and algebraic multiplicities do not coincide, it follows from Proposition~\ref{th:typeiii} that $\lvert\lambda\rvert<\sqrt{-c}/2$. We assume, changing the orientation if necessary, that $\lambda\geq 0$. Recall that, if a second distinct principal curvature $\mu$ exists, then $c+4\lambda\mu=0$, which implies $\lambda,\mu\neq0$. We may choose $r\geq 0$ such that
\begin{equation}\label{eq:r}
\lambda = \frac{\sqrt{-c}}{2} \tanh\Bigl(\frac{r \sqrt{-c}}{2}\Bigr)\quad
\text{ and }\quad
\mu=\frac{\sqrt{-c}}{2} \coth\Bigl(\frac{r \sqrt{-c}}{2}\Bigr).
\end{equation}

Coming back to the differential of $\tilde{\Phi}^t$, it now follows from $\tilde{\Phi}^t_{*}(X)=\zeta_X(t)$ and \eqref{eq:Jacobi} that, if $t\in [0,r)$, then $\tilde{\Phi}^t_{*}$ is an isomorphism for each $q\in\tilde{\mathcal{W}}$. This is simply because $g_\lambda,g_\mu>0$ in $[0,r)$. Therefore, by making $\tilde{\mathcal{W}}$ smaller if necessary, we conclude that $\tilde{\mathcal{W}}^t=\tilde{\Phi}^t(\tilde{\mathcal{W}})$ is an equidistant hypersurface to~$\tilde{\mathcal{W}}$  for each $t\in[0,r)$, and $\eta^t$ can be seen as a unit normal vector field along~$\tilde{\mathcal{W}}^t$.

We now determine the extrinsic geometry of the hypersurface $\tilde{\mathcal{W}}^t$.
For each $t\in[0,r)$ it is known that the shape operator $\tilde{\Ss}^t$ of $\tilde{\mathcal{W}}^t$ at $\tilde{\Phi}^t(q)$ with respect to $\eta^t(q)$ is determined by the formula $\tilde{\Ss}^t\tilde{\Phi}^t_{*q}X=-\zeta_X'(t)$ for each $X\in T_q\tilde{\mathcal{W}}$ (again, see \cite[\S8.2]{BCO03}). Before using the explicit expressions of the Jacobi vector fields in terms of the parallel translation obtained above, we define the functions
\begin{equation}\label{eq:functions}
\begin{aligned}
\lambda(t) &{}= \frac{\sqrt{-c}}{2}\tanh\Bigl(\frac{ \sqrt{-c}}{2}  (r-t)\Bigr),\qquad
\mu(t) = \frac{\sqrt{-c}}{2}\coth\Bigl(\frac{ \sqrt{-c}}{2}  (r-t)\Bigr),\\
\alpha(t) &{}=\frac{2}{\sqrt{-c}}  \cosh^3\Bigl( \frac{r \sqrt{-c}}{2}\Bigr)
\sech^3\Bigl(\frac{ \sqrt{-c}}{2}(r-t)\Bigr) \sinh\Bigl(\frac{t \sqrt{-c}}{2}\Bigr),\\
\beta(t) &{}=\cosh^2\Bigl(\frac{ r \sqrt{-c}}{2}\Bigr)
\sech^2\Bigl(\frac{ \sqrt{-c}}{2}(r-t)\Bigr),
\end{aligned}
\end{equation}
which are positive for each $t\in [0,r)$, and the vector fields along $\tilde{\Phi}^t$
\begin{equation}\label{eq:seminull_basis_t}
\begin{aligned}
E_1^t(q)&{}=\beta(t)\mathcal{P}_{E_1(q)}(t),\\[1ex]
E_2^t(q)&{}=-\frac{\alpha(t)^2}{8\beta(t)^3}\mathcal{P}_{E_1(q)}(t)
+\frac{1}{\beta(t)}\mathcal{P}_{E_2(q)}(t)
-\frac{\alpha(t)}{2\beta(t)^2}\mathcal{P}_{E_3(q)}(t),\\
E_3^t(q)&{}=\frac{\alpha(t)}{2\beta(t)}\mathcal{P}_{E_1(q)}(t)+
\mathcal{P}_{E_3(q)}(t).
\end{aligned}
\end{equation}
Now, using~\eqref{eq:Jacobi} and $\tilde{\Ss}^t\tilde{\Phi}^t_{*q}X=-\zeta_X'(t)$, it follows after some calculations that $\tilde{\mathcal{W}}^t$ has principal curvatures $\lambda(t)$ and $\mu(t)$ with algebraic multiplicities $m_\lambda$ and $m_\mu$, and the tangent vectors $E_1^t$, $E_2^t$, $E_3^t$ satisfy~\eqref{eq:seminull_basis} at each point (with $\lambda(t)$ instead of $\lambda$). Moreover, the principal curvature spaces of $\tilde{\mathcal{W}}^t$ are obtained by parallel translation of $T_\lambda$ and $T_\mu$ along the geodesics~$\tilde{\gamma}_q$, that is, $T_{\lambda(t)}=\mathcal{P}_{T_\lambda}(t)$ and $T_{\mu(t)}=\mathcal{P}_{T_\mu}(t)$. In particular, $\tilde{\mathcal{W}}^t$ is isoparametric for all $t\in[0,r)$, and all points of $\tilde{\mathcal{W}}^t$ are of type~III.

Finally, we show that the $S^1$-fiber of $\pi$ is tangent to $\tilde{\mathcal{W}}^t$ for each $t\in[0,r)$. This follows from the fact that the vertical vector field $V$ satisfies
\[
\langle\tilde{\gamma}_q'(0),V_{\tilde{\gamma}_q(0)}\rangle=0\quad
\text{ and }\quad
\frac{d}{dt}\langle\tilde{\gamma}_q',V\rangle
=\langle\tilde{\gamma}_q',\tilde{\nabla}_{\tilde{\gamma}_q'(t)}V\rangle=0,
\]
for all $t$, because $V$ is a Killing vector field (and thus $\tilde{\nabla}V$ is skew-symmetric with respect to the metric).

We can summarize the information obtained about $\tilde{\mathcal{W}}^t$ so far as follows

\begin{proposition}\label{th:wt}
If $t\in[0,r)$, then the $S^1$-fibers of $\pi$ are tangent to the parallel hypersurface $\tilde{\mathcal{W}}^t$, which has constant principal curvatures $\lambda(t)$ and $\mu(t)$ with algebraic multiplicities $m_\lambda$ and $m_\mu$. All points of $\tilde{\mathcal{W}}^t$ are of type~III, $\{E_1^t,E_2^t,E_3^t\}$ are three tangent vector fields satisfying~\eqref{eq:seminull_basis} at each point (with $\lambda(t)$ instead of $\lambda$), and the spaces $T_{\lambda(t)}\ominus\R E_2^t$ and $T_{\mu(t)}$ are obtained by parallel translation of $T_\lambda\ominus \R E_2$ and $T_\mu$ along normal geodesics.
\end{proposition}

Now we focus our attention on $t=r$. Recall from Proposition~\ref{th:typeiii} that if $\lambda=0$, then $\mu$ does not exist and $m_\mu=0$. In general, it follows from~\eqref{eq:Jacobi} that $\ker\tilde{\Phi}_*^r=T_\mu$, and thus, $\tilde{\Phi}^r$ has constant rank $2n-m_\mu$. Hence, making $\tilde{\mathcal{W}}$ smaller if necessary, we deduce that $\tilde{\mathcal{W}}^r$ is an embedded submanifold of $H_1^{2n+1}$ of codimension $m_\mu+1$.

Let $q_r\in\tilde{\mathcal{W}}^r$. The map $\eta^r\colon (\tilde{\Phi}^r)^{-1}(q_r)\to \nu_{q_r}^1 \tilde{\mathcal{W}}^r$, $q\mapsto\eta^r(q)$, from $(\tilde{\Phi}^r)^{-1}(q_r)\subset \tilde{\mathcal{W}}$ to the unit normal space $\nu_{q_r}^1 \tilde{\mathcal{W}}^r$ of $\tilde{\mathcal{W}}^r$ at $q_r$ is differentiable. By~\eqref{eq:Jacobi},
\[
\tilde{\nabla}_X \eta^r=\zeta_X'(r)=
-\frac{\sqrt{-c}}{2}\csch\Bigl(\frac{r\sqrt{-c}}{2}\Bigr)\mathcal{P}_X(t),
\]
for each $X\in T_\mu(q)$ with $q\in(\tilde{\Phi}^r)^{-1}(q_r)$. Since $T_\mu(q)$ is the tangent space of $(\tilde{\Phi}^r)^{-1}(q_r)$ at~$q$, it follows that $\eta^r((\tilde{\Phi}^r)^{-1}(q_r))$ is open in $\nu_{q_r}^1 \tilde{\mathcal{W}}^r$.

As we have seen above, $\langle\tilde{\gamma}'_q(t),V\rangle=0$ for all $t$ and all $q\in\tilde{\mathcal{W}}$. Setting $t=r$ we get $\langle\eta^r,V\rangle=0$ for all $q_r\in\tilde{\mathcal{W}}^r$, and since $\eta^r$ maps $\tilde{\mathcal{W}}$ to an open subset of the unit normal bundle of $\tilde{\mathcal{W}}^r$ we get that $V$ is orthogonal to $\nu\tilde{\mathcal{W}}^r$, and thus tangent to $\tilde{\mathcal{W}}^r$. This implies that $\tilde{\mathcal{W}}^r$ contains locally the $S^1$-fiber of the submersion $\pi\colon H_1^{2n+1}\to \C H^n$.

On the other hand, the tangent space $T_{q_r}\tilde{\mathcal{W}}^r
=\tilde{\Phi}^r_{*q}(T_\lambda(q)\oplus\R E_2(q)\oplus\R E_3(q))$ is, according to~\eqref{eq:Jacobi}, precisely the parallel translation of $T_\lambda(q)\oplus\R E_2(q)\oplus\R E_3(q)$ along the geodesic $\tilde{\gamma}_q$ for $q\in\tilde{\mathcal{W}}$. Again by~\eqref{eq:Jacobi}, $(\nu_{q_r}\tilde{\mathcal{W}}^r)\ominus\R \eta^r(q)$ is obtained by parallel translation of $T_\mu(q)$ along $\tilde{\gamma}_q$.

In order to determine the geometry of the submanifold $\tilde{\mathcal{W}}^r$, we take $q\in\tilde{\mathcal{W}}$ and calculate the shape operator $\tilde{\Ss}^r_{\eta^r(q)}$  of $\tilde{\mathcal{W}}^r$ at $q_r=\tilde{\Phi}^r(q)$ with respect to $\eta^r(q)$. It is known that $\tilde{\Ss}_{\eta^r(q)}^r\tilde{\Phi}^r_{*q}X=-(\zeta_X'(t))^\top$ for each $X\in T_q\tilde{\mathcal{W}}$, where $(\cdot)^\top$ denotes orthogonal projection onto the tangent space $T\tilde{\mathcal{W}}^r$.

Taking this into account, and using~\eqref{eq:functions} and~\eqref{eq:seminull_basis_t} for $t=r$, one can see that $\tilde{\Ss}^r_{\eta^r(q)}$ has exactly one principal curvature $\lambda(r)=0$, and $\{E_1^r(q),E_2^r(q),E_3^r(q)\}$ are vectors satisfying the same relations as in~\eqref{eq:seminull_basis} for $\tilde{\Ss}^r_{\eta^r(q)}$ at~$q_r$ (with $\lambda=0$ in~\eqref{eq:seminull_basis}). The parallel translation of $T_\lambda(q)\ominus\R E_2(q)$ along the normal geodesic $\tilde{\gamma}_q$ is in the kernel of $\tilde{\Ss}^r_{\eta^r(q)}$.

In particular it follows that $(\tilde{\Ss}_{\eta^r(q)}^r)^2\neq 0$ and $(\tilde{\Ss}_{\eta^r(q)}^r)^3=0$ for each $q\in\tilde{\mathcal{W}}$. Since $\eta^r(\tilde{\mathcal{W}})$ is open in the unit normal bundle of $\tilde{\mathcal{W}}^r$, the analiticity of $(\tilde{\Ss}_{\eta}^r)^3$ with respect to $\eta$ implies that $(\tilde{\Ss}_\eta^r)^3= 0$ for any $\eta\in\nu\tilde{\mathcal{W}}^r$.

We summarize these results in the following

\begin{proposition}\label{th:wr}
The submanifold $\tilde{\mathcal{W}}^r$ has codimension $m_\mu+1$ in $H_1^{2n+1}$ and the $S^1$-fibers of $\pi$ are tangent to it. Moreover, if $q_r=\tilde{\Phi}^r(q)$, with $q\in\tilde{\mathcal{W}}$, then $(\nu_{q_r}\tilde{\mathcal{W}}^r)\ominus\R\eta^r(q)$ is obtained by parallel translation of $T_\mu(q)$ along a geodesic normal to $\tilde{\mathcal{W}}$ through $q$. For any $\eta\in\nu_{q_r}^1\tilde{\mathcal{W}}^r$, the shape operator $\tilde{\Ss}^r_\eta$ is $3$-step nilpotent, and its kernel is obtained by parallel translation of $T_\lambda(q)$ along a geodesic normal to $\tilde{\mathcal{W}}$ through $q$.
\end{proposition}

It is worthwhile to emphasize that, although $E_1^r(q)$, $E_2^r(q)$, $E_3^r(q)$ are tangent vectors at $q_r\in\tilde{\mathcal{W}}^r$, these depend on $q\in\tilde{\mathcal{W}}$. The next subsection is devoted to a more thorough study of the geometry of the focal submanifold~$\tilde{\mathcal{W}}^r$.


\subsection{Algebraic study of the focal submanifold}\label{sec:focal}\hfill

Let $q_r\in\tilde{\mathcal{W}}^r$. The main idea in what follows is to prove Proposition~\ref{th:z}, which implies that a certain vector
does not depend on the choice of $q\in(\tilde{\Phi}^r)^{-1}(q_r)$. This vector will be fundamental to determine the geometry of $\pi(\tilde{\mathcal{W}}^r)$, which is the aim of this subsection. We continue using the notation introduced in Section~\ref{sec:jacobi}.

\begin{proposition}\label{th:z}
Let $q_r\in\tilde{\mathcal{W}}^r$. Then, the map
\[
(\tilde{\Phi}^r)^{-1}(q_r)\to T_{q_r}\tilde{\mathcal{W}}^r,\quad q\mapsto -\frac{1}{\langle V_{q_r},E_1^r(q)\rangle}E_1^r(q)-V_{q_r},
\]
is constant in $(\tilde{\Phi}^r)^{-1}(q_r)$.
\end{proposition}

\begin{proof}
Let $q\in(\tilde{\Phi}^r)^{-1}(q_r)$ and let $\zeta_{q_r}\in\nu_{q_r}\tilde{\mathcal{W}}^r\ominus\R\eta^r(q)$ be a unit vector. We calculate $\tilde{\Ss}^r_{\zeta_{q_r}} E_1^r(q)$. Let $\sigma$ be an integral curve of $E_1$ in $\tilde{\mathcal{W}}$ and extend $\zeta_{q_r}$ to a smooth vector field $\zeta$ along $s\mapsto\tilde{\Phi}^r(\sigma(s))$ in such a way that $\langle\zeta_{\tilde{\Phi}^r(\sigma(s))},
\eta^r_{\tilde{\Phi}^r(\sigma(s))}\rangle=0$. Then, there exists a unique vector field $Y\in \Gamma(\sigma^* T_\mu)$ along $\sigma$ tangent to $T_\mu$ such that $\mathcal{P}_{Y_{\sigma(s)}}(r)=\zeta_{\tilde{\Phi}^r(\sigma(s))}$ for all $s$ by Proposition~\ref{th:wr}. We define the geodesic variation $F(s,t)=\exp_{\sigma(s)}(t\xi_{\sigma(s)}^L)$, where $\xi^L$ is the unit normal vector of $\tilde{\mathcal{W}}$ that was fixed at the beginning of Subsection~\ref{sec:jacobi}.
We use Proposition~\ref{th:wt} twice, and Proposition~\ref{th:DE1W} applied to $\tilde{\mathcal{W}}^t$, $t\in[0,r)$, to conclude that $\mathcal{P}_{Y_{\sigma(s)}}(t)\in T_{\mu(t)}(F(s,t))$ and $\tilde{\nabla}_{E_1^t({\sigma(s)})}\mathcal{P}_{Y_{\sigma(s)}}(t)\in T_{\mu(t)}(F(s,t))$. By Proposition~\ref{th:wt} we have that the principal curvature distribution of $\tilde{\mathcal{W}}^t$ associated with $\mu(t)$ at $F(s,t)$ is the parallel translation of $T_{\mu}(\sigma(s))$ along a normal geodesic, that is, $T_{\mu(t)}(F(s,t))=\mathcal{P}_{T_{\mu}(\sigma(s))}(t)$. By continuity we get $\tilde{\nabla}_{E_1^r({\sigma(s)})}\mathcal{P}_{Y_{\sigma(s)}}(r)\in \mathcal{P}_{T_{\mu}(\sigma(s))}(r)$. Combining this with $\zeta_{\tilde{\Phi}^r(\sigma(s))}=\mathcal{P}_{Y_{\sigma(s)}}(r)$ and Proposition~\ref{th:wr} yields $\tilde{\nabla}_{E_1^r({\sigma(s)})}\zeta\in (\nu_{\tilde{\Phi}^r(\sigma(s))}\tilde{\mathcal{W}}^r)\ominus
\R\eta^r_{\sigma(s)}$. Therefore,
\[
\tilde{\Ss}^r_{\zeta_{q_r}}E_1^r(q)=-(\tilde{\nabla}_{E_1^r(q)}\zeta)^\top=0,
\]
as we wanted to calculate.

Since $\zeta_{q_r}\in\nu_{q_r}\tilde{\mathcal{W}}^r\ominus\R\eta^r(q)$ was arbitrary and we already had $\tilde{\Ss}_{\eta^r(q)}E_1^r(q)=0$ by Proposition~\ref{th:wr} and~\eqref{eq:seminull_basis_t}, we conclude that
$\tilde{\Ss}^r_\eta E_1^r(q)=0$, for any $\eta\in\nu_{q_r}\tilde{\mathcal{W}}^r$.
Since $q$ is also arbitrary, we get
\begin{equation}\label{eq:SE1}
\tilde{\Ss}^r_\eta E_1^r(q)=0,\text{ for any $\eta\in\nu_{q_r}\tilde{\mathcal{W}}^r$,
and any $q\in(\tilde{\Phi}^r)^{-1}(q_r)$.}
\end{equation}

Now take another point $\hat{q}\in(\tilde{\Phi}^r)^{-1}(q_r)$. According to Proposition~\ref{th:wr}, we can write $E_1^r(\hat{q})=a_1 E_1^r(q)+a_2 E_2^r(q)+a_3 E_3^r(q)+u$, with $a_i\in\R$, and $u\in(\ker\tilde{\Ss}^r_{\eta^r(q)})\ominus\R E_2^r(q)$. By~\eqref{eq:SE1} we have
\[
0=\tilde{\Ss}^r_{\eta^r(q)}E_1^r(\hat{q})=
a_2 E_3^r(q)+a_3 E_1^r(q).
\]
Thus, $a_2=a_3=0$. On the other hand, since $E_1^r(\hat{q})$ is a null vector, we also obtain
$0=\langle E_1^r(\hat{q}),E_1^r(\hat{q})\rangle=\langle u,u\rangle$, and as $u$ is spacelike, we get $u=0$. Thus, $E_1^r(\hat{q})=a_1 E_1^r(q)$, which easily implies the result.
\end{proof}

Recall that the submanifold $\tilde{\mathcal{W}}^t$ contains locally the $S^1$-fiber of the semi-Riemannian submersion $\pi\colon H_1^{2n+1}\to\C H^n$ as we have seen in propositions~\ref{th:wt} and~\ref{th:wr}. If we denote $\mathcal{W}^t=\pi(\tilde{\mathcal{W}}^t)$, $t\in[0,r]$, and consider the map $\Phi^t\colon\mathcal{W}\to\C H^n$, $p\mapsto\Phi^t(p)=\exp_p(t\xi_p)$, then it follows that $\Phi^t(\pi(\tilde{\mathcal{W}}))=\pi(\tilde{\Phi}^t(\tilde{\mathcal{W}}))$, that is, $\Phi^t(\mathcal{W})=\pi(\tilde{\mathcal{W}}^t)$, or in other words, the Hopf map commutes with the parallel displacement map.

Coming back to the study of the geometry of the submanifold $\mathcal{W}^r$, we write $V_{q_r}=s_1(q)E_1^r(q)+s_2(q)E_2^r(q)+s_3(q)E_3^r(q)+u_q$, for $s_i(q)\in\R$ and $u_q\in T_{q_r}\tilde{\mathcal{W}}^r\ominus(\R E_1^r(q)\oplus\R E_2^r(q)\oplus\R E_3^r(q))=(\ker\tilde{\Ss}^r_{\eta^r(q)})\ominus\R E_2^r(q)$. Arguing as in~\eqref{eq:modified_seminull}, we can assume $u_q=0$. Note that the procedure at the beginning of the proof of Proposition~\ref{th:typeiii} which leads to~\eqref{eq:modified_seminull} does not change the vector $E_1^r(q)$. Thus, $-1=\langle V,V\rangle =2s_1s_2+s_3^2$, which immediately implies $s_1,s_2\neq 0$. We can assume, changing the signs of $E_1(q)$, $E_2(q)$ and $E_3(q)$, that $s_2 >0$.

If $\xi$ is now a unit normal vector field of $\mathcal{W}^r$, we write $J\xi=P\xi+F\xi$, where $P\xi$ is the orthogonal projection of $J\xi$ onto~$T{\mathcal{W}}^r$ and $F\xi$ is the orthogonal projection of $J\xi$ onto~$\nu{\mathcal{W}}^r$. We also write $J\xi^L=P\xi^L+F\xi^L$ for $\tilde{\mathcal{W}}^r$, accordingly. Notice that $P(\xi^L)=(P\xi)^L$ and $F(\xi^L)=(F\xi)^L$. From~\eqref{sub2} we get $\tilde{\Ss}^r_{\xi^L} V =- (\tilde{\nabla}_{V} \xi^L)^{\top} = - ({\sqrt{-c}}/{2})P \xi^L$. Hence, taking $\xi\in\Gamma(\nu\mathcal{W}^r)$ such that $\xi_{q_r}^L=\eta^r(q)$ we get
\[
0 = -\frac{\sqrt{-c}}{2}\langle P\xi^L,V\rangle=\langle \tilde{\Ss}^r_{\eta^r} V , V \rangle
= \langle s_2 E_3^r(q)+s_3 E_1^r(q),V\rangle
= 2 s_2 s_3,
\]
which implies $s_3 =0$. We may also write
\begin{equation}\label{eq:Jeta}
J \eta^r =-\frac{2}{\sqrt{-c}}\tilde{\Ss}^r_{\eta^r} V + F\eta^r=-\frac{2s_2}{\sqrt{-c}}\,E_3^r(q)+F\eta^r.
\end{equation}
Thus, $1 = \langle J\eta^r, J\eta^r \rangle = -({4}/{c})s^2_2 + \langle F\eta^r,F\eta^r\rangle$, and consequently we can choose a real number $\varphi(q) \in (0, {\pi}/{2}]$, such that
\begin{align*}
s_2(q) &{}= \frac{\sqrt{-c}}{2} \sin(\varphi(q)),
&\langle F\eta^r(q),F\eta^r(q)\rangle &{}= \cos^2(\varphi(q)).
\end{align*}

If $\Ss^r_\xi$ denotes the shape operator of $\mathcal{W}^r$ with respect to $\xi\in\Gamma(\nu\mathcal{W}^r)$, then~\eqref{sub1} implies
\[
\tilde{\Ss}^r_{\xi^L}X^L= (\Ss^r_{\xi}X)^L+ \frac{\sqrt{-c}}{2}\langle J\xi^L,X^L\rangle V\
\text{ and }\
\Ss^r_\xi X=\pi_{*}\tilde{\Ss}^r_{\xi^L}X^L,\ \text{ for each $X\in T\mathcal{W}^r$}.
\]
The vectors in $(\ker\tilde{\Ss}^r_{\eta^r(q)})\ominus\R E_2^r(q)$ are orthogonal to $J\eta^r(q)$ and $V_{q_r}$ by~\eqref{eq:Jeta}, and by the previous equation, project bijectively onto $\ker\Ss^r_{\pi_*\eta^r(q)}$. For dimension reasons, there are only two eigenvectors left to determine $\Ss^r_{\pi_*\eta^r(q)}$ completely.

In view of Proposition~\ref{th:z} we can define
\[
Z_{\pi(q_r)}=\pi_*\Bigl(-\frac{1}{\langle V_{q_r},E_1^r(q)\rangle}E_1^r(q)-V_{q_r}\Bigr)
=-\frac{1}{\langle V_{q_r},E_1^r(q)\rangle}\pi_*E_1^r(q),
\text{ for $q\in(\tilde{\Phi}^r)^{-1}(q_r)$}.
\]
Note that this vector field is smooth because $E_1^t$ is smooth along the map $\tilde{\Phi}^t$ by the smooth dependence on the initial conditions of solutions to an ordinary differential equation.
For the subsequent calculations, we consider $\xi\in \nu_{\pi(q_r)}\mathcal{W}^r$ such that its lift to $\nu_{q_r}\tilde{\mathcal{W}}^r$ satisfies $\xi^L=\eta^r(q)$. Thus we can write $P\xi^L=P\eta^r$. We have
\begin{align*}
Z^L_{q_r}&{}=-\frac{1}{\langle V,E_1^r(q)\rangle}E_1^r(q)-V_{q_r},
&P\xi^L&{}=-\sin(\varphi(q))E_3^r(q).
\end{align*}
These two vectors are tangent to $\tilde{\mathcal{W}}^r$ and orthogonal to $V$. Thus they are mapped isometrically to $Z$ and $P\xi$ respectively; in particular, $\lVert P\xi\rVert=\sin(\varphi(q))$. Furthermore, by~\eqref{eq:Jeta} we also have $\langle Z^L_{q_r},J\eta^r(q)\rangle=0$ for any $q\in(\tilde{\Phi}^r)^{-1}(q_r)$. Since $\eta^r((\tilde{\Phi}^r)^{-1}(q_r))$ is open in $\nu_{q_r}^1 \tilde{\mathcal{W}}^r$, we deduce that $Z^L$ is orthogonal to $J\nu\tilde{\mathcal{W}}^r$, and hence, $Z$ is orthogonal to $J\nu\mathcal{W}^r$. Thus, we have that $T\mathcal{W}^r\ominus P\nu\mathcal{W}^r$ is the maximal complex distribution of $T\mathcal{W}^r$ and $Z$ is tangent to~it.

Using the above formulas we obtain
\begin{align*}
\Ss^r_{\xi}Z &{}= \pi_{*q_r} \tilde{\Ss}^r_{\xi^L}Z^L=  -\pi_{*q_r} \tilde{\Ss}^r_{\xi^L}V =\frac{\sqrt{-c}}{2}\pi_{*q_r} P\xi^L
=\frac{\sqrt{-c}}{2}P\xi,\\
\Ss^r_{\xi} P\xi &{}= \pi_{*q_r} \tilde{\Ss}^r_{\xi^L}P\xi^L
= -\sin(\varphi(q))\pi_{*q_r}E_1^r(q)
=\sin(\varphi(q))s_2(q)Z=\frac{\sqrt{-c}}{2}\sin^2(\varphi(q))Z.
\end{align*}
Therefore, by analiticity of $\Ss^r_\xi$ with respect to $\xi$,
\[
\langle \II (Z, P\xi), \eta \rangle = \langle \Ss^r_{\eta} Z, P \xi \rangle
= \frac{\sqrt{-c}}{2}\langle P \eta , P \xi \rangle
= -\frac{\sqrt{-c}}{2} \langle \eta, JP\xi \rangle,
\]
for all $\xi$, $\eta\in\nu\mathcal{W}^r$. We can summarize the results obtained so far in

\begin{proposition}\label{th:focal}
The vector field $Z$ is tangent to the maximal complex distribution of $T\mathcal{W}^r$. The second fundamental form of $\mathcal{W}^r$ is determined by the trivial symmetric bilinear extension of
\[
2\,\II(Z, P \xi) = - \sqrt{-c}\, (JP\xi)^{\perp},
\]
for any $\xi\in\nu\mathcal{W}^r$.
\end{proposition}

\section{Rigidity of the focal submanifold}\label{sec:rigidity}

In this section we prove that a submanifold of $\C H^n$ under the conditions of Proposition~\ref{th:focal} is congruent to an open part of a submanifold $W_{\g{w}}$ defined in Subsection~\ref{sec:examples:W}. The precise statement is as follows.

\begin{theorem}\label{th:rigidity}
Let $M$ be a connected $(2n-k)$-dimensional submanifold of $\C H^n$, $n\geq 2$. Assume
that there exists a smooth unit vector field $Z$ tangent to the maximal
complex distribution of $M$ such that the second fundamental form
$\II$ of $M$ is given by the trivial symmetric bilinear extension of
\begin{equation}\label{eq:II}
2\,\II(Z,P\xi)=-\sqrt{-c}\,(JP\xi)^\perp,
\end{equation}
for $\xi\in \nu M$,
where $P\xi$ is the tangential component of $J\xi$, and $(\cdot)^\perp$ denotes orthogonal projection onto the normal space $\nu M$.
Then, a point $o\in M$ and $B_o=-J Z_o$ determine an Iwasawa decomposition $\g{su}(1,n)=\g{k}\oplus\g{a}\oplus\g{g}_\alpha\oplus\g{g}_{2\alpha}$ of the Lie algebra of the isometry group of $\C H^n$, such that
$M$ is congruent to
an open part of the minimal submanifold $W_{\g{w}}$, where $\g{w}=T_o M\ominus(\R B_o\oplus\R Z_o)\subset\g{g}_\alpha$.
\end{theorem}

Before beginning the proof, we start with a more geometric construction of the submanifolds $W_{\g{w}}$. This will make use of several Lie theoretic concepts that were introduced in Subsection~\ref{sec:examples:W}. See~\cite{BTV95} for further details.

\begin{proposition}\label{th:construction}
Let $k \in \{1,\ldots,n-1\}$, fix a totally geodesic $\C H^{n-k}$ in $\C H^n$ and points $o \in \C H^{n-k}$ and $x \in
\C H^{n-k}(\infty)$. Let $KAN$ be the Iwasawa decomposition of
$SU(1,n)$ with respect to $o$ and $x$, and let $\hat{H}$ be the
subgroup of $AN$ that acts simply transitively on $\C H^{n-k}$.
Now, let $\g{v}$ be a proper subspace of $\nu_o\C H^{n-k}$ such that $\g{v}\cap J\g{v} =0$. Left translation of $\g{v}$ by $\hat{H}$ to all points of
$\C H^{n-k}$ determines a subbundle $\mathfrak{V}$ of the normal
bundle $\nu \C H^{n-k}$. At each point $p \in \C H^{n-k}$ attach the
horocycles determined by $x$ and the linear lines in $\mathfrak{V}_p$. The resulting subset $M$ of $\C H^n$ is
congruent to the submanifold $W_{\g{w}}$, where $\g{w}=(\hat{\g{h}}\ominus(\g{a}\oplus\g{g}_{2\alpha}))\oplus\g{v}
\subset\g{g}_\alpha$.
\end{proposition}

\begin{proof}
Let $W_{\g{w}}$ be the minimal submanifold of $\C H^n$ constructed
from the Iwasawa decomposition $KAN$ associated with $o$ and $x$ and
from $\g{w}=(\hat{\g{h}}\ominus(\g{a}\oplus\g{g}_{2\alpha}))\oplus\g{v}$, as described in Subsection~\ref{sec:examples:W}. We recall that $T_o\C H^n$ is now identified with $\g{a}\oplus\g{n}$ and we denote by $\g{w}^\perp=\g{g}_\alpha\ominus\g{w}$ the orthogonal complement of $\g{w}$ in $\g{g}_\alpha$. We have that the Lie algebra of $\hat{H}$ is $\hat{\g{h}}=\g{s}_{\g{w}}\ominus P\g{w}^\perp$, with $\g{s}_\g{w}=\g{a}\oplus\g{w}\oplus\g{g}_{2\alpha}$, and where, as usual, $P\xi$ denotes the orthogonal projection of $J\xi$ on $\g{w}$ for each $\xi\in\g{w}^\perp$. Since $\g{v}\cap J\g{v}=0$, we have that $\hat{\g{h}}$ is the maximal complex subspace of $\g{s}_\g{w}$.

Let $p \in W_{\g{w}}$. By definition, there exists an isometry $s\in S_{\g{w}}$ with $p = s(o)$.
There is a unique vector $X$ in the Lie algebra $\g{s}_{\g{w}}$ of $S_{\g{w}}$ such
that $s = \Exp_{\g{a}\oplus\g{n}}(X)$. We can write $X = aB + U + W + xZ$
with $a$, $x \in \R$, $U \in \hat{\g{h}} \ominus (\g{a}\oplus\g{g}_{2\alpha})$, and $W \in \g{v}$. Since $U$ and $W$ are complex-orthogonal, we get $[U,W]=0$ by~\eqref{eq:brackets} from Subsection~\ref{sec:examples}. Using this notation we can define the elements
$g=\Exp_{\g{a}\oplus\g{n}}(\rho(a/2)W)$
and $h=\Exp_{\g{a}\oplus\g{n}}(aB+U+xZ)\in \hat{H}$. Using~\eqref{eq:product_AN} we obtain,
\[
gh
=\Exp_{\g{a}\oplus\g{n}}\Bigl(\rho\Bigl(\frac{a}{2}\Bigr)W\Bigr)\cdot
\Exp_{\g{a}\oplus\g{n}}(aB+U+xZ)=\Exp_{\g{a}\oplus\g{n}}(aB+U+W+xZ)=s.
\]
By construction, $h(o) \in \C H^{n-k}$, and $s(o) = g(h(o))$ is in the
horocycle through $h(o)$, tangent to $\R W$, and with center $x$ at infinity. Hence, $p=s(o)\in M$ and we conclude that $W_{\g{w}} \subset M$.

Now we prove the converse. Let $\sigma$ be a horocycle such that $\sigma(0)=o$, $\sigma'(0)=U\in\g{v}$, $\lVert U\rVert=1$, and $2(\bar{\nabla}_{\sigma'}\sigma')(0)=\sqrt{-c}\,B$. We show that $\sigma$ is contained in $W_{\g{w}}$. First, using~\eqref{eq:Levi-Civita}, we get $\bar{\nabla}_B B=\bar{\nabla}_B U=0$, $2\bar{\nabla}_U B=-\sqrt{-c}\,U$ and $2\bar{\nabla}_U U=\sqrt{-c}\,B$. Hence, it follows that the distribution generated by $B$ and $U$ is autoparallel and its integral submanifolds are totally geodesic real hyperbolic spaces $\R H^2$ of curvature $c/4$.
Now, we denote by $\tau$ an integral curve of the left-invariant vector field $U$ such that $\tau(0)=o$. Using~\eqref{eq:Levi-Civita} we get $\bar{\nabla}_U\bar{\nabla}_U U+\langle\bar{\nabla}_U U,\bar{\nabla}_U U\rangle U=0$. Thus, $\tau$ is a cycle in a totally geodesic $\R H^2$ of curvature $c/4$, and since $2(\bar{\nabla}_{\tau'}\tau')(0)=\sqrt{-c} B$, it follows that $\tau$ is a horocycle determined by $o$, $U$ and the point at infinity $x$. By uniqueness of solutions to ordinary differential equations we get $\tau=\sigma$, and thus $\sigma$ is contained in $W_{\g{w}}$.

If $\sigma$ is an arbitrary horocycle determined by initial conditions $p\in \C H^{n-k}$, $U_p\in\g{V}_p$ and $\sqrt{-c}\,B_p/2$, then there is a unique $h\in\hat{H}$ such that $p=h(o)$. Since $h$ is an isometry of $\C H^n$, it is easy to see that $h^{-1}\circ\sigma$ satisfies the conditions of horocycle in the previous paragraph. Hence, $h^{-1}\circ\sigma$ is contained in $W_{\g{w}}$, from where it follows that $\sigma$ is contained in $W_{\g{w}}$ because $h\in\hat{H}\subset S_{\g{w}}$. This shows that $M\subset W_{\g{w}}$ and finishes the proof of the proposition.
\end{proof}

The rest of this Section is devoted to the proof of the rigidity result given by Theorem~\ref{th:rigidity}. In what follows, $M$ will denote a submanifold of $\C H^n$ under the assumptions of Theorem~\ref{th:rigidity}.

\subsection{The structure of the normal bundle}\label{sec:normal}\hfill

For $\xi\in\nu M$ recall that $J\xi=P\xi+F\xi$, where $P\xi$ and $F\xi$ denote the orthogonal projections of $J\xi$ onto $TM$ and $\nu M$ respectively. The maps $P\colon\nu M\to T M$ and $F\colon\nu M\to\nu M$ are vector bundle homomorphisms. We will use some of their properties in the rest of the paper. We start with

\begin{lemma}\label{th:DF}
The endomorphism $F$ of $\nu M$ is parallel with respect to the normal connection of $M$, that is, $\nabla^\perp F=0$.
\end{lemma}

\begin{proof}
Let $\xi$, $\eta\in\Gamma(\nu M)$ and $X\in\Gamma(T M)$. Using~\eqref{eq:II} we get
\[
\langle\II(Z,P\xi),\eta\rangle
=-\frac{\sqrt{-c}}{2}\langle JP\xi,\eta\rangle
=\frac{\sqrt{-c}}{2}\langle P\xi,P\eta\rangle
=-\frac{\sqrt{-c}}{2}\langle \xi,JP\eta\rangle
=\langle\II(Z,P\eta),\xi\rangle.
\]
This relation yields $\langle\II(X,P\xi),\eta\rangle=\langle\II(X,P\eta),\xi\rangle$ using the fact that $\II$ is obtained by the trivial symmetric bilinear extension of~\eqref{eq:II}. Since $\C H^n$ is K\"{a}hler,
\begin{align*}
\langle\nabla_X^\perp F\xi,\eta\rangle
&{}=\langle\bar{\nabla}_X J\xi,\eta\rangle-\langle\bar{\nabla}_X P\xi,\eta\rangle
=-\langle\bar{\nabla}_X \xi,P\eta+F\eta\rangle-\langle\II(X,P\xi),\eta\rangle\\
&{}=\langle\II(X,P\eta),\xi\rangle
-\langle{\nabla}^\perp_X \xi,F\eta\rangle-\langle\II(X,P\xi),\eta\rangle
=-\langle{\nabla}^\perp_X \xi,J\eta\rangle=\langle F{\nabla}^\perp_X \xi,\eta\rangle.
\end{align*}
Hence, $(\nabla_X^\perp F)\xi=\nabla_X^\perp F\xi-F\nabla_X^\perp\xi=0$, as we wanted to show.
\end{proof}

For each $p\in M$, the normal space $\nu_p M$ is a real vector subspace of the complex vector space $T_p\C H^n$. According to Subsection~\ref{sec:Kahler_angles}, $\nu_p M$ has a decomposition as a sum of subspaces of constant K\"{a}hler angle. These angles are called the principal K\"{a}hler angles of $\nu_p M$. We show that they do not depend on $p\in M$.

\begin{proposition}\label{th:const_Kahler_angles}
The principal K\"{a}hler angles of $\nu M$ and their multiplicities are constant along $M$.
\end{proposition}

\begin{proof}
Let $p$, $q\in M$ be two arbitrary points, and let $\sigma\colon [0,1]\to M$ be a smooth curve in $M$ such that $\sigma(0)=p$ and $\sigma(1)=q$. We take a basis $\{\xi_1,\dots,\xi_k\}$ of principal K\"{a}hler vectors, that is, an orthonormal basis of $\nu_p M$ such that $\langle F\xi_i(p),F\xi_j(p)\rangle=\cos^2(\varphi_i(p))\delta_{ij}$, for $i$, $j\in\{1,\dots,k\}$ (see Subsection~\ref{sec:Kahler_angles}). We extend this basis to a $\nabla^\perp$-parallel orthonormal basis $\{\xi_1(t),\dots,\xi_k(t)\}$ of smooth vector fields along~$\sigma$. Since $F$ is parallel by Lemma~\ref{th:DF}, it follows that $\langle F\xi_i,F\xi_j\rangle$ is constant along~$\sigma$. Therefore, $\{\xi_1(1),\dots,\xi_k(1)\}$ is also a basis of principal K\"{a}hler vectors of $\nu_q M$, and it follows that the principal K\"{a}hler angles and their multiplicities of $\nu M$ at $p$ and $q$ coincide.
\end{proof}

Let $\Phi$ be the set of constant principal K\"{a}hler angles of $\nu M$. We write
$\nu_p M=\oplus_{\varphi\in\Phi}\g{W}_\varphi^\perp(p)$ as in Subsection~\ref{sec:Kahler_angles}, where each $\g{W}_\varphi^\perp(p)$ has constant K\"{a}hler angle $\varphi$. Since the principal K\"{a}hler angles are constant, $\g{W}_\varphi^\perp$ is a smooth vector subbundle of $\nu M$. If $\g{W}_0^\perp$ is nonzero we can simplify matters because there is a reduction of codimension.

\begin{proposition}\label{th:reduction}
If $\g{W}_0^\perp\neq 0$ there exists a totally geodesic $\C H^k$ in $\C H^n$ containing $M$ where $0$ is no longer principal K\"{a}hler angle of $M$ in $\C H^k$, the normal bundle of $M$ is obtained by inclusion, and the second fundamental form is obtained by restriction.
\end{proposition}

\begin{proof}
We first show that each distribution $\g{W}_\varphi^\perp$ is parallel with respect to the normal connection.
Let $\varphi\in\Phi$, $\xi\in\Gamma(\g{W}_\varphi^\perp)$ and $X\in\Gamma(TM)$. As we argued in Subsection~\ref{sec:Kahler_angles}, we have $F^2\xi=-\cos^2(\varphi)\xi$. Since $\nabla^\perp F=0$ by Lemma~\ref{th:DF}, we get
\[
F^2\nabla_X^\perp\xi=\nabla_X^\perp F^2\xi
=\nabla_X^\perp(-\cos^2(\varphi)\xi)
=-\cos^2(\varphi)\nabla_X^\perp\xi,
\]
and again from the results in Subsection~\ref{sec:Kahler_angles} it follows that $\nabla_X^\perp\xi\in\Gamma(\g{W}_\varphi^\perp)$. Therefore
\begin{equation}\label{eq:DWvarphi}
\nabla^\perp\g{W}_\varphi^\perp\subset\g{W}_\varphi^\perp\quad
\text{ for each $\varphi\in\Phi$}.
\end{equation}

Recall from Subsection~\ref{sec:Kahler_angles} that we can decompose $T M=\g{W}_0\oplus
(\oplus_{\varphi\in\Phi\setminus\{0\}}\g{W}_\varphi)$ with $\C\g{W}_\varphi^\perp=\g{W}_\varphi^\perp\oplus\g{W}_\varphi$ and $\dim\g{W}_\varphi^\perp=\dim\g{W}_\varphi$ for all $\varphi\in\Phi\setminus\{0\}$.
Now we consider the bundle
\[
\mathcal{F}=T M\oplus
\Biggl(\bigoplus_{\varphi\in\Phi\setminus\{0\}}\g{W}_\varphi^\perp\Biggr)=
\g{W}_0\oplus\Biggl(\bigoplus_{\varphi\in\setminus\{0\}}
\C\g{W}_\varphi^\perp\Biggr)
\]
along $M$. Then, $\mathcal{F}$ is a complex vector bundle and, at a point $p\in M$, $\mathcal{F}_p$ is the tangent space of a totally geodesic complex hyperbolic space $\C H^{n-m_0^\perp}$, $m_0^\perp=\dim_\C\g{W}_0^\perp$, in $\C H^n$. Using~\eqref{eq:II} and~\eqref{eq:DWvarphi} we get $\bar{\nabla}_X \phi=\nabla_X^\mathcal{F}\phi$ for each $\phi\in\Gamma(\mathcal{F})$ and where $\nabla^\mathcal{F}$ denotes the connection on $\mathcal{F}$ induced from $\bar{\nabla}$. Hence, by~\cite[Theorem~1 (with $h=0$ in the notation of this paper)]{R80} we conclude that $M$ is contained in the totally geodesic $\C H^{n-m_0^\perp}$ mentioned above.
\end{proof}

In other words, what Proposition~\ref{th:reduction} states is that we can, and we will, assume from now on that $\g{W}_0^\perp=0$. Otherwise, we just take a smaller complex hyperbolic space where this condition is fulfilled.

\subsection{Proof of Theorem~\ref{th:rigidity}}\label{sec:rigidity:proof}\hfill

In order to prove Theorem~\ref{th:rigidity} we use the construction of
$W_{\g{w}}$ as described in Proposition~\ref{th:construction}. Part of the proof goes along the lines of the rigidity result in~\cite{BD09}, although the argument here is more involved.

As we have just seen in Subsection~\ref{sec:normal}, we may assume that the normal bundle $\nu M$ does not contain a nonzero complex subbundle.
We decompose the tangent bundle $TM$ of $M$ orthogonally into $TM =
\g{C}\oplus\g{D}$, where $\g{C}$ is the maximal
complex subbundle of $TM$. Thus, $\g{D}\cap J\g{D}=0$. For each $\xi \in \Gamma(\nu M)$ we
have $J\xi = P\xi + F\xi$, where $P\xi
\in \Gamma(\g{D})$ and $F\xi \in \Gamma(\nu M)$. Since $\g{D}=P\nu M$,  then we argued in Subsection~\ref{sec:Kahler_angles} that $\g{D}$ has the same K\"{a}hler angles, with the same multiplicities as $\nu M$ (note that $0$ is not a K\"{a}hler angle of $\nu M$ by the assumption we have made after Subsection~\ref{sec:normal}). Since the principal K\"{a}hler angles are never $0$, it follows that $P\colon\nu M\to\g{D}$ is an isomorphism of vector bundles.

\begin{lemma}\label{th:complex}
The distribution $\g{C}$ is autoparallel and each integral submanifold is an open part of a totally geodesic complex hyperbolic space $\C H^{n-k}$ in $\C H^n$.
\end{lemma}

\begin{proof}
For all $U,V \in \Gamma(\g{C})$ and $\xi \in \Gamma(\nu M)$ we have,
using~\eqref{eq:II} and $\bar{\nabla} J=0$,
\[
\langle\bar{\nabla}_U V,\xi\rangle=\langle\II(U,V),\xi\rangle = 0,
\quad\text{ and }\quad
\langle\bar{\nabla}_UV,J\xi\rangle = -\langle J\bar{\nabla}_UV,\xi\rangle
= - \langle \II(U,JV),\xi \rangle = 0.
\]
Thus $\g{C}$ is autoparallel and as $\g{C}$ is a complex subbundle of complex rank $n-k$, each of
its integral manifolds is an open part of a totally geodesic
$\C H^{n-k}$ in $\C H^{n}$.
\end{proof}

From now on we fix $o \in M$ and let $\mathcal{L}_o$ be the leaf of $\g{C}$ through
$o$, which is an open part of a totally geodesic $\C H^{n-k}$ in
$\C H^{n}$ by Lemma~\ref{th:complex}. We have

\begin{lemma}\label{th:normal}
If $\gamma\colon I \to\mathcal{L}_o$ is a curve with $\gamma(0) =
o$ then the normal spaces of $M$ along $\gamma$ are
uniquely determined by the differential equation
\begin{equation}\label{eq:diffeq}
2\bar{\nabla}_{\gamma'}\eta +\sqrt{-c}\langle \gamma' , Z \rangle J\eta = 0
\end{equation}
for $\eta\in\Gamma(\gamma^*\nu\mathcal{L}_o)$, where $\gamma^*\nu\mathcal{L}_o$ is the bundle of vectors along $\gamma$ that are orthogonal to $\mathcal{L}_o$.
\end{lemma}

\begin{proof}
Let $X\in\Gamma(TM)$ and
$\xi\in\Gamma(\nu M)$. Using~\eqref{eq:II} we get
\[
-\langle
\bar{\nabla}_{{\gamma'}}\xi,X \rangle = \langle
\II({\gamma'},X),\xi \rangle = \langle {\gamma'},Z \rangle
\frac{\langle X,P\xi \rangle}{\langle
P\xi,P\xi\rangle} \langle \II(Z,P\xi),\xi \rangle
= \frac{\sqrt{-c}}{2}\langle {\gamma'},Z \rangle \langle
P\xi,X \rangle,
\]
which implies
\begin{equation}\label{eq:diffeq2}
\bar{\nabla}_{{\gamma'}}\xi = -\frac{\sqrt{-c}}{2}\langle {\gamma'},Z
\rangle P\xi + \nabla_{{\gamma'}}^\perp \xi,
\end{equation}
where $\nabla^\perp$ is the normal connection of $M$. Now, we take a vector field $X$
along $\gamma$ with $X_0\in\nu_o M$ and satisfying
\eqref{eq:diffeq}. We write $X=U+J\eta+\xi$, where we have
$U\in\Gamma(\gamma^*\g{C})$, $\xi$, $\eta\in\Gamma(\gamma^*\nu M)$ and
$U_0=\eta_0=0$. Using~\eqref{eq:diffeq2} and taking into account that $\bar{\nabla} J=0$, we
obtain
\begin{align*}
0
&{}= 2\bar{\nabla}_{\gamma'}X+\sqrt{-c}\langle\gamma',Z\rangle JX\\
&{}= 2\bar{\nabla}_{\gamma'}U
+2J\bar{\nabla}_{\gamma'}\eta+2\bar{\nabla}_{\gamma'}\xi
+\sqrt{-c}\langle\gamma',Z\rangle JU
+\sqrt{-c}\langle\gamma',Z\rangle J^2\eta
+\sqrt{-c}\langle\gamma',Z\rangle J\xi\\
&{}= 2\bar{\nabla}_{\gamma'}U+\sqrt{-c}\langle\gamma',Z\rangle JU
+P\left(2\nabla_{\gamma'}^\perp\eta
    +\sqrt{-c}\langle\gamma',Z\rangle F\eta\right)\\
&\phantom{{}={}}
+2\nabla_{\gamma'}^\perp\xi+\sqrt{-c}\langle\gamma',Z\rangle F\xi
+F\left(2\nabla_{\gamma'}^\perp\eta
    +\sqrt{-c}\langle\gamma',Z\rangle F\eta\right).
\end{align*}
We have that $2\bar{\nabla}_{\gamma'}U+\sqrt{-c}\langle\gamma',Z\rangle JU$
is tangent to $\g{C}$ since $\g{C}$ is a complex autoparallel
distribution. Thus, it follows that
$2\bar{\nabla}_{\gamma'}U+\sqrt{-c}\langle\gamma',Z\rangle JU=0$. Since
$U_0=0$, the uniqueness of solutions to ordinary differential
equations implies $U_{t}=0$ for all $t$, and thus $X\in\Gamma(\gamma^*\nu\mathcal{L}_o)$. Similarly, the component
tangent to $P\nu M$ in the previous equation yields $2\nabla_{{\gamma'}}^\perp \eta +
\sqrt{-c}\langle {\gamma'},Z \rangle F\eta = 0$ and since
$\eta_{0} = 0$ we have $\eta_{t}=0$ for any~$t$ by uniqueness of solution.
Hence, $X_{t}\in \nu_{\gamma(t)}M$ for all~$t$, which proves
our assertion.
\end{proof}

We define $B=-JZ$.

The point $o\in M$ and the tangent vector $B_o$ uniquely determine a point at infinity $x\in\C H^n(\infty)$ and thus, a corresponding Iwasawa decomposition $\g{k}\oplus\g{a}\oplus\g{n}=\g{k}\oplus\g{a}\oplus\g{g}_\alpha\oplus
\g{g}_{2\alpha}$ of the isometry group of $\C H^n$, where $\g{a}=\R B_o$ and $\g{g}_{2\alpha}=\R Z_o$. We define the subspace $\g{w}=T_o M\ominus(\R B_o\oplus\R Z_o)\subset\g{g}_\alpha$ and consider the submanifold $W_{\g{w}}$ defined by this Iwasawa decomposition and $\g{w}$. As we have already seen, the integral submanifold $\mathcal{L}_o$ is an open part of a totally geodesic $\C H^{n-k}$ contained in $\C H^n$ that is tangent to the maximal complex distribution of $W_{\g{w}}$ at $o$. Since by Lemma~\ref{th:normal} the normal bundle is uniquely determined by the ordinary differential equation~\eqref{eq:diffeq}, and both $M$ and $W_{\g{w}}$ satisfy the hypotheses of Theorem~\ref{th:rigidity}, it follows that $\nu_p M=\nu_p W_{\g{w}}$ for each $p\in\mathcal{L}_o$. As a consequence, $\nu_p M$ is obtained by left translation of $\nu_o M$ by the subgroup of $AN$ that acts simply transitively on~$\mathcal{L}_o$. In view of Proposition~\ref{th:construction} it only remains to prove that for each $p\in\mathcal{L}_o$ the horocycles determined by the point at infinity $x$ and the lines of $P\nu_p M$ are locally contained in $M$.

Before continuing our argument we need to calculate certain covariant derivatives of some vector fields.

\begin{lemma}
Let $X\in\Gamma(TM\ominus\R B)$ and $\xi\in\Gamma(\nu M)$. Then
\begin{align}
\bar{\nabla}_XB
&{}=-\frac{\sqrt{-c}}{2}X-\frac{\sqrt{-c}}{2}\langle X,Z\rangle Z,\label{eq:DXB}\\[0.5ex]
\bar{\nabla}_BP\xi
&{}=P\nabla_B^\perp\xi,\label{eq:DBPxi}\\
\bar{\nabla}_{P\xi}P\xi
&{}=\frac{\sqrt{-c}}{2}\langle P\xi,P\xi\rangle B
+P\nabla_{P\xi}^\perp\xi.\label{eq:DPxiPxi}
\end{align}
\end{lemma}

\begin{proof}
Let $\eta\in\Gamma(\nu M)$ be a local unit vector field. Using
\eqref{eq:II} we obtain
$\langle\bar{\nabla}_X B,\eta\rangle=\langle\II(X,B),\eta\rangle=0$. Moreover, $\langle\bar{\nabla}_XB,B\rangle=0$. Next,~\eqref{eq:II} yields
\begin{equation}\label{eq:DXBPxi}
\begin{aligned}
2\langle\bar{\nabla}_XB,P\eta\rangle
&{}= -2\langle\bar{\nabla}_{X}JZ,J\eta-F\eta\rangle
     = -2\langle\II(X,Z),\eta\rangle
            -2\langle\II(X,B),F\eta\rangle\\
&{}= -2\langle X,P\eta\rangle
        \langle\II(P\eta,Z),\eta\rangle/\langle P\eta,P\eta\rangle
    =-\sqrt{-c}\langle X,P\eta\rangle.
\end{aligned}
\end{equation}
Now, let $Y\in\Gamma(\g{C}\ominus\R B)$ and assume that
$X\in\Gamma(\g{C}\ominus\R B)$. For any $\xi\in\Gamma(\nu M)$ we have
$\langle\nabla_{P\eta}JY,P\xi\rangle=\langle\II(P\eta,Y),\xi\rangle
-\langle\II(P\eta,JY),F\xi\rangle={\sqrt{-c}}\langle
Y,Z\rangle\langle P\eta,P\xi\rangle/2$. This, the explicit expression for
the curvature tensor $\bar{R}$ of $\C H^n$, the Codazzi equation,~\eqref{eq:II} and $\bar\nabla J = 0$
imply
\begin{align*}
c\langle P\eta,P\eta\rangle\langle X,Y\rangle
&{}= 4\langle\bar{R}(X,P\eta)JY,\eta\rangle
    = 4\langle(\nabla_X^\perp\II)(P\eta,JY)
        -(\nabla_{P\eta}^\perp\II)(X,JY),\eta\rangle\\
&{}= -4\langle\II(P\eta,\nabla_XJY),\eta\rangle
    +4\langle\II(X,\nabla_{P\eta}JY),\eta\rangle\\
&{}= -4\langle\nabla_XJY,Z\rangle\langle\II(P\eta,Z),\eta\rangle
    +4\langle X,Z\rangle\langle\II(Z,\nabla_{P\eta}JY),\eta\rangle\\
&{}= 2\sqrt{-c}\langle P\eta,P\eta\rangle\langle\bar{\nabla}_XB,Y\rangle
    -c\langle P\eta,P\eta\rangle\langle X,Z\rangle\langle Z,Y\rangle.
\end{align*}
Thus, if $X\in\Gamma(\g{C}\ominus\R B)$ we have, taking into account
$\bar{\nabla}_XB\in\Gamma(\g{C})$, that $2\bar{\nabla}_XB=-\sqrt{-c}\,(X+\langle
X,Z\rangle Z)$.

Next we assume that $X\in\Gamma(P\nu M)$ and we put $X=P\xi$ with
$\xi\in\Gamma(\nu M)$. Then, we have $\langle \nabla_{JY}P\xi , Z \rangle
= - \langle \bar\nabla_{JY}Z , J\xi - F\xi \rangle = - \langle
\II(JY,B),\xi \rangle + \langle \II(JY,Z),F\xi \rangle=0$. This,
together with the expression for $\bar{R}$, the Codazzi
equation,~\eqref{eq:II} and $\bar\nabla J = 0$ yields
\begin{align*}
0 & {}= 2\langle\bar{R}(P\xi,JY)P\xi,\xi\rangle
=2\langle
(\nabla_{P\xi}^\perp\II)(JY,P\xi)
-(\nabla_{JY}^\perp\II)(P\xi,P\xi),\xi \rangle \\
& {}= -2\langle \II (\nabla_{P\xi} JY,P\xi),\xi \rangle
+ 4 \langle \II(\nabla_{JY}P\xi,P\xi),\xi\rangle \\
& {}= -2\langle \nabla_{P\xi} JY , Z \rangle \langle \II
(Z,P\xi),\xi \rangle + 4 \langle \nabla_{JY}P\xi , Z \rangle
\langle \II(Z,P\xi),\xi\rangle \\
& {}=  -\sqrt{-c}\langle P\xi,P\xi\rangle\langle \bar{\nabla}_{P\xi}JY , Z \rangle =
\sqrt{-c}\langle P\xi,P\xi\rangle\langle \bar{\nabla}_{P\xi}B , Y \rangle.
\end{align*}
Hence $\langle \bar{\nabla}_{P\xi}B,Y\rangle=0$, and using~\eqref{eq:DXBPxi} we get $2\bar{\nabla}_{P\xi}B=-\sqrt{-c}\,P\xi$.
Altogether we get~\eqref{eq:DXB}.

Now we prove~\eqref{eq:DBPxi}. Let
$\xi,\zeta\in\Gamma(\nu M)$ and $Y\in\Gamma(\g{C})$. As $\g{C}$
is autoparallel, we have $\langle\bar{\nabla}_BP\xi,Y\rangle=0$. Using~\eqref{eq:II} we get
$\langle\bar{\nabla}_BP\xi,\zeta\rangle=\langle\II(B,P\xi),\zeta\rangle=0$.
Moreover, using~\eqref{eq:II}, we obtain
$\Ss_\xi B=0$ and thus
\begin{align*}
\langle\bar{\nabla}_B P\xi,P\zeta\rangle
&{}=\langle\bar{\nabla}_B(J-F)\xi,P\zeta\rangle
=-\langle\bar{\nabla}_B\xi,JP\zeta\rangle+\langle F\xi,\bar{\nabla}_BP\zeta\rangle\\
&{}=\langle \Ss_\xi B,JP\zeta\rangle-\langle\nabla_B^\perp\xi,JP\zeta\rangle
=\langle P\nabla_B^\perp\xi,P\zeta\rangle.
\end{align*}
This implies~\eqref{eq:DBPxi}.

Finally, if $Y\in\Gamma(\g{C})$, using again~\eqref{eq:II} we have
\begin{align*}
2\langle \bar{\nabla}_{P\xi}P\xi ,Y\rangle &{}= -2\langle
\bar{\nabla}_{P\xi}Y, J\xi-F\xi \rangle =2\langle JY,Z \rangle\langle
\II(P\xi,Z),\xi \rangle
    + 2\langle Y,Z \rangle\langle \II(P\xi,Z),F\xi \rangle \\
&{}= -\sqrt{-c}\langle P\xi,P\xi\rangle\langle JZ,Y \rangle
    -\sqrt{-c}\langle Y,Z \rangle \langle JP\xi,F\xi \rangle
=\sqrt{-c}\langle P\xi,P\xi \rangle\langle B,Y \rangle,
\end{align*}
where we have used $\langle JP\xi,F\xi\rangle=\langle JP\xi,J\xi-P\xi\rangle
=\langle P\xi,\xi\rangle-\langle JP\xi,P\xi\rangle=0$.
Obviously,~\eqref{eq:II} implies
$\langle\bar{\nabla}_{P\xi}P\xi,\zeta\rangle
=\langle\II(P\xi,P\xi),\zeta\rangle=0$. Using~\eqref{eq:II} we obtain
\begin{align*}
\langle\bar{\nabla}_{P\xi}P\xi,P\zeta\rangle
&{}=\langle\bar{\nabla}_{P\xi}(J-F)\xi,P\zeta\rangle
=-\langle\bar{\nabla}_{P\xi}\xi,JP\zeta\rangle
+\langle\bar{\nabla}_{P\xi}P\zeta,F\xi\rangle\\
&{}=\langle\Ss_\xi P\xi,JP\zeta\rangle
-\langle\nabla_{P\xi}^\perp\xi,JP\zeta\rangle
=\langle P\nabla_{P\xi}^\perp\xi,P\zeta\rangle.
\end{align*}
Altogether this yields~\eqref{eq:DPxiPxi}.
\end{proof}

The next lemma basically says that the point at infinity determined by $B$ does not depend on the point $o\in M$ that was chosen.

\begin{lemma}\label{th:B}
The vector field $B$ is a geodesic vector field
and all its integral curves are pieces of geodesics in $\C H^{n}$ converging to
the point $x\in \C H^n(\infty)$.
\end{lemma}

\begin{proof}
Since $B\in\Gamma(\g{C})$ we have $\bar{\nabla}_B B\in\Gamma(\g{C})$. Clearly, $\langle\bar{\nabla}_B B,B\rangle=0$. Let $X \in
\Gamma(\g{C} \ominus {\mathbb R}B)$ and $\eta \in
\Gamma(\nu M)$ be a local unit normal vector field. Using
the expression for $\bar{R}$, the Codazzi equation,
\eqref{eq:II} and $\bar{\nabla} J = 0$ we obtain
\begin{align*}
0 &{}= 2\langle \bar{R}(B,P\eta) JX, \eta\rangle = 2\langle
(\nabla_B^\perp\II)(P\eta,JX)-(\nabla_{P\eta}^\perp\II)(B,JX),\eta
\rangle \\
& {}=  -2\langle \II (P\eta,\nabla_BJX),\eta \rangle
 =  -2\langle\nabla_BJX,Z\rangle\langle\II(P\eta,Z),\eta\rangle\\
& {}= \sqrt{-c}\langle JP\eta,\eta\rangle\langle \bar{\nabla}_BJX , Z \rangle =
\sqrt{-c}\langle P\eta,P\eta\rangle\langle \bar{\nabla}_BB , X \rangle.
\end{align*}
This yields $\langle\bar{\nabla}_BB,X\rangle=0$ and hence
$\bar{\nabla}_BB=0$. This implies that the integral curves of $B$ are
geodesics in $\C H^{n}$.

Now let $X\in\Gamma(TM\ominus\R B)$ be a unit vector field, and $\gamma$ an integral curve
of $X$. We define the geodesic variation
$F(s,t)=\exp_{\gamma(s)}(tB_{\gamma(s)})$, where $F_s(t)=F(s,t)$ are integral curves of $B$. We
prove that $d(F(s_1,t),F(s_2,t))$ tends to $0$ as $t$ goes to
infinity, where $d$ stands for the Riemannian distance function of $\C H^n$.

The transversal vector field of $F$,
$\zeta(s,t)=(\partial F/\partial s)(s,t)$, is a Jacobi field along
each $F_s$ satisfying
\begin{align*}
4\frac{\partial^2\zeta}{\partial t^2}+c\zeta+3c\langle \zeta,Z\rangle Z&{}=0,&
\zeta(s,0)&{}=X_{\gamma(s)},&
\frac{\partial \zeta}{\partial t}(s,0)&{}=\bar{\nabla}_{X_{\gamma(s)}} B.
\end{align*}
If $\mathcal{P}_X$ denotes $\bar\nabla$-parallel translation
of $X$ along $F_s$, one can directly show that
\[
\zeta(s,t)=e^{-t\sqrt{-c}/2}\mathcal{P}_X(s,t)
+(e^{-t\sqrt{-c}}-e^{-t\sqrt{-c}/2})\langle X_{F_s(0)},Z_{F_s(0)}\rangle Z_{F_s(t)},
\]
where we have used~\eqref{eq:DXB} and the fact that $Z$ is a parallel vector field along $F_s$
since $\bar{\nabla}_{B_{F(s,t)}}Z=J\bar{\nabla}_{B_{F(s,t)}}B=0$. It is easy to see
that $\lim_{t\to\infty}\lVert \zeta(s,t) \rVert = 0$. Using the mean value theorem of
integral calculus we get
\[
d(F(s_1,t),F(s_2,t))\leq
\int_{s_1}^{s_2}\lVert\frac{\partial F}{\partial s}(s,t)\rVert\,ds=
\int_{s_1}^{s_2}\lVert\zeta(s,t)\rVert\,ds=(s_2-s_1)\lVert\zeta(s_*,t)\rVert
\to 0,
\]
for some $s_*\in(s_1,s_2)$.
Therefore the integral curves of $B$ are geodesics
converging to the point $x\in \C H^n(\infty)$ at infinity.
\end{proof}

Now take $p\in\mathcal{L}_o$ and let $\xi_p\in\nu_pM$ be a unit vector.
As we argued before, the theorem will follow if we prove that the horocycle determined
by $P\xi_p/\lVert P\xi_p\rVert$ and the point $x\in \C H^n(\infty)$ is
locally contained in $M$. To this end we will construct a local unit
vector field $\xi\in\Gamma(\nu M)$ such that the aforementioned
horocycle is an integral curve of $P\xi/\lVert P\xi\rVert$.

Let $\gamma\colon I\to M$ be a curve satisfying the initial value problem
\begin{equation}\label{eq:gamma}
\nabla_{\gamma'}\gamma'
=\frac{\sqrt{-c}}{2}\langle\gamma',\gamma'\rangle B,\qquad
\gamma'(0)=P\xi_p/\lVert P\xi_p\rVert.
\end{equation}

\begin{lemma}
A curve $\gamma$ satisfying~\eqref{eq:gamma} is parametrized by arc length and remains
tangent to $P\nu M$.
\end{lemma}

\begin{proof}
Write $\gamma'=aB+xZ+X+P\eta$ for certain differentiable functions
$a,x\colon I\to\R$, and vector fields
$X\in\Gamma(\gamma^*(\g{C}\ominus(\R B\oplus\R Z)))$ and
$\eta\in\Gamma(\gamma^*\nu M)$. As $Z=JB$, the definition of
$\gamma$ and~\eqref{eq:DXB} show
\[
\frac{dx}{dt}
=\frac{d}{dt}\langle\gamma',Z\rangle
=\langle\nabla_{\gamma'}\gamma',Z\rangle
    +\langle\nabla_{\gamma'}Z,\gamma'\rangle
=\sqrt{-c}\,\langle xB-\frac{1}{2}JX-\frac{1}{2}JP\eta,\gamma'\rangle
=\sqrt{-c}\,ax.
\]
Since $x(0)=0$, the uniqueness of solutions to ordinary differential
equations gives $x(t)=0$ for all $t$.

Let $Y\in\Gamma(\R B\oplus P\nu M)$ and $\zeta\in\Gamma(\nu M)$. Then,
\eqref{eq:II} yields
$\langle\bar{\nabla}_YX,\zeta\rangle=\langle\II(Y,X),\zeta\rangle=0$ and
$\langle\bar{\nabla}_YX,J\zeta\rangle=-\langle\II(Y,JX),\zeta\rangle=0$.
Since $\bar{\nabla}_YB\in\Gamma(P\nu M)$ by~\eqref{eq:DXB}, we
have $\langle\bar{\nabla}_YX,B\rangle=-\langle\bar{\nabla}_YB,X\rangle=0$.
Moreover, $2\langle\bar{\nabla}_XX,B\rangle=-2\langle\bar{\nabla}_XB,X\rangle
=\sqrt{-c}\langle X,X\rangle$ and
$\langle\bar{\nabla}_X X,P\eta\rangle=-\langle\II(X,JX),\eta\rangle
-\langle\II(X,X),F\eta\rangle=0$. Hence, using also the definition of the curve $\gamma$,
\begin{align*}
\frac{d}{dt}\langle X,X\rangle
&{}= \frac{d}{dt}\langle\gamma',X\rangle
= \langle\nabla_{\gamma'}\gamma',X\rangle
    +\langle\nabla_{\gamma'}X,\gamma'\rangle\\
&{}= a\langle\bar{\nabla}_{\gamma'}X,B\rangle
    +\langle\bar{\nabla}_{\gamma'}X,X\rangle
    +\langle\bar{\nabla}_{\gamma'}X,P\eta\rangle
= \langle\bar{\nabla}_{\gamma'}X,X\rangle
    +a\langle\bar{\nabla}_{X}X,B\rangle\\
&{}= \langle\nabla_{\gamma'}X,X\rangle
    +\frac{a\sqrt{-c}}{2}\langle X,X\rangle
= \frac{1}{2}\frac{d}{dt}\langle X,X\rangle
    +\frac{a\sqrt{-c}}{2}\langle X,X\rangle.
\end{align*}
This gives $(d/dt)\langle X,X\rangle=a\sqrt{-c}\langle X,X\rangle$. Since
$\langle X(0),X(0)\rangle=0$ we get $\langle X(t),X(t)\rangle=0$
for all $t$, and thus $X=0$.

The definition of $\gamma$ gives
\[
\frac{d}{dt}\langle\gamma',\gamma'\rangle
=2\langle\nabla_{\gamma'}\gamma',\gamma'\rangle
=a\sqrt{-c}\langle\gamma',\gamma'\rangle.
\]

Using again the definition of $\gamma$, the fact that $B$ is geodesic and
\eqref{eq:DXB}, we get
\begin{align*}
\frac{da}{dt}&{}=\frac{d}{dt}\langle\gamma',B\rangle
=\langle\nabla_{\gamma'}\gamma',B\rangle
    +\langle\nabla_{\gamma'}B,\gamma'\rangle\\
&{}=\frac{\sqrt{-c}}{2}\bigl(\langle\gamma',\gamma'\rangle
    -\langle P\eta,\gamma'\rangle\bigr)
=\frac{\sqrt{-c}}{2}\bigl(\langle\gamma',\gamma'\rangle
    -\langle P\eta,P\eta\rangle\bigr).
\end{align*}

Finally, from~\eqref{eq:DBPxi} and~\eqref{eq:DPxiPxi} we obtain
\begin{align*}
\frac{d}{dt}\langle P\eta,P\eta\rangle
&{}= \frac{d}{dt}\langle\gamma',P\eta\rangle
    = \langle\nabla_{\gamma'}\gamma',P\eta\rangle
    +\langle\nabla_{\gamma'}P\eta,\gamma'\rangle\\
&{}= \frac{a\sqrt{-c}}{2}\langle P\eta,P\eta\rangle
    +a\langle P\nabla_B^\perp\eta,P\eta\rangle
    +\langle P\nabla_{P\eta}^\perp\eta,P\eta\rangle\\
&{}= \frac{a\sqrt{-c}}{2}\langle P\eta,P\eta\rangle
    +\langle \bar{\nabla}_{\gamma'}P\eta,P\eta\rangle
    = \frac{a\sqrt{-c}}{2}\langle P\eta,P\eta\rangle
    +\frac{1}{2}\frac{d}{dt}\langle P\eta,P\eta\rangle,
\end{align*} and thus
\[
\frac{d}{dt}\langle P\eta,P\eta\rangle=a\sqrt{-c}\langle P\eta,P\eta\rangle.
\]

If we define $b=\langle\gamma',\gamma'\rangle$ and $h=\langle
P\eta,P\eta\rangle$, we get the initial value problem:
\[
a'=\frac{\sqrt{-c}}{2}(b-h),\quad b'=\sqrt{-c}\,ab,\quad h'=\sqrt{-c}\,ah,\quad a(0)=0,
\quad b(0)=h(0)=1.
\]
Again, by uniqueness of solution we deduce
$a(t)=0$, $b(t)=h(t)=1$ for all $t$. Hence,
$\langle\gamma'(t),\gamma'(t)\rangle=1$ and $\gamma'(t)\in
P\nu M$ for all $t$ as we wanted to show.
\end{proof}

Let us assume then that $\gamma\colon I\to M$ is a curve satisfying
equation~\eqref{eq:gamma}. Since the map $P\colon\nu M\to \g{D}=P\nu M$ is an isomorphism of vector bundles, there exists a smooth unit normal vector field
$\eta$ of $M$ in a neighborhood of $p$ such that
$\gamma'(t)=P\eta_{\gamma(t)}/\lVert P\eta_{\gamma(t)}\rVert$ for all sufficiently
small $t$. Since $B$ is a unit vector field and $\gamma$ is orthogonal to $B$,
we can find a hypersurface $\mathcal{N}$ in $M$ containing $\gamma$ and
transversal to $B$ in a small neighborhood of $p$. The restriction of
$\eta$ to this hypersurface $\mathcal{N}$ is a smooth unit normal vector field along $\mathcal{N}$. We
define $\xi$ to be the unit normal vector field on a neighborhood of
$p$ such that $\xi=\eta$ on $\mathcal{N}$, and such that $\xi$ is obtained by
$\nabla^\perp$-parallel translation along the integral curves of
$B$. It follows that $\xi$ is smooth by the smooth dependence on initial conditions of ordinary
differential equations, and by definition
$\nabla_B^\perp\xi=0$.

The definition of $\xi$ and equations~\eqref{eq:DXB} and
\eqref{eq:DBPxi} imply
$2[B,P\xi]=2\bar{\nabla}_BP\xi-2\bar{\nabla}_{P\xi}B={\sqrt{-c}}\,P\xi$. Thus,
the distribution generated by $B$ and $P\xi$ is integrable. We denote by $\mathcal{U}$
the integral submanifold through $p$.

\begin{lemma}\label{th:U}
We have:
\begin{enumerate}[{\rm (i)}]
\item The norm of $P\xi$ is constant along the integral curves of $P\xi$, that is, $P\xi(\lVert P\xi\rVert)=0$.
\item $\bar{\nabla}_{P\xi}P\xi=\sqrt{-c}\langle P\xi,P\xi\rangle B$.
\item The submanifold $\mathcal{U}$ is an
open part of a totally geodesic $\R H^2$ in $\C H^{n}$.
\end{enumerate}
\end{lemma}

\begin{proof}
We calculate $\bar{\nabla}_{P\xi}P\xi$. Equation~\eqref{eq:II}
implies that $\Ss_\eta B=0$ for all $\eta\in\nu M$. Then, for any
$\eta,\zeta\in\nu M$ the Ricci equation of $M$ yields
\[
\langle R^\perp(B,P\xi)\eta,\zeta\rangle
=\langle\bar{R}(B,P\xi)\eta,\zeta\rangle
+\langle[\Ss_\eta,\Ss_\zeta]B,P\xi\rangle=0,
\]
where $R^\perp$ denotes the curvature tensor of the normal
connection $\nabla^\perp$. This, $2[B,P\xi]=\sqrt{-c}\,P\xi$,
and the definition of $\xi$ give
\[
0=R^\perp(B,P\xi)\xi
=\nabla_B^\perp\nabla_{P\xi}^\perp\xi
    -\nabla_{P\xi}^\perp\nabla_B^\perp\xi
    -\nabla_{[B,P\xi]}^\perp\xi
=\nabla_B^\perp\nabla_{P\xi}^\perp\xi
    -\frac{\sqrt{-c}}{2}\nabla_{P\xi}^\perp\xi,
\]
and therefore,
\begin{equation}\label{eq:DperpPxixi}
2\nabla_B^\perp\nabla_{P\xi}^\perp\xi=\sqrt{-c}\,\nabla_{P\xi}^\perp\xi.
\end{equation}

By definition of $\xi$, along $\gamma$ we have $\gamma'(t)=P\xi_{\gamma(t)}/\lVert P\xi_{\gamma(t)}\rVert$, and thus,
along $\gamma$ we get
\begin{align*}
\bar{\nabla}_{P\xi}P\xi
&{}=\bar{\nabla}_{\lVert P\xi\rVert\gamma'}\bigl(\lVert P\xi\rVert\gamma'\bigr)
=\lVert P\xi\rVert\Bigl(\gamma'(\lVert P\xi\rVert)\gamma'
+\lVert P\xi\rVert\bar{\nabla}_{\gamma'}\gamma'\Bigr)\\
&{}=\gamma'(\lVert P\xi\rVert)P\xi
+\frac{\sqrt{-c}}{2}\langle P\xi,P\xi\rangle B.
\end{align*}
Comparing this equation with~\eqref{eq:DPxiPxi} yields
$P\nabla^\perp_{P\xi}\xi=\gamma'(\lVert P\xi\rVert)P\xi$, and since $P\colon\nu M\to \g{D}=P\nu M$ is an isomorphism of vector bundles we get $\nabla_{P\xi}^\perp \xi=\gamma'(\lVert P\xi\rVert)\xi$. Finally, $\langle\xi,\xi\rangle=1$ implies $\langle\nabla^\perp_{P\xi}\xi,\xi\rangle=0$. Thus, $\gamma'(\lVert P\xi\rVert)=0$, which is our first assertion, and hence $\nabla_{P\xi}^\perp\xi=0$
along $\gamma$.

Now, let $\alpha$ be an integral curve of $B$
such that $\alpha(0)=\gamma(s)$. We have just shown that
$\nabla_{P\xi}^\perp\xi\,_{\vert_{\alpha(0)}}
=\nabla_{P\xi}^\perp\xi \,_{\vert_{\gamma(s)}}=0$. Next, from~\eqref{eq:DperpPxixi} and since $\Ss_\eta B=0$ for each
$\eta\in\nu M$, we get
\[
2\bar{\nabla}_{\alpha'}\nabla_{P\xi}^\perp\xi\,_{\vert_{t}}
=2\nabla_B^\perp\nabla_{P\xi}^\perp\xi\,_{\vert_{\alpha(t)}}
    -2\Ss_{\nabla_{P\xi}^\perp\xi}B\,_{\vert_{\alpha(t)}}
=\sqrt{-c}\,\nabla_{P\xi}^\perp\xi\,_{\vert_{\alpha(t)}}.
\]
Hence, by the uniqueness of solutions to differential equations
we get $\nabla_{P\xi}^\perp\xi\,_{\vert_{\alpha(t)}}=0$ for all~$t$,
and consequently $2\bar{\nabla}_{P\xi}P\xi=\sqrt{-c}\,\langle P\xi,P\xi\rangle
B$ along the integral submanifold $\mathcal{U}$. This is our second assertion.

Since $B$ is a geodesic vector field we have $\bar{\nabla}_BB=0$. By~\eqref{eq:DXB}
we have $2\bar{\nabla}_{P\xi}B=-\sqrt{-c}\,P\xi$, and by definition of $\xi$ and~\eqref{eq:DBPxi} we get $\bar{\nabla}_BP\xi=P\bar{\nabla}_B^\perp\xi=0$. Together with~(ii) we deduce that $\mathcal{U}$ is an open part of
a totally geodesic $\R H^2\subset\C H^{n}$.
\end{proof}

We define $\bar P\xi = P\xi/\lVert P\xi \rVert$
along $\mathcal{U}$. From Lemma~\ref{th:U} we obtain $2\bar{\nabla}_{\bar
P\xi}\bar P\xi = \sqrt{-c}\,B$. Using this and~\eqref{eq:DXB} we obtain
\[
\bar{\nabla}_{\bar P\xi}\bar{\nabla}_{\bar P\xi}\bar P\xi+ \langle
\bar{\nabla}_{\bar P\xi}\bar P\xi, \bar{\nabla}_{\bar P\xi}\bar P\xi\rangle
\bar P\xi= \frac{\sqrt{-c}\,}{2}\bar{\nabla}_{\bar P\xi}B - \frac{c}{4}\langle
B,B\rangle\bar P\xi= 0.
\]
Therefore, the integral curves of $\bar{P}\xi$ are horocycles contained in $\mathcal{U}$ with center $x\in\C H^n(\infty)$, where $\mathcal{U}$ is an open part of a totally geodesic real hyperbolic plane in $\C H^n$. The rigidity of totally geodesic
submanifolds of Riemannian manifolds (see e.g.~\cite[p.~230]{BCO03}), and of horocycles in real hyperbolic planes (see e.g.~\cite[pp.~24--26]{BCO03}), together with the construction method described in Proposition~\ref{th:construction}, imply that a neighborhood of any $o$ in $M$ is congruent to an open part of a submanifold $W_\g{w}$ determined by the point  $o\in\C H^n$, $x\in\C H^n(\infty)$ and $\g{w}=T_o M\ominus(\R B_o\oplus\R Z_o)$.

The argument above was local, so we still need to prove that the connected submanifold $M$ is contained in the $W_{\g{w}}$ stated above. Since $W_{\g{w}}$ is an orbit of a Lie group  action on an analytic manifold, it follows that $W_{\g{w}}$ is analytic and complete. Since $M$ is a smooth minimal submanifold in an analytic Riemannian manifold, it is well known that $M$ is also an analytic submanifold of $\C H^n$. As an open neighborhood of $M$ is contained in $W_{\g{w}}$ it follows that $M$ is an open part of the submanifold $W_{\g{w}}$.

\section{Proofs of the Main Theorem and Theorem~\ref{th:congruence}}\label{sec:proof}

We are now ready to summarize our arguments and conclude the proofs of the main theorems of this paper.

\begin{proof}[Proof of the Main Theorem]
Assume that $M$ is a connected isoparametric hypersurface in the complex hyperbolic space $\C H^n$, $n\geq 2$. Then, its lift to the anti-De Sitter space $\tilde{M}=\pi^{-1}(M)$ is also  an isoparametric hypersurface. If at some point the shape operator of $\tilde{M}$ is of type~II or of type~IV, then by Proposition~\ref{th:different_types} we have that $M$ is an open part of a horosphere or a tube around a totally geodesic real hyperbolic space $\R H^n$ in $\C H^n$, respectively. This corresponds to cases~(\ref{th:main:horosphere}) and~(\ref{th:main:rhn}) of the Main Theorem. If all points of $\tilde{M}$ are of type~I, then Remark~\ref{rmk:tubes_CHk} implies that $M$ is an open part of a tube around a totally geodesic $\C H^k$ in $\C H^n$ (Main Theorem~(\ref{th:main:chk})).

Finally, if there is a point $q\in\tilde{M}$ of type~III, then there is a neighborhood $\tilde{\mathcal{W}}$ of $q$ where all points are of type~III by Proposition~\ref{th:different_types}. Then, by the results of Section~\ref{sec:type_III}, there is $r\geq 0$ such that the parallel displacement at distance $r$, that is, $\mathcal{W}^r=\Phi^r(\pi(\tilde{\mathcal{W}}))$, is a submanifold of $\C H^n$ such that its second fundamental form is given by the trivial symmetric bilinear extension of $2\II(Z,P\xi)=-\sqrt{-c}\,(JP\xi)^\perp$, $\xi\in\nu \mathcal{W}^r$, where $Z$ is a vector field tangent to the maximal complex distribution of $\mathcal{W}^r$, and $(\cdot)^\perp$ denotes orthogonal projection on $\nu\mathcal{W}^r$. Using Theorem~\ref{th:rigidity} we conclude that there exists an Iwasawa decomposition $\g{k}\oplus\g{a}\oplus\g{g}_\alpha\oplus\g{g}_{2\alpha}$ of the Lie algebra of the isometry group of $\C H^n$ and a subspace $\g{w}$ of $\g{g}_\alpha$, such that $\mathcal{W}^r$ is an open part of $W_{\g{w}}$.

Therefore, we have proved that there is an open subset of $M$ that is an open part of a tube of radius $r$ around the submanifold $W_{\g{w}}$. Since both $M$ and the tubes around $W_\g{w}$ are smooth hypersurfaces with constant mean curvature, they are real analytic hypersurfaces of $\C H^n$. Thus, we conclude that $M$ is an open part of a tube of radius $r$ around $W_{\g{w}}$. Note that $W_{\g{w}}$ is minimal, as shown in~\cite{DD12}, and ruled by totally geodesic complex hyperbolic subspaces, as follows from Lemma~\ref{th:complex}.

If $\g{w}$ is a hyperplane, $W_{\g{w}}$ is denoted by $W^{2n-1}$, and we get one of the examples in Main Theorem~(\ref{th:main:w}). In this case we can have $r=0$ and we get exactly $W^{2n-1}$. Both $W^{2n-1}$ and its equidistant hypersurfaces are homogeneous  (see for example~\cite{Be98}).

If $\g{w}^\perp$ has constant K\"{a}hler angle $\varphi\in(0,\pi/2]$, then $W_{\g{w}}$ is denoted by $W_\varphi^{2n-k}$, where $k$ is the codimension. If $\varphi\neq \pi/2$, then $k$ is even~\cite{BB01}. In any case the tubes around $W_\varphi^{2n-k}$ are homogeneous as was shown in~\cite{BB01}. These correspond to case~(\ref{th:main:wphi}) of the Main Theorem.

If $\g{w}^\perp$ does not have constant K\"{a}hler angle, then the tubes around $W_{\g{w}}$ are isoparametric but not homogeneous~\cite{DD12}. These remaining examples correspond to case~(\ref{th:main:ww}) of the Main Theorem.
\end{proof}

\begin{proof}[Proof of Theorem~\ref{th:congruence}]
An isoparametric family corresponding to cases~(\ref{th:main:horosphere}) or~(\ref{th:main:w}) in the Main Theorem cannot be congruent to a family in one of the other four cases, since the former are regular Riemannian foliations, whereas the latter families always have a singular leaf. Foliations in cases~(\ref{th:main:horosphere}) and~(\ref{th:main:w}) give rise to exactly two congruence classes. Indeed, the family in~(\ref{th:main:w}) has a minimal leaf $W^{2n-1}$ whereas the family in~(\ref{th:main:w}) does not (see Remark~\ref{rmk:Hopf}). Furthermore, all horosphere foliations are mutually congruent, as well as all solvable foliations. Now, any  family in~(\ref{th:main:chk}) and~(\ref{th:main:rhn}) has a totally geodesic singular leaf, whereas the singular leaf $W_\g{w}$ in~(\ref{th:main:wphi}) and~(\ref{th:main:ww}) is not totally geodesic. Moreover, the classification of totally geodesic submanifolds of $\C H^n$ allows to distinguish between cases~(\ref{th:main:chk}) and~(\ref{th:main:rhn}).

In order to finish the proof it is convenient to consider the families~(\ref{th:main:chk}),~(\ref{th:main:w}),~(\ref{th:main:wphi}) and~(\ref{th:main:ww}) as tubes around a submanifold $W_{\g{w}}$ as described in Subsection~\ref{sec:examples:W}. Thus, a totally geodesic $\C H^k$, $k\in\{1,\dots,n-1\}$, corresponds to a submanifold $W_{\g{w}}$, where $\g{w}\subset\g{g}_\alpha$ is complex, a Lohnherr submanifold $W^{2n-1}$ corresponds to a hyperplane $\g{w}$ in $\g{g}_\alpha$, and a Berndt-Br\"{u}ck submanifold $W^{2n-k}_\varphi$ corresponds to a subspace $\g{w}^\perp$ of $\g{g}_\alpha$ of constant K\"{a}hler angle. Thus, the congruence classes of isoparametric families of hypersurfaces in $\C H^n$ are parametrized by the disjoint union of the singular foliation by geodesic spheres $\mathcal{F}_o$, the horosphere foliation $\mathcal{F}_H$, the singular foliation $\mathcal{F}_{\R H^n}$ of tubes around a totally geodesic $\R H^n$, and the congruence classes of isoparametric families of tubes around the submanifolds~$W_{\g{w}}$, which we still have to determine.

The submanifold $W_{\g{w}}$ depends on the choice of a root space decomposition. Since any two such decompositions are conjugate by an element of $SU(1,n)$, it suffices to take a fixed root space decomposition $\g{g}=\g{g}_{-2\alpha}\oplus\g{g}_{-\alpha}\oplus\g{k}_0\oplus\g{a}
\oplus\g{g}_\alpha\oplus\g{g}_{2\alpha}$, real subspaces $\g{w}_1$, $\g{w}_2\subset\g{g}_\alpha$ and determine when the family of tubes around $W_{\g{w}_1}$ and $W_{\g{w}_2}$ are congruent. By dimension reasons, and by the minimality of $W^{2n-1}$ if both $\g{w}_1$, $\g{w}_2$ are hyperplanes, such families are congruent if and only if the two submanifolds $W_{\g{w}_1}=S_1\cdot o$ and $W_{\g{w}_2}=S_2\cdot o$ are congruent, where  $S_i$ is the connected Lie subgroup of $SU(1,n)$ with Lie algebra $\g{s}_i=\g{a}\oplus\g{w}_i\oplus\g{g}_{2\alpha}$, $i=1,2$.

Let $\phi$ be an isometry of $\C H^n$ such
that $\phi(W_{\g{w}_1})=W_{\g{w}_2}$, and assume, without loss of generality, that $\phi(o)=o$. The identification $T_o\C H^n\cong\g{a}\oplus\g{n}$ thus allows us to deduce that $\phi_*(\g{a}\oplus\g{w}_1\oplus\g{g}_{2\alpha})=\g{a}\oplus\g{w}_2\oplus\g{g}_{2\alpha}$. We consider the K\"{a}hler angle decompositions
$\g{w}_i=\oplus_{\varphi\in\Phi_i}\g{w}_{i,\varphi}$ as described in Subsection~\ref{sec:Kahler_angles}. Since $\phi$ is an isometry of $\C H^n$ fixing $o$, it follows that $\phi_*$ is a unitary or anti-unitary transformation of $T_o\C H^n\cong\g{a}\oplus\g{n}\cong \C^n$. Hence, it maps subspaces of constant K\"{a}hler angle to subspaces of the same constant K\"{a}hler angle, and thus we have $\Phi:=\Phi_1=\Phi_2$ and
\[
\phi_*(\g{a}\oplus\g{w}_{1,0}\oplus\g{g}_{2\alpha})
=(\g{a}\oplus\g{w}_{2,0}\oplus\g{g}_{2\alpha}),\
\phi_*(\g{w}_{1,\varphi})
=\g{w}_{2,\varphi},
\text{ for all $\varphi\in\Phi\setminus\{0\}$.}
\]
Therefore, $\g{w}_1$ and $\g{w}_2$ have the same K\"{a}hler angles with the same multiplicities. Now set $\g{k}_0=\g{g}_0\cap\g{k}$, where $\g{k}$ is the Lie algebra of $K$, the isotropy group at $o$. It is known (see e.g.~\cite{DDK}) that $\g{k}_0$ is a Lie subalgebra of $\g{g}$ and that the connected subgroup $K_0$ of $G=SU(1,n)$ whose Lie algebra is $\g{k}_0$ acts on $\g{g}_\alpha$, and its action is equivalent to the standard action of $U(n-1)$ on $\C^{n-1}$. The action of $K_0$ on $\g{a}$ and on $\g{g}_{2\alpha}$ is trivial. Since $\g{w}_1$ and $\g{w}_2$ are subspaces of $\g{g}_\alpha$ with the same K\"{a}hler angles and the same multiplicities, it follows that there exists $k\in K_0$ such that $\Ad(k)\g{w}_1=\g{w}_2$ (see the end of Subsection~\ref{sec:Kahler_angles} or~\cite[Remark~2.10]{DDK} for further details), and thus, $k(W_{\g{w}_1})=W_{\g{w}_2}$.

As a consequence, we have proved that the congruence classes of the submanifolds of type $W_{\g{w}}$ are in one-to-one correspondence with proper real subspaces of $\g{g}_\alpha\cong\C^{n-1}$ modulo the action of $K_0=U(n-1)$. Altogether this implies Theorem~\ref{th:congruence}.
\end{proof}


\end{document}